\documentclass[11pt,leqno]{article}

\usepackage{amsthm,amsfonts,amssymb,amsmath,oldgerm}
\usepackage{fullpage}
\usepackage{graphicx}
\usepackage{mathrsfs}
\usepackage{xcolor}
\usepackage{mathtools}
\usepackage[colorinlistoftodos]{todonotes} 
\usepackage[hidelinks]{hyperref}
\usepackage{pgfplots}
\pgfplotsset{compat=1.18}
\usepackage{enumitem}
\usepackage{comment}
\usepackage{bm}

\usepackage{tocloft}
\usepackage[backend=biber,
style=alphabetic,
sorting=nyt,
giveninits=true,
isbn=false,
url=false,
maxbibnames=99]{biblatex}

\renewbibmacro*{url+urldate}{}

\newtoggle{bbx:boldentries}
\DeclareBibliographyCategory{boldentry}
\AtEveryBibitem{\ifboolexpr{togl {bbx:boldentries} and test {\ifcategory{boldentry}}}{\bfseries}{}}

\AtEveryBibitem{\clearfield{month}}
\AtEveryBibitem{\clearfield{day}}
\AtEveryBibitem{\clearfield{note}}

\setlist[enumerate,2]{label=(\alph*), ref=\theenumi(\alph*)}

\renewbibmacro{in:}{%
  \ifentrytype{article}{}{%
    \printtext{\bibstring{in}\intitlepunct}}}

\usepackage{hyperref}
\usepackage{xurl}

\usepackage{tikz}
\usepackage{pgfplots}
\pgfplotsset{compat=1.11}
\usepgfplotslibrary{fillbetween}
\usetikzlibrary{spy,calc,arrows.meta,calc,decorations.markings,math,arrows.meta,shapes.misc}
\tikzset{cross/.style={cross out, draw=black, fill=none, minimum size=2*(#1-\pgflinewidth), inner sep=0pt, outer sep=0pt}, cross/.default={2pt}}

\numberwithin{equation}{section}

\renewcommand{\theenumi}{(\roman{enumi})}

\newcommand{\R}{{\mathbb R}}\newcommand{\N}{{\mathbb N}}

\newcommand{\Z}{{\mathbb Z}}

\newcommand{\de}{\mathrm{d}}

\newcommand{\curlB}{\mathcal{B}}
\newcommand{\curlL}{\mathcal{L}}

\newcommand{\curlR}{\mathcal{R}}
\newcommand{\curlA}{\mathcal{A}}

\newcommand{\curlU}{\mathcal{U}}

\newcommand{\curlF}{\mathcal{F}}
\newcommand{\curlS}{\mathcal{S}}
\newcommand{\curlG}{\mathcal{G}}
\newcommand{\curlC}{\mathcal{C}}
\newcommand{\curlK}{\mathcal{K}}
\newcommand{\curlX}{\mathcal{X}}
\newcommand{\curlY}{\mathcal{Y}}
\newcommand{\curlM}{\mathcal{M}}
\newcommand{\curlT}{\mathcal{T}}
\newcommand{\snorm}[2][]{\ensuremath{\left\vert#2\right\vert_{#1}}}
\newcommand{\norm}[2][]{\ensuremath{\left\|#2\right\|_{#1}}}

\newcommand{\frakF}{\mathfrak{F}}
\newcommand{\frakG}{\mathfrak{G}}

\newcommand{\Holdernormper}[3]{\left\| #1 \right\|_{C^{#2,#3}_\mathrm{per}}}
\newcommand{\Holderspace}[2][\gamma]{C^{#2,#1}_\mathrm{per}}

\newcommand{\pernorm}[2]{\left\| #1 \right\|_{C^{#2}_\mathrm{per}}}
\newcommand{\perspace}[1]{C^{#1}_\mathrm{per}}

\newcommand{\curl}{\nabla \times}

\DeclareMathOperator{\spn}{span}
\DeclareMathOperator{\ran}{ran}

\let\epsilon\varepsilon

\newtheorem{theorem}{Theorem}[section]\newtheorem{lemma}[theorem]{Lemma}

\newtheorem{proposition}[theorem]{Proposition}
\newtheorem{corollary}[theorem]{Corollary}

\newtheorem{remark}[theorem]{Remark}
\allowdisplaybreaks[3]

\title{Global bifurcation of doubly periodic gravity-capillary waves on Beltrami flows}

\author{\hyperlink{affBH}{Bastian Hilder}, \hyperlink{affGT}{Giang To}, \hyperlink{affEW}{Erik Wahl\'en}}

\begin{document}

\maketitle

\begin{abstract}
    We prove the existence of a global family of steady, doubly periodic gravity-capillary waves on Beltrami flows. This is the first rigorous existence result for genuinely three-dimensional inviscid surface waves, with or without vorticity, beyond the perturbative regime close to simple explicit solutions. The proof is based on reformulating the steady water wave problem as a bifurcation problem of the form `identity plus compact' and applying a global bifurcation argument in Hölder spaces. The main challenge is that the kernel of the linearisation at the bifurcation point is two-dimensional, and that both kernel elements are necessary to obtain genuinely three-dimensional solutions. Since this prevents the use of classical global bifurcation theory, we introduce a novel reformulation of the bifurcation problem using the parameterisation of a local family of solutions bifurcating from laminar flow. In this new parameter space, we apply a variation of analytic global bifurcation theory. Along the branch, we then prove a sharper blow-up alternative, namely blow-up of the surface gradient in $C^{0,\gamma}$.

    \vspace{1em}



\end{abstract}

\setcounter{tocdepth}{1}
\tableofcontents

\section{Introduction}

The mathematical understanding of steady nonlinear surface waves in inviscid fluids, known as the steady water wave problem, has seen remarkable progress over the last century. Yet, while there is by now a well-developed theory for two-dimensional water waves, even in the presence of vorticity, many questions remain open in the three-dimensional case. This discrepancy can be explained by the breakdown of mathematical structures that are essential to the two-dimensional theory. Most notably, there is no clear analogue of the stream function in the three-dimensional case. Despite these structural obstructions, a number of three-dimensional water waves have been constructed recently. However, all results are restricted to a perturbative regime close to a `simple' solution to the water wave problem, and there is currently a complete lack of existence theory beyond onset. To address this gap, we prove the existence of a global family of genuinely three-dimensional, doubly periodic steady water waves with vorticity, bifurcating from a trivial laminar flow.

We now describe the problem in detail. We consider three-dimensional steady water waves driven by gravity and surface tension. More precisely, we study an incompressible, inviscid fluid, which occupies the a priori unknown domain
\begin{equation*}
    \Omega^\eta = \{\bm x=(\bm{x}',z) \in \R^2 \times \R : -d < z < \eta(\bm{x}^\prime)\},
\end{equation*}
where $\bm{x}'=(x,y)$, $\eta \colon \R^2 \rightarrow \R$ is the free surface profile and $d > 0$ is the constant depth of the fluid; see Figure \ref{fig:water-waves}. The fluid motion in the interior of the fluid domain $\Omega^\eta$ is described by the velocity field $\bm{u} : \overline{\Omega^\eta} \rightarrow \R^3$ and the pressure $p \colon \overline{\Omega^\eta} \rightarrow \R$. After potentially shifting to an appropriate co-moving frame, we assume that $\bm{u}$ and $p$ are solutions to the stationary Euler equations with gravity and surface tension given by
\begin{alignat*}{2}
        (\bm{u} \cdot \nabla) \bm{u} &= - \nabla p - g \bm{e}_3 \qquad &&\text{ in } \Omega^\eta, \\
        \nabla \cdot \bm{u} &= 0 &&\text{ in } \Omega^\eta,
\end{alignat*}
where $g > 0$ is the gravitational constant and $\bm{e}_3 = (0,0,1)$ is the unit vector in vertical direction. To close the system, these interior equations are supplemented by kinematic boundary conditions at the flat bed and the free surface, namely,
\begin{equation*}
    \bm{u} \cdot \bm{n} = 0 \qquad \text{ on } \partial\Omega^\eta,
\end{equation*}
where $\bm{n}$ is the outward unit normal to $\partial\Omega^\eta$, and the dynamic boundary condition at the free surface, which is given by
\begin{equation*}
    p = p_\mathrm{atm} - 2\sigma K_M \qquad \text{ on } z = \eta(\bm{x}^\prime).
\end{equation*}
Here, $p_\mathrm{atm}$ is the constant atmospheric pressure, $\sigma > 0$ is the coefficient of surface tension, and $K_M$ is the mean curvature given by
\begin{align*}
    2K_M = \nabla \cdot \left(\dfrac{\nabla \eta}{\sqrt{1+\snorm{\nabla\eta}^2}}\right).
\end{align*}

\begin{figure}
    \centering
    \includegraphics[width=0.7\textwidth]{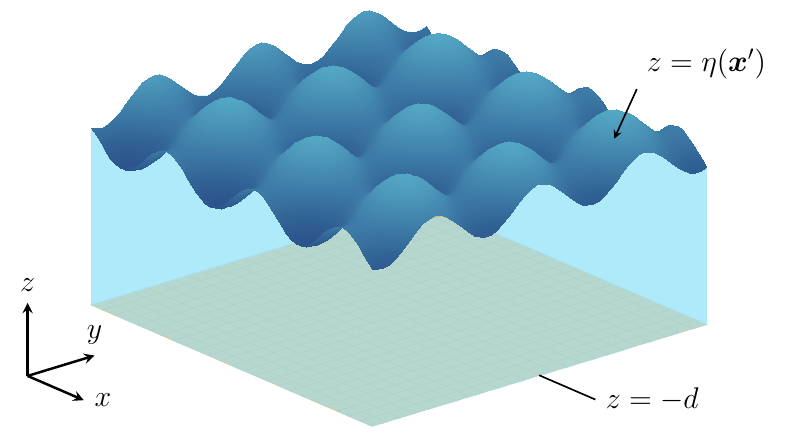}
    \caption{Schematic depiction of a three-dimensional water wave problem with finite depth and a doubly periodic surface profile.}
    \label{fig:water-waves}
\end{figure}

An important observable in fluid flows is the vorticity $\bm{\omega} = \nabla \times \bm{u}$. Historically, the vast majority of research on steady water waves has assumed that the flow is irrotational, that is, $\bm{\omega} = 0$. However, in recent years, the study of rotational water waves has gained increasing attention, starting with \cite{constantin2004}. A special class of flows with non-trivial vorticity are Beltrami flows, where the vorticity is collinear to the velocity field, that is,
\begin{equation*}
    \nabla \times \bm{u} = \alpha \bm{u} \qquad \text{ in } \Omega^\eta.
\end{equation*}
Specifically, we consider the case of a strong Beltrami field, where $\alpha$ is a constant. Beltrami fields, also called force-free fields, are well-known in solar and plasma physics; see, for example, \cite{priest1982,boulmezaoud1999,freidberg2014,priest2014}. In fluid mechanics, they are of special interest since any strong, divergence-free Beltrami field is a solution of the steady Euler equations, where the pressure is given by
\begin{equation*}
    p = C - \dfrac{|\bm{u}|^2}{2} - g z.
\end{equation*} 
Hence, the water wave problem for Beltrami flows can be reformulated as a div-curl problem for the velocity field $\bm{u}$ together with appropriate boundary conditions for the free surface profile $\eta$, which can be written as
\begin{subequations}
    \label{eq:formulation-without-integral-constraint}
    \begin{alignat}{2}
            \nabla \times \bm{u} &= \alpha \bm{u}  &&\text{in } \Omega^\eta,\\
            \nabla \cdot \bm{u} &= 0 &&\text{in } \Omega^\eta,\\
            \bm{u} \cdot \bm{n} &= 0 &&\text{on } \partial\Omega^\eta,\\
            \dfrac{1}{2} \snorm{\bm{u}}^2 + g\eta - 2\sigma K_M &= Q \qquad &&\text{on } z = \eta,
    \end{alignat}
\end{subequations}
where $Q \in \R$ is the Bernoulli constant.

The elliptic free-boundary problem \eqref{eq:formulation-without-integral-constraint} has a family of explicit solutions with flat surface profile $\eta \equiv 0$, called \emph{laminar flows}, given by
\begin{equation}\label{eq:laminar-flow}
    \bm{U}(\bm{c}) = c_1 \bm{U}^{(1)} + c_2 \bm{U}^{(2)}, \quad c_1,c_2 \in \R,
\end{equation}
where
\begin{equation*}
    \bm{U}^{(1)} = (\cos(\alpha z), -\sin(\alpha z),0), \quad \bm{U}^{(2)} = (\sin(\alpha z), \cos(\alpha z),0);
\end{equation*}
see Figure \ref{fig:laminar-flow}.
In this case, the corresponding Bernoulli constant is given by $Q(\bm c) = \frac{1}{2}(c_1^2+c_2^2)$.

\begin{remark}\label{rem:non-uniqueness}
    We point out that, for a given surface profile $\eta$, a corresponding velocity field satisfying \eqref{eq:formulation-without-integral-constraint} is in general not unique. Indeed, if $\bm{u}$ is a solution, then so is $\bm{u}+\bm{v}$ for any $\bm{v}$ solving the homogeneous div-curl problem $\nabla \times \bm{v} = \alpha \bm{v}$, $\nabla \cdot \bm{v} = 0$ in $\Omega^\eta$, with $\bm{v} \cdot \bm{n} = 0$ on $\partial\Omega^\eta$, with the Bernoulli constant $Q$ adjusted accordingly in the dynamic boundary condition. This can be seen even in the case of a flat surface $\eta \equiv 0$ by noting that the laminar flows \eqref{eq:laminar-flow} occur in a two-parameter family for fixed $\alpha$. To restore uniqueness, we will follow \cite{lokharu2020} and impose an additional integral constraint; see \eqref{eq:integralCond}.
\end{remark}

\begin{remark}
    Although it would be desirable to treat more general flows, the assumption that the velocity field is a Beltrami field is mathematically convenient, since it yields an elliptic free-boundary problem. In this context, we also recall Arnold's structure theorem (see, for example, \cite[Theorem II.1.2]{arnold2021}), which establishes a dichotomy for steady solutions to the Euler equations: any non-Beltrami velocity field has a relatively simple topological structure and, in particular, has a first integral. In contrast, Beltrami fields allow for a much more complicated topology; see, for example, \cite[Appendix A.II.1]{arnold2021}.
\end{remark}

\begin{figure}[h!]
    \centering
    \includegraphics[width=0.65\textwidth]{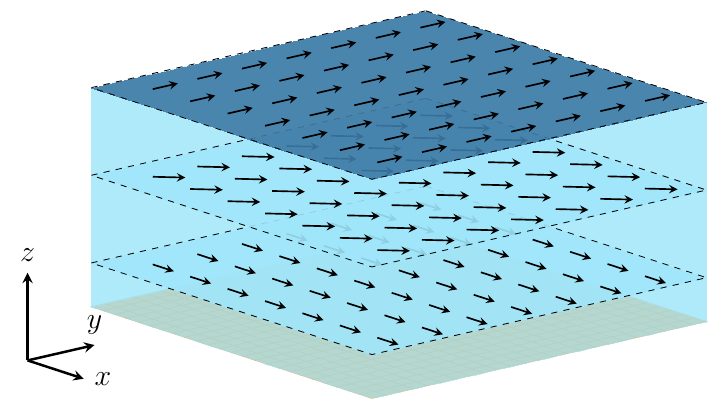}
    \caption{Schematic depiction of the laminar flow \eqref{eq:laminar-flow} with flat surface profile $\eta \equiv 0$.}
    \label{fig:laminar-flow}
\end{figure}

\subsection{Problem formulation and main results}

In this paper, we are interested in the existence of three-dimensional, doubly periodic water waves where the velocity field is a strong Beltrami field and which bifurcate from the laminar flow \eqref{eq:laminar-flow}. Therefore, we consider solutions to \eqref{eq:formulation-without-integral-constraint}, which are spatially periodic with respect to the lattice
\begin{align}\label{eq:lattice}
    \Lambda \coloneqq\{\bm{\lambda} = m_1 \bm{\lambda}_1 + m_2 \bm{\lambda}_2 :  m_1,m_2 \in \Z\}
\end{align}
for two linearly independent vectors $\bm{\lambda}_1,\bm{\lambda}_2$; see Figure \ref{fig:periodic-lattice}. This means that
\begin{align*}
    \eta(\bm{x}^\prime + \bm{\lambda}) = \eta(\bm{x}^\prime), \quad \bm{u}(\bm{x}^\prime + \bm{\lambda}, z) = \bm{u}(\bm{x}^\prime,z)
\end{align*}
for all $\bm{\lambda} \in \Lambda$.
Hence, we split the domain $\Omega^\eta$ into periodic cells
\begin{align*}
    \Omega^\eta_{lj} &\coloneqq \{(\bm{x}^\prime,z) : -d < z < \eta(\bm{x}^\prime),\, \bm{x}^\prime \in B_{lj}\}, \\
    B_{lj} &\coloneqq \{(a_1 + l)\bm{\lambda}_1 + (a_2 + j)\bm{\lambda}_2 : a_1,a_2 \in (0,1)\}
\end{align*}
for  $l,j \in \Z$. This set-up allows us to define periodic H\"older spaces $C^{k,\gamma}_\mathrm{per}(\mathcal{O};\R^m)$ for $k \geq 0$, $\gamma \in (0,1)$, $m \geq 1$, and $\mathcal{O} \in \{\R^2, \overline{\Omega^\eta}\}$, which denote the usual spaces of H\"older continuous functions on $\mathcal{O}$ with values in $\R^m$, which are spatially periodic with respect to the lattice $\Lambda$. Additionally, the space $C^{k,\gamma}_\mathrm{per}(\mathcal{O};\R^m)$ is equipped with the usual H\"older norm 
\begin{equation*}
    \|\bm{u}\|_{C_\mathrm{per}^{k,\gamma}(\mathcal{O};\R^m)} = \sum_{j \leq k} \sup_{\bm x\in \mathcal{O}}|D^{j} \bm{u}(\bm x)| + \sup_{\substack{\bm{x}, \bm{y} \in \mathcal{O} \\ \bm{x} \neq \bm{y}}} \dfrac{|D^k \bm{u}(\bm{x}) - D^k \bm{u}(\bm{y})|}{|\bm{x} - \bm{y}|^\gamma}.
\end{equation*}

\begin{figure}
    \centering
    \includegraphics[width=0.5\textwidth]{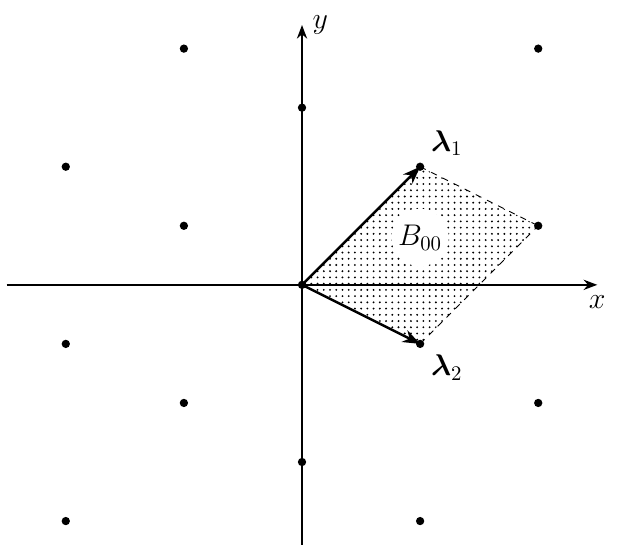}
    \caption{Periodic lattice spanned by $\bm{\lambda}_1$ and $\bm{\lambda}_2$ defining the doubly periodic waves and the bottom cell $B_{00}$.}
    \label{fig:periodic-lattice}
\end{figure}

\begin{remark}
    When it is clear from the context, we will omit the domains in the notation of the periodic H\"older spaces and for example write $C^{k,\gamma}_\mathrm{per}$ instead of $C^{k,\gamma}_\mathrm{per}(\R^2;\R)$ or $C^{k,\gamma}_\mathrm{per}(\overline{\Omega^\eta};\R^3)$, which are the usual spaces for the surface profile and the velocity field, respectively.
\end{remark}

As discussed in Remark \ref{rem:non-uniqueness}, we follow \cite{lokharu2020} and enforce uniqueness by adding an integral constraint to \eqref{eq:formulation-without-integral-constraint}, so that we consider
\begin{subequations}
    \label{eq:formulation}
    \begin{alignat}{2}
            \nabla \times \bm{u} &= \alpha \bm{u}  &&\text{in } \Omega^\eta, \label{eq:curl}\\
            \nabla \cdot \bm{u} &= 0 &&\text{in } \Omega^\eta, \label{eq:div} \\
            \bm{u} \cdot \bm{n} &= 0 &&\text{on } \partial\Omega^\eta, \label{eq:kinboundary}\\
            \int_{\Omega_{00}^\eta} u_j \,\de V &= M_j(\eta, \bm{c}) \quad &&\text{for } j = 1,2, \label{eq:integralCond}\\
            \dfrac{1}{2} \snorm{\bm{u}}^2 + g\eta - \sigma \nabla \cdot \left(\dfrac{\nabla \eta}{\sqrt{1+\snorm{\nabla\eta}^2}}\right) &= Q(\bm{c}) &&\text{on } z = \eta, \label{eq:dynBoundary}
    \end{alignat}
\end{subequations}
where $Q(\bm{c}) = \frac{1}{2}(c_1^2+c_2^2)$ and
\[
M_j(\eta, \bm{c})\coloneqq \int_{\Omega_{00}^\eta} U_j(\bm{c}) \,\de V
\]
for the laminar flow $\bm{U}(\bm{c})$; see \eqref{eq:laminar-flow}. Indeed, the div-curl problem \eqref{eq:curl}--\eqref{eq:kinboundary} with the additional integral constraint \eqref{eq:integralCond} admits a unique solution for each fixed surface profile; see Theorem \ref{thm:div-curl-physical-domain}.

Doubly periodic solutions to \eqref{eq:formulation} close to a laminar flow have recently been constructed in \cite{lokharu2020} using multi-parameter bifurcation theory. The contribution of this paper is to extend these solutions globally, and we now summarise the main results and techniques. We show that
\begin{itemize}
    \item all steady, three-dimensional, doubly periodic water waves on Beltrami flows solving \eqref{eq:formulation-without-integral-constraint} are analytic under mild regularity assumptions; see Corollary \ref{cor-app:analyticity} and Remark \ref{rem:analyticity};
    \item after including the integral constraint \eqref{eq:integralCond}, the problem \eqref{eq:formulation} admits a reformulation as `identity plus compact' in appropriate H\"older spaces; see Section \ref{sec:final-reformulation};
    \item in this formulation, a local family of doubly periodic waves bifurcates from the laminar flow, see Theorem \ref{thm:local bifurcation}, which matches the result in \cite{lokharu2020};
    \item the local solutions can be continued to a global family of doubly periodic waves; see Theorem \ref{thm:global-bif-Beltrami}. Specifically, we prove a sharper version of the blow-up alternative, that is, if the solution blows up in an appropriate norm, then the surface gradient $\nabla \eta$ must become unbounded in $C^{0,\gamma}_\mathrm{per}$.
\end{itemize}

To the best of our knowledge, these results are the first rigorous existence results on three-dimensional, doubly periodic water waves with vorticity, which go beyond a local analysis close to the laminar flow solutions. Additionally, although our results are formulated for Beltrami fields ($\alpha \neq 0$), they also apply to the irrotational case $\alpha = 0$. Specifically, to the best of our knowledge, this also provides the first global bifurcation result for three-dimensional water waves, with or without vorticity.

\paragraph{Reformulation as `identity plus compact'.}

The first step in our bifurcation analysis is to find a suitable reformulation of the problem \eqref{eq:formulation} that is amenable to global bifurcation analysis. For this, we first transform the free-boundary problem to a fixed domain using the flattening transform $\frakF : \overline{\Omega^0} \rightarrow \overline{\Omega^\eta}$; see \eqref{eq:flattening-trafo}. This yields a div-curl problem for the flattened velocity field $\bar{\bm{u}}$ on a fixed domain $\Omega^0$, which reads as
\begin{subequations}
    \label{eq:flattened-div-curl-into}
    \begin{alignat}{2}
            \nabla \times (A(\eta)\bar{\bm{u}}) &= \alpha \bar{\bm{u}} \quad &&\text{in } \Omega^0,\\
            \nabla \cdot \bar{\bm{u}} &= 0 &&\text{in } \Omega^0, \\
            \bar{\bm{u}} \cdot \bm{n} &= 0 &&\text{on } \partial \Omega^0, \\
            \int_{\Omega_{00}^0} \bar{u}_j \,\de V  &= M_j \quad &&\text{for } j = 1,2,
    \end{alignat}
\end{subequations}
where $A(\eta) \in \R^{3 \times 3}$ is given by \eqref{eq:Aeta}. In addition, the flattened dynamic boundary condition then becomes
\begin{equation}\label{eq:flattened-dyn-bc-into}
    \dfrac{1}{2} B(\bar{\bm{u}},\eta) + g\eta - \sigma \nabla \cdot \left(\dfrac{\nabla \eta}{\sqrt{1 + |\nabla \eta|^2}}\right) = Q(\bm c)
\end{equation}
with $B$ defined in \eqref{eq:dyn-bc-flattened}.

The goal is to rewrite this as a fixed-point problem of the form `identity plus compact'. This is a well-known approach for global bifurcation analysis since it guarantees that the linearisation of the bifurcation problem about any solution is a Fredholm operator with index zero; see, for example, \cite{rabinowitz1971,kielhofer2012}. This reformulation is based on two main steps.

First, we show that the flattened div-curl problem \eqref{eq:flattened-div-curl-into}, after replacing $\alpha \bar{\bm{u}}$ by a function $\bar{\bm{w}}$, admits a unique solution $\bar{\curlL}$ for given $\bar{\bm{w}}$, $\eta$ and $\bm{M}$, which satisfies a regularity estimate in H\"older spaces; see Theorem \ref{thm:solution-flattened-div-curl}. This result is based on Schauder estimates for div-curl problems, which are proved in Appendix \ref{sec:Schauder estimates}. Here, we also exploit that the solution is unique due to the additional integral constraint to find an upper bound in the regularity estimate purely in terms of $\bar{\bm{w}}$, $\eta$ and $\bm{M}$. Using the operator $\bar{\curlL}$, the div-curl problem \eqref{eq:flattened-div-curl-into} can then be written as a fixed point problem $\bar{\bm{u}} = \bar{\curlL}(\alpha \bar{\bm{u}}, \eta, \bm{M})$ for the flattened velocity field $\bar{\bm{u}}$.

\begin{remark}
\label{rem:reduction to boundary}
In contrast to \cite{lokharu2020, groves2024, seth2024}, we do not reduce the equations to the boundary. The reason is that this is only possible if $\alpha$ is not an eigenvalue of the curl operator \cite{Picard98, YoshidaGiga90}. While our assumptions guarantee that $\alpha$ is not an eigenvalue at the bifurcation point, we cannot guarantee that this property continues to hold along the set of non-trivial solutions.   
\end{remark}

Second, we rewrite the flattened dynamic boundary condition \eqref{eq:flattened-dyn-bc-into} as a fixed point problem for the surface profile $\eta$ for given $\bar{\bm{u}}$ and $\bm{c} = (c_1,c_2)$. For this, we write the mean curvature term as $\curlM(\nabla \eta) : D^2 \eta$ such that the dynamic boundary condition can be written as
\begin{equation*}
    L(\nabla \eta) \eta \coloneqq -\curlM(\nabla \eta) : D^2 \eta + g\eta = - \dfrac{1}{2} B(\bar{\bm{u}},\eta) + Q(\bm c).
\end{equation*}
The main observation is that the equation $L(\bm{h}) \eta = f$ for given $\bm{h}$ and $f$ has a unique solution $\eta = \curlK(\bm{h},f)$ in H\"older spaces using the method of continuity; see Theorem \ref{thm:dynBoundaryCond}. This also shows that the solution satisfies an elliptic regularity estimate in H\"older spaces, where we additionally utilise a strong maximum principle.

\begin{remark}
    A similar analysis of the dynamic boundary condition was recently done in \cite{nguyen2026-03Nonlinearity} in the context of a continuation analysis for periodic travelling waves for Darcy flows. Instead of treating the regularity theory via linear theory as outlined above, the author obtains invertibility and regularity directly for the full nonlinear capillary-gravity operator $L(\nabla \eta) \eta$ under slightly stronger initial regularity assumptions. 
\end{remark}

Combining these two steps, we can rewrite the flattened existence problem \eqref{eq:flattened-div-curl-into}--\eqref{eq:flattened-dyn-bc-into} as a fixed point problem
\begin{equation*}
    \bar{\bm{u}} = \bar{\curlL}(\alpha \bar{\bm{u}}, \eta, \bm{M}), \quad \eta = \curlK\left(\nabla \eta, - \dfrac{1}{2} B(\bar{\bm{u}},\eta) + Q(\bm c)\right),
\end{equation*}
see Section \ref{sec:final-reformulation}. This is indeed of the form `identity plus compact' in appropriate H\"older spaces; see Theorem \ref{thm:compactness} and Corollary \ref{Fredholm_op}. For this, we use the regularity estimates for the solution operators $\bar{\curlL}$ and $\curlK$, together with standard compact embeddings for H\"older spaces. In particular, we use that $\bar{\bm{u}}$ in the fixed point problem for $\eta$ can be replaced by $\bar{\curlL}(\alpha \bar{\bm{u}}, \eta, \bm{M})$, which is of higher regularity due to the regularity estimate for $\bar{\curlL}$.

\paragraph{Local bifurcation analysis.}
In preparation for the global bifurcation analysis, we prove the local bifurcation result in \cite{lokharu2020} in our formulation; see Theorem \ref{thm:local bifurcation}. The main obstacle preventing the use of the classical bifurcation theory by Crandall and Rabinowitz \cite{crandall1971} is that the kernel of the linearisation about the laminar flow at the bifurcation point is, in general, two-dimensional. In fact, it was pointed out in \cite[Rem.~4.7]{lokharu2020} that a two-dimensional kernel is necessary to obtain genuinely three-dimensional waves, and that the case of a one-dimensional kernel leads to $2\tfrac{1}{2}$-dimensional waves; see, for example, \cite[Sec.~1.1.2]{lokharu2020}. Instead, we utilise a recent result on abstract multi-parameter bifurcation theory in \cite{seth2024}, which we recapitulate in Appendix \ref{sec:abstract local bif} for the convenience of the reader; see Theorem \ref{thm:seth}. This shows that there exists a local family of genuinely three-dimensional doubly periodic waves
\begin{equation*}
    \{(\bar{\bm{u}}(\bm{s}), \eta(\bm{s}), \bm{c}(\bm{s})) : \bm{s} \in B_\varepsilon(0;\R^2)\}
\end{equation*}
for $\varepsilon > 0$ sufficiently small, bifurcating from the laminar flow at $\bm{c}(0) = \bm{c}^* \in \R^2$; see Theorem \ref{thm:local bifurcation}. Here, $B_\varepsilon(0;\R^2)$ denotes an open ball of radius $\varepsilon$ centred at the origin in $\R^2$.

\paragraph{Global bifurcation analysis.}

The goal is to extend the local bifurcation branch to a global branch of solutions using analytic global bifurcation theory \cite{buffoni2003}. Similar to the local analysis, the main obstacle in the global bifurcation analysis is again the two-dimensional kernel at the bifurcation point, which prevents the use of classical analytic global bifurcation results from a one-dimensional kernel (see, e.g., \cite[Thm.~9.1.1]{buffoni2003}). Moreover, the fact that the kernel dimension is even presents a serious obstacle for the application of global bifurcation methods based on degree theory (see, e.g., \cite{antman2005}), as the crossing of an even number of eigenvalues through zero typically does not result in a change of degree. We also refer to Section \ref{sec:related-results} for a more detailed review of the multiparameter global bifurcation literature.

We overcome this issue by rewriting the bifurcation problem to treat the parameterisation $\bm{s} \in \R^2$ in the local bifurcation result as new bifurcation parameters. We then reduce the problem to a one-parameter bifurcation problem by fixing a direction $\hat{\bm{s}} = (\hat{s}_1,\hat{s}_2) \in S^1 \subset \R^2$ with $\hat{s}_1\hat{s}_2 \neq 0$, where $S^1$ denotes the unit circle in $\R^2$, and considering the problem along the line $\{\bm{s} = r\hat{\bm{s}} \,:\ r \in \R\}$; see Figure \ref{fig:schematic-global-bifurcation} and \eqref{eq:curlG}. The key observation is now that the linearisation of this new bifurcation problem about a solution on the local bifurcation branch is an isomorphism for all $r \in (0,r_0)$ for some $r_0 > 0$ depending on the choice of $\hat{\bm{s}}$; see Proposition \ref{prop:isomorphism}. We can only prove this result for $\hat{\bm{s}}$ with non-zero components. However, this only excludes cases where the bifurcating solutions locally are $2\tfrac{1}{2}$-dimensional waves; see Remark \ref{rem:choice-of-hat-s}.

\begin{figure}[h!]
    \includegraphics[width=\textwidth]{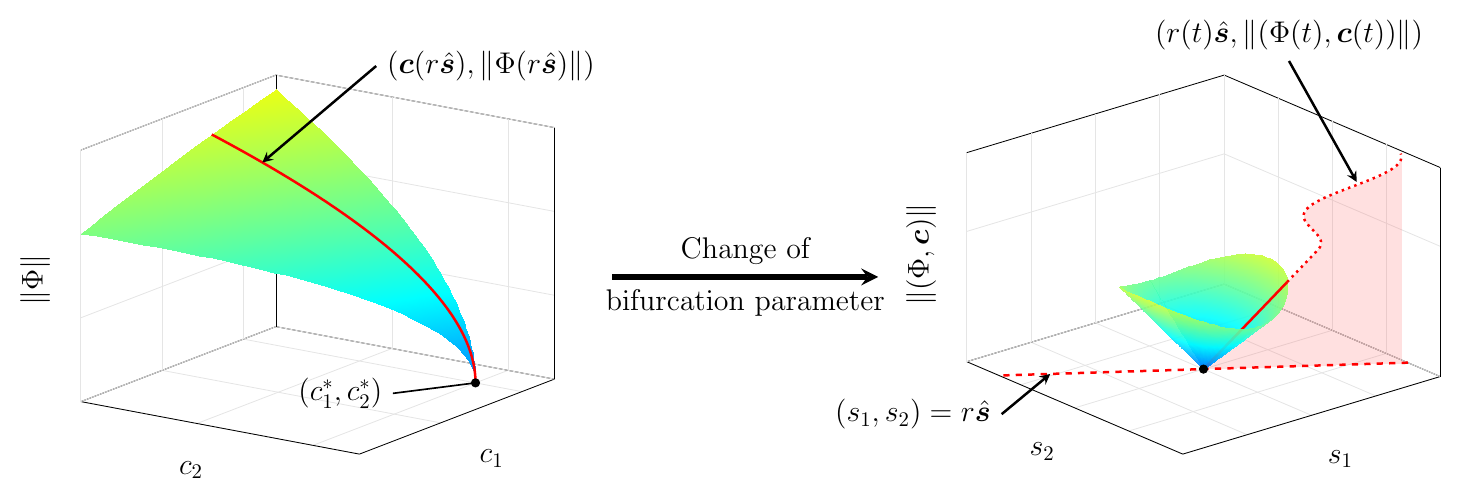}
    \caption{Visual of the global bifurcation strategy. Here, $\Phi = (\bm{u},\bm{c})$ denotes the collection of unknowns. Left: schematic local bifurcation diagram of with surface $(\bm{c}(\bm{s}),\|\Phi(\bm{s})\|)$ of bifurcating solutions parameterised by $\bm{s} \in \R^2$ in a neighbourhood of $\bm{c}^*$. The red curve indicates a one-parameter slice along the direction $\hat{\bm{s}} \in S^1$. Right: schematic bifurcation diagram after changing the bifurcation parameters to $\bm{s}$. The surface again depicts the local bifurcation set with the curve obtained by selecting $\hat{\bm{s}}$ depicted in red. Here, we continue the selected curve globally (dotted red curve) using one-parameter global bifurcation theory. Varying $\hat{\bm{s}}$ then yields a continuum of global bifurcation branches.}
    \label{fig:schematic-global-bifurcation}
\end{figure}

While the bifurcation problem still has a two-dimensional kernel at the bifurcation point, we can apply a variation of the analytic global bifurcation theorem, see Theorem \ref{thm-app:global-bifurcation}, which is based on the observation that the one-dimensional kernel is only needed in the local analysis to guarantee that the linearisation is an isomorphism in a neighbourhood of the bifurcation point. Since this is automatically guaranteed for the reformulated bifurcation problem, we obtain a global continuation for any fixed $\hat{\bm{s}}$ with $\hat{s}_1 \hat{s}_2 \neq 0$; see Figure \ref{fig:schematic-global-bifurcation}. We can then vary $\hat{\bm{s}}$ to obtain a family of global solutions, which bifurcate from the laminar flow; see Theorem \ref{thm:global-bif-Beltrami}.

\begin{remark}
    We expect this approach to apply more generally to multi-parameter global bifurcation problems with a higher-dimensional kernel. In fact, we present the key results providing the local isomorphic structure as well as the necessary variation of the analytic global bifurcation theory in an abstract fashion, which can readily be applied to other problems; see Proposition \ref{prop-app:isomorphism} and Theorem \ref{thm-app:global-bifurcation}.
\end{remark}

\paragraph{Refinement of alternatives.}

For each $\hat{\bm{s}}$, analytic global bifurcation theory yields a number of alternatives, at least one of which has to occur along the bifurcation branch. These are: (i) a blow-up of the solution in an appropriate norm, (ii) that the bifurcation branch approaches the boundary of the domain of definition of the bifurcation problem along a subsequence, or (iii) that the bifurcation branch is periodic. Here, alternative (ii) can be interpreted as the surface profile $\eta$ eventually intersecting with the flat bed $z = -d$.

In our case, the first alternative reads as $(\bar{\bm{u}}, \eta, \bm{c})$ blows up in $C^{k,\gamma}_\mathrm{per}(\overline{\Omega^0};\R^3) \times C^{k+2,\gamma}_\mathrm{per}(\R^2) \times \R^2$ for some $k \geq 0$. It turns out that, using the regularity estimates for the solution operators $\bar{\curlL}$ and $\curlK$ obtained in the reformulation of the problem, this can be reduced to the blow-up of $(\bar{\bm{u}}, \eta, \bm{c})$ in $C^{0,\gamma}_\mathrm{per}(\overline{\Omega^0};\R^3) \times C^{1,\gamma}_\mathrm{per}(\R^2) \times \R^2$.

To further refine the blow-up alternative, we establish a low-order Schauder-type estimate for the (flattened) velocity field in $C^{0, \gamma}_\mathrm{per}$; see \eqref{est-u_0_gamma}. Although its proof follows similar ideas to those in \cite{agmon1964,morrey1966}, to the best of our knowledge, it has not been reported in the literature and may therefore be of independent interest. Additionally, we exploit that for any solution of the bifurcation problem, the flattened velocity field in $C^{0,\gamma}_\mathrm{per}$ can be estimated by the flattened velocity field at the top boundary $z = 0$ and the mean values $M_j$; see Lemma \ref{est-of-u-in-eta}. Combining these two estimates with the Bernoulli equation, we show that blow-up of the velocity field cannot occur as an isolated alternative. Rather, we prove that the blow-up alternative (i) can be refined to the blow-up of $(\eta, \bm{c})$ in $C^{1,\gamma}_\mathrm{per}(\R^2) \times \R^2$; see Theorem \ref{thm:global-bif-Beltrami}.

\begin{remark}
    \label{rem:analyticity}
    Although low-order elliptic estimates are necessary for the fine analysis of the blow-up alternative in the global bifurcation result, see Theorem \ref{thm:global-bif-Beltrami}, the solutions themselves are smooth as long as the surface does not intersect with the bottom, that is, $\eta > -d$. In fact, we show that the solutions are analytic in Appendix \ref{app:analyticity}; see Corollary \ref{cor-app:analyticity}.
\end{remark}

\subsection{Related results}\label{sec:related-results}

Although steady water waves have been of much interest for over a century (see e.g.~\cite{haziot2022} for a recent review), rigorous results on three-dimensional water waves, especially with vorticity, have only been obtained recently. In the irrotational case, the first existence results on doubly periodic capillary-gravity water waves were obtained by \cite{reeder1981, sun1993} in the symmetric case, where the fundamental domain has the shape of a symmetric diamond, and by \cite{craig2000} in the asymmetric case. The former are based on the implicit function theorem and arguments in the spirit of Crandall--Rabinowitz, while the latter is based on a variational Lyapunov--Schmidt reduction and critical point theory. The recent preprint \cite{barbieri2026} improves upon the results in \cite{craig2000} by more fully exploiting the symmetries of the problem. Existence results have also been obtained using a spatial dynamics approach in \cite{groves2001,groves2003,nilsson2019}. We also refer to \cite{ahmad2024} for a result on doubly periodic hydroelastic waves. In the absence of surface tension, the existence of doubly periodic water waves is much more involved since it leads to a small divisor problem. Nevertheless, their existence has been shown in \cite{iooss2009} in the symmetric case and in \cite{iooss2011} in the asymmetric case. However, all results are restricted to local solutions close to the trivial state with a flat surface profile and, to the best of our knowledge, no global results exist even in the irrotational case.

In contrast to the irrotational case, the existence of three-dimensional water waves with vorticity is much less understood. The main challenge compared to the two-dimensional case, which has been studied, for example, in \cite{constantin2004,wahlen2024}, is that the existence problem cannot, in general, be formulated as an elliptic free-boundary problem. In fact, there are a number of non-existence results for rotational water waves, such as the non-existence of genuinely three-dimensional waves with constant vorticity in finite depth \cite{wahlen2014}, which even extends to the case of non-steady water waves \cite{martin2018,geyer2025}; see also \cite{martin2022} for a recent review. Notably, a family of trivially three-dimensional \emph{internal} waves with constant vorticity exists if the densities of the two fluids are the same; see \cite{chen2023}. Finally, we mention the recent result in \cite{seth2026}, which establishes that symmetric, doubly periodic capillary-gravity waves bifurcating from a non-uniform shear flow must be two-dimensional to leading order.

Nevertheless, the existence of three-dimensional rotational water waves has been established recently in two cases. The first case, also considered in this paper, deals with Beltrami flows, where the existence of local bifurcating branches of doubly periodic waves has been shown in \cite{lokharu2020}. A similar result has also been obtained recently for internal waves in Beltrami flows \cite{seth2024}. Both results are based on multi-parameter bifurcation theory, which allows the authors to prove the existence of bifurcating solutions from a two-dimensional kernel using an analytic Lyapunov--Schmidt reduction. A similar result was also obtained using a different formulation of the three-dimensional water wave problem with vorticity based on a generalised Dirichlet--to--Neumann operator \cite{groves2024}, which was first introduced in a variational formulation for steady water waves in Beltrami flows \cite{groves2020}. In this context, we also refer to \cite{lokharu2019} for an alternative variational formulation, which allows for overhanging surface profiles. The second case considers waves with small vorticity, where a local bifurcation result for symmetric diamond waves has been established in \cite{seth2024a}. The proof is based on a construction introduced by Lortz \cite{lortz1970} for magnetohydrostatic equilibria in reflection-symmetric toroidal domains.

In contrast to the inviscid problem, large \emph{viscous} surface waves have recently been constructed. The main difference from the inviscid theory is that the viscous problem requires a forcing term to allow non-trivial travelling surface waves. Therefore, large viscous surface waves have so far been constructed for given large-amplitude external forcing, which plays the role of a control parameter. We specifically point to \cite{nguyen2026-03Nonlinearity}, where large periodic capillary-gravity waves are constructed for Darcy flows for both finite and infinite depth. The proof in \cite{nguyen2026-03Nonlinearity} is based on a reformulation of the problem in the form `identity plus compact' and an application of a global implicit function theorem based on Leray--Schauder degree theory.  While the analysis relies heavily on surface tension, for two-dimensional waves, a similar global continuation result has recently been obtained without it \cite{nguyen2026-01preprint}. Finally, we also point out \cite{brownfield2024-10CommunMathPhys} where large, slowly-travelling waves have been obtained as a perturbation from an explicit family of large stationary solutions for Darcy flows. A similar analysis has also been performed for slowly-travelling surface waves in the free-boundary (Navier--)Stokes equations \cite{banihashemi2026-01preprint}. 

As discussed above, the main technical contribution of this paper is the global extension of the local results in \cite{lokharu2020}, and we provide a strategy to obtain global bifurcation results when the kernel at the bifurcation point is two-dimensional. Classical results in bifurcation theory, see, for example, \cite{rabinowitz1971,buffoni2003}, often rely on the zero eigenvalue at the bifurcation point being simple, which is often satisfied in applications (after exploiting symmetries of the solutions); see, for example, \cite{constantin2016,ehrnstrom2019,wahlen2024}. Nevertheless, results for the bifurcation of solutions from a non-simple kernel have been obtained using topological methods based on degree theory; see, for example, \cite{kielhofer2012} for a local bifurcation result. In the one-parameter case, we refer to the recent result \cite{lopez-gomez2024}, which also generalises analytic bifurcation theory \cite{buffoni2003} to situations where the zero eigenvalue at the bifurcation point has algebraic multiplicity larger than one. However, the results remain restricted to the case that the kernel is one-dimensional. In the multiparameter setting, there are a number of results; see, for example, \cite[Thm.~4.23]{antman2005}. These rely on the assumption that there are points in the parameter space at which the linearisations have different Leray--Schauder degrees. Under this assumption, each $C^1$-curve between these points contains a bifurcation point, and a set of non-trivial solutions with at least the dimension of the parameter space arises.

However, if an even number (counted with multiplicity) of eigenvalues crosses through zero, the Leray--Schauder degree of the linearisation typically does not change; see, for example, \cite[Chap.~II.3]{kielhofer2012}. Therefore, global bifurcation results for problems with a two-dimensional kernel are rare, and often numerical continuation from a local result is used; see, for example, \cite{aston1991,aston1993,akers2019} for studies in the context of mode interactions for two-dimensional gravity-capillary waves. A notable exception is \cite{walsh2014}, where analytic global bifurcation theory is used to establish a global continuation of steady stratified periodic gravity water waves with surface tension in a case where a two-dimensional kernel occurs. In fact, \cite{walsh2014} exploits a similar observation to that in this paper that it is sufficient to check that the linearisation about the local bifurcation is an isomorphism.

\subsection{Outline}

The paper is organised as follows. In Section \ref{sec:reformulation}, we present the reformulation of the free-boundary problem \eqref{eq:formulation} into a fixed point problem of the form `identity plus compact' on a fixed domain. In Section \ref{sec:local-bifurcation}, we prove a local bifurcation result using the reformulation, which matches the result in \cite{lokharu2020}. For this, we use an abstract multi-parameter bifurcation theorem from \cite{seth2024}, which we recapitulate in Appendix \ref{sec:abstract local bif}.  Here, we also provide a reformulation of the bifurcation problem using the parameterisation of the local bifurcation branch along a selected direction as a new bifurcation parameter. In Section \ref{sec:global-bif}, we then continue the local bifurcation branch to a global branch of solutions using a variant of analytic global bifurcation theory provided in Appendix \ref{app:globalBifurcation} and provide a sharper analysis of the blow-up alternative using a low-order Schauder estimate obtained in Appendix \ref{sec:Schauder estimates}. We conclude the main part of the paper by discussing related open problems in Section \ref{sec:discussion}. In Appendix \ref{sec-app:technical-result}, we then collect some technical results used to obtain uniqueness of the div-curl problem with integral constraint \eqref{eq:div-curl-full} in Section \ref{sec:reformulation}. Finally, Appendix \ref{app:analyticity} contains a proof that solutions are in fact locally analytic as long as the surface does not intersect with the bottom.

\section{Reformulation of the problem}\label{sec:reformulation}

The goal of this section is to rewrite the problem \eqref{eq:formulation} in a form which is amenable to global bifurcation analysis. This is achieved by reformulating it as a system of the form `identity plus compact'. Our strategy is to split the problem \eqref{eq:formulation} into the div-curl problem \eqref{eq:curl}--\eqref{eq:integralCond} to determine the velocity field $\bm{u}$ and the dynamic boundary condition \eqref{eq:dynBoundary} to determine the surface profile $\eta$. We then first analyse the div-curl problem individually for a given surface profile; see Section \ref{sec:div-curl-problem}. To make the influence of the surface profile more explicit, we transform the system \eqref{eq:formulation} into a system on a fixed domain $\Omega^0 = \{(\bm{x}',z) \in \R^2 \times \R : -d < z < 0\}$ and establish elliptic regularity estimates in Section \ref{sec:flattening-trafo}. Afterwards, we turn to the dynamic boundary condition in Section \ref{sec:dynamic-boundary-condition} and establish the existence and regularity of a surface profile for a given velocity field. Finally, we combine the results to obtain a reformulation of the full problem as a fixed point problem of the form `identity plus compact' in Section \ref{sec:final-reformulation}.

\subsection{The div-curl problem}\label{sec:div-curl-problem}

Assume that the surface profile $\eta$ satisfies $\eta(x,y) > -d$ for all $(x,y) \in \R^2$ and belongs to the H\"older space $C^{k+2,\gamma}_\mathrm{per}(\R^2)$ for some integer $k\ge 0$ and $\gamma \in (0,1)$, where `per' denotes periodicity with respect to the lattice $\Lambda$; see \eqref{eq:lattice}. We write this as $\eta \in \curlB^{k+2,\gamma}$, where
\[
\curlB^{k, \gamma}=\{  \eta \in C^{k, \gamma}_\mathrm{per}(\R^2) : \min \eta>-d\}.
\]
We first consider \eqref{eq:formulation} without the dynamic boundary condition \eqref{eq:dynBoundary} and write it as
\begin{subequations}
\label{eq:div-curl-full}
    \begin{alignat}{2}
        \nabla \times \bm{u} &= \bm{w}  \quad &&\text{in } \Omega^\eta,  \label{eq:div-curl-full curl}\\
        \nabla \cdot \bm{u} &= 0 &&\text{in } \Omega^\eta,  \label{eq:div-curl-full eq} \\
        \bm{u} \cdot \bm{n} &= 0 &&\text{on } \partial\Omega^\eta,  \label{eq:div-curl-full kinboundary}\\
        \int_{\Omega_{00}^\eta} u_j \,\de V &= M_j\quad &&\text{for } j = 1,2,  \label{eq:div-curl full integralCond}
    \end{alignat}
\end{subequations}
where we have replaced $\alpha \bm{u}$ in the right-hand side by a general vector field $\bm{w} \in C^{k,\gamma}_\mathrm{per, div, T}(\overline{\Omega^\eta}; \R^3)$, in which
\[
 C^{k,\gamma}_\mathrm{per, div, T}(\overline{\Omega^\eta}; \R^3)\coloneqq \{ \bm{w}\in C_\mathrm{per}^{k, \gamma}(\overline{\Omega^\eta}; \R^3) : \nabla \cdot \bm{w}=0 \text{ in } \Omega^\eta, \, \bm{w}\cdot \bm{n}=0 \text{ on } \partial \Omega^\eta\}.
\]
Note that we will sometimes leave out the domain and/or codomain from the function spaces when they are clear from the context, and that the condition $\nabla \cdot \bm{w}=0$ is interpreted in a weak sense when $k=0$.
We begin by constructing a solution to \eqref{eq:div-curl-full curl}--\eqref{eq:div-curl-full kinboundary} while ignoring the integral conditions. For this, we use the periodic Biot--Savart potential of $\bm{w}$, 
\[
\mathrm{BS}(\bm{w})(\bm{x}) =\mathrm{p.v.}  \dfrac{1}{4\pi} \int_{\Omega^\eta} \bm{w}(\bm{y}) \times \dfrac{\bm{x} - \bm{y}}{\snorm{\bm{x}-\bm{y}}^3} \,d\bm{y} \in C^{k+1,\gamma}_\mathrm{per},
\]
which is divergence-free and satisfies
\[
    \nabla \times \mathrm{BS}(\bm{w}) = \bm{w}.
\]
Indeed, the regularity of $\mathrm{BS}(\bm{w})$ can, for example, be proved by a straightforward modification of \cite[Lemma 2.2]{lokharu2019}. We can therefore solve \eqref{eq:div-curl-full curl}--\eqref{eq:div-curl-full kinboundary} by adding a correction $\nabla \phi$ where
\begin{alignat*}{2}
        \Delta \phi &= 0 \quad &&\text{in } \Omega^\eta, \\
        \partial_n \phi &= -\mathrm{BS}(\bm{w}) \cdot \bm{n}  \quad &&\text{on } \partial\Omega^\eta
\end{alignat*}
to satisfy the boundary condition \eqref{eq:div-curl-full kinboundary}.
Since $\int_{\partial\Omega_{00}^\eta} \mathrm{BS}(\bm{w}) \cdot \bm{n}  \, \de S = 0$ by the divergence theorem, this problem has a solution which is unique up to a constant, and the solution is in $C^{k+2,\gamma}_\mathrm{per}$.
We thus obtain a solution to the original problem, including the integral condition \eqref{eq:div-curl full integralCond}, of the form $\bm{u}=\mathrm{BS}(\bm{w}) +\nabla \phi +\bm{v}$, if $\bm{v}$ solves the homogeneous 
problem
\begin{alignat*}{2}
        \nabla \times \bm{v} &= 0 \quad &&\text{in } \Omega^\eta, \\
        \nabla \cdot \bm{v} &= 0 &&\text{in } \Omega^\eta, \\
        \bm{v} \cdot \bm{n} &= 0 &&\text{on } \partial\Omega^\eta, \\
        \int_{\Omega_{00}^\eta} v_j \,\de V &= \widehat{M}_j \quad &&\text{for } j = 1,2,
\end{alignat*}
with $\widehat{M}_j = M_j-\int_{\Omega_{00}^\eta} (\mathrm{BS}(\bm{w})_j-\partial_j \phi) \,\de V$.
In the appendix, Proposition \ref{prop:div-curl-isomorphism}, it is proved that this problem indeed has a unique solution in $C^{k+1,\gamma}_\mathrm{per}$. Hence, we have proven the following result.

\begin{theorem}\label{thm:div-curl-physical-domain}
    Let $k \geq 0$. Then there exists an operator 
    \begin{equation*}
        \curlL \colon C^{k,\gamma}_\mathrm{per, div, T}(\overline{\Omega^\eta};\R^3) \times \curlB^{k+2, \gamma}\times \R^2 \to  C^{k+1,\gamma}_\mathrm{per, div, T}(\overline{\Omega^\eta};\R^3),
    \end{equation*}
    which maps 
    $(\bm{w}, \eta, \bm{M})$
    to the unique solution of \eqref{eq:div-curl-full}.
\end{theorem}

\subsection{Flattening transformation and elliptic regularity estimates}\label{sec:flattening-trafo}

To capture the dependence of the solution $\bm{u}$ on the surface $\eta$, it is useful to introduce the sets
\[
\curlB^{k, \gamma}_\delta=\{  \eta \in C^{k, \gamma}_\mathrm{per}(\R^2) : \min \eta>-d+\delta\}, \quad 0<\delta< d, 
\]
and
\[
\curlB^{k, \gamma}_{\delta, R}=\{  \eta \in \curlB_{\delta}^{k, \gamma} : \|\eta\|_{C^{k, \gamma}} <R\}, \quad 0<\delta <d, R>0.
\]
For the most part, the precise dependence on $\delta$ and $R$ is not important, and we will simply work with $\curlB^{k, \gamma}_{R^{-1}, R}$ for $R > \tfrac{1}{d}$ to reduce the number of parameters.
The strategy to study how $\bm{u}$ depends on $\eta$ is to 
transform the problem onto a fixed domain. For this we follow the treatment in \cite{lokharu2020} and introduce the flattening transformation
\begin{equation}\label{eq:flattening-trafo}
    \frakF \colon \overline{\Omega^0} \to \overline{\Omega^\eta}, \quad (\bar{x},\bar{y},\bar{z}) \mapsto \left(\bar{x}, \bar{y}, \bar{z} + \eta(\bar{x},\bar{y}) \left(\dfrac{\bar{z}}{d} + 1\right)\right).
\end{equation}
Then, following for example \cite{seth2024a}, for a vector field $\bm{u} \in \Holderspace{k}(\overline{\Omega^\eta};\R^3)$ we define the corresponding flattened vector field $\bar{\bm{u}} \in \Holderspace{k}(\overline{\Omega^0};\R^3)$ by
\[
    \bar{\bm{u}} = \frakG \bm{u} \eqqcolon \det(D\frakF) D\frakF^{-1} \bm{u} \circ \frakF,
\]
where $D\frakF$ denotes the Fréchet derivative of $\frakF$ and $D\frakF^{-1}$ the pointwise inverse. Note that $D\frakF(x,y,z)$ is invertible at $(x,y,z) \in \Omega^0$ if $\eta(x,y) > -d$. Additionally, since $D\frakF$ contains at most first derivatives of $\eta$, the following lemma holds.

\begin{lemma}\label{lem:bound-flattening-trafo}
    Let $\eta \in \curlB^{k+1, \gamma}_{R^{-1}, R}$. Then, the operator $\frakG$ is a linear and continuous map from $C^{k,\gamma}_\mathrm{per, div, T}(\overline{\Omega^\eta};\R^3)$ to $C^{k,\gamma}_\mathrm{per, div, T}(\overline{\Omega^0};\R^3)$ with
    \[
        \norm[\Holderspace{k}(\overline{\Omega^0};\R^3)]{\frakG \bm{u}} \leq C(R) \norm[\Holderspace{k}(\overline{\Omega^\eta}; \R^3)]{u},
    \]
    where $C(R) < \infty$ is a constant only depending on $R$.
    Additionally, $\frakG$ is invertible with a bounded inverse satisfying
    \[
        \norm[\Holderspace{k}(\overline{\Omega^\eta};\R^3)]{\frakG^{-1} \bar{\bm{u}}} \leq C(R) \norm[\Holderspace{k}(\overline{\Omega^0}; \R^3)]{\bar{\bm{u}}}.
    \]
\end{lemma}

Lemma \ref{lem:bound-flattening-trafo} shows that if $\bm{u}$ solves \eqref{eq:div-curl-full}, then the flattened function $\bar{\bm{u}} = \frakG \bm{u}$ is a solution to the flattened system
\begin{subequations}
    \label{eq:formulation-flat}
    \begin{alignat}{2}
            \nabla \times (A(\eta)\bar{\bm{u}}) &=\bar{\bm{w}} \quad &&\text{in } \Omega^0, \label{eq:curl-flat}\\
            \nabla \cdot \bar{\bm{u}} &= 0 &&\text{in } \Omega^0, \label{eq:div-flat}\\
            \bar{\bm{u}} \cdot \bm{n} &= 0 &&\text{on } \partial \Omega^0, \label{eq:kinboundary-flat} \\
            \int_{\Omega_{00}^0} \bar{u}_j \,\de V  &= M_j \quad &&\text{for } j = 1,2. \label{eq:integralCond-flat} 
    \end{alignat}
\end{subequations}
Here, $A(\eta) \in \R^{3\times 3}$ is given by
\begin{equation}\label{eq:Aeta}
        A(\eta) = \dfrac{1}{\det(D\frakF)}D\frakF^T D\frakF.
\end{equation}

In Appendix \ref{sec:Schauder estimates}, we prove Schauder estimates for \eqref{eq:formulation-flat} which are uniform in $\eta \in \curlB^{k+2,\gamma}_{\delta,R}$; see Theorem \ref{thm-app:schauder}.
Combining this with the existence theory for \eqref{eq:div-curl-full}, we obtain the following result.

\begin{corollary}\label{cor:schauder}
    Let $k \geq 0$ and $\eta \in \curlB^{k+2,\gamma}_{R^{-1}, R}$. Additionally, let $\bar{\bm{w}} \in C^{k,\gamma}_\mathrm{per, div, T}(\overline{\Omega^0}; \R^3)$. Then, there exists a unique solution $\bar{\bm{u}} \in C^{k+1,\gamma}_\mathrm{per, div, T}(\overline{\Omega^0}; \R^3)$ to \eqref{eq:formulation-flat}, which satisfies
    \begin{equation}\label{eq:schauder-estimate}
        \Holdernormper{\bar{\bm{u}}}{k+1}{\gamma} \leq C(R) (\|\bar{\bm{u}}\|_{C^0_\mathrm{per}} + \Holdernormper{\bar{\bm{w}}}{k}{\gamma}).
    \end{equation}
\end{corollary}
\begin{proof}
    We recall from Theorem \ref{thm:div-curl-physical-domain} that there exists a unique solution to the div-curl problem in the original variables \eqref{eq:div-curl-full} with right-hand side given by $\frakG^{-1} \bar{\bm{w}} \in C^{k,\gamma}_{\mathrm{per, div, T}}$, where we also use the mapping properties of the flattening transform provided in Lemma \ref{lem:bound-flattening-trafo}. This provides a unique solution to the flattened system with the desired regularity by applying the flattening transformation. Using the equivalence of \eqref{eq:div-curl-full} and \eqref{eq:formulation-flat}, this establishes the existence and uniqueness of a solution with desired regularity. The Schauder estimate \eqref{eq:schauder-estimate} then follows from Theorem \ref{thm-app:schauder}.
\end{proof}

It turns out that the lower order contribution $\|\bar{\bm{u}}\|_{C^0_\mathrm{per}}$ can be removed, using that the solution of the physical div-curl problem \eqref{eq:div-curl-full} together with the boundary and integral conditions is unique in $\Holderspace{k}$ for any $\gamma \in (0,1)$; see Theorem \ref{thm:div-curl-physical-domain}. We point out that this also implies that the corresponding flattened problem is uniquely solvable due to the invertibility of the flattening transform; see Lemma \ref{lem:bound-flattening-trafo}.

\begin{proposition}\label{prop:schauder-est-no-lower-order}
    Let $k \geq 0$, $\bm{M} \in \R^2$, $\bar{\bm{w}} \in C^{k, \gamma}_\mathrm{per, div, T}(\overline{\Omega^0}; \R^3)$  and $\eta\in \curlB^{k+2, \gamma}_{R^{-1}, R}$.
    Additionally, let $\bar{\bm{u}}$ be the solution to \eqref{eq:formulation-flat}. Then, $\bar{\bm{u}}$ satisfies the inequality
    \begin{equation}\label{eq:schauder-without-lower-order-term}
        \Holdernormper{\bar{\bm{u}}}{k+1}{\gamma} \leq C(R) (|\bm{M}|+ \Holdernormper{\bar{\bm{w}}}{k}{\gamma}).
    \end{equation}
\end{proposition}

\begin{proof}
    If $M_1 = M_2 = 0$, the desired estimate follows from a standard contradiction argument using the uniqueness of solutions to the flattened div-curl problem \eqref{eq:curl-flat}--\eqref{eq:integralCond-flat}; see Corollary \ref{cor:schauder}. In the case $(M_1,M_2) \neq (0,0)$ we introduce
    \[
        \tilde{\bar{\bm{u}}} = \bar{\bm{u}} + \dfrac{M_1}{|\Omega_{00}^0|} \bm{e}_1  + \dfrac{M_2}{|\Omega_{00}^0|} \bm{e}_2.
    \]
    Then, $\tilde{\bar{\bm{u}}}$ satisfies
    \[
        \begin{split}
            \curl (A(\eta) \tilde{\bar{\bm{u}}}) &= \bar{\bm{w}} + \dfrac{M_1}{|\Omega_{00}^0|} \curl (A(\eta) \bm{e}_1) + \dfrac{M_2}{|\Omega_{00}^0|} \curl (A(\eta) \bm{e}_2), \\
            \nabla \cdot \tilde{\bar{\bm{u}}} &= 0
        \end{split}
    \]
    in $\Omega^0$ and
    \[
        \tilde{\bar{\bm{u}}} \cdot \bm{n} = \bar{\bm{u}} \cdot \bm{n} = 0
    \]
    on the boundary $\partial \Omega^0$ using that $\bm{n} = (0,0,1)^T$ is the normal direction on the flattened boundary. Finally, the integral condition
    \[
        \int_{\Omega_{00}^0} \bm{e}_j \cdot \tilde{\bar{\bm{u}}} \,\de V  = \int_{\Omega_{00}^0} \bm{e}_j \cdot \bar{\bm{u}} \,\de V  - M_j = 0
    \]
    holds for $j = 1,2$. Therefore, we can apply \eqref{eq:schauder-without-lower-order-term} in the case $M_1 = M_2 = 0$ to $\tilde{\bar{\bm{u}}} \in \Holderspace{k+1}(\Omega^0)$, $\eta \in \Holderspace{k+2}(\R^2)$ and 
    \[
        \tilde{\bar{\bm{w}}} \coloneqq \bar{\bm{w}} + \dfrac{M_1}{|\Omega_{00}^0|} \curl (A(\eta) \bm{e}_1) + \dfrac{M_2}{|\Omega_{00}^0|} \curl (A(\eta) \bm{e}_2) \in \Holderspace{k}(\Omega^0)
    \]
    since $A(\eta)$ contains at most one derivative of $\eta$. This yields the estimate
    \begin{equation}\label{eq:estimate-tildeu}
        \Holdernormper{\tilde{\bar{\bm{u}}}}{k+1}{\gamma} \leq C(R) \Holdernormper{\tilde{\bar{\bm{w}}}}{k}{\gamma} \leq C(R) (|M_1| + |M_2| + \Holdernormper{\bar{\bm{w}}}{k}{\gamma})
    \end{equation}
    with a different constant in the second step, which still depends only on $R$. To obtain \eqref{eq:schauder-without-lower-order-term} we finally estimate
    \[
        \Holdernormper{\bar{\bm{u}}}{k+1}{\gamma} = \Holdernormper{\tilde{\bar{\bm{u}}} - \dfrac{M_1}{|\Omega_{00}^0|} \bm{e}_1  - \dfrac{M_2}{|\Omega_{00}^0|} \bm{e}_2}{k+1}{\gamma} \leq  \Holdernormper{\tilde{\bar{\bm{u}}}}{k+1}{\gamma} + \dfrac{1}{|\Omega_{00}^0|}(|M_1| + |M_2|)
    \]
    and then apply \eqref{eq:estimate-tildeu}. This completes the proof.
\end{proof}

Combining Theorem \ref{thm:div-curl-physical-domain} and Proposition \ref{prop:schauder-est-no-lower-order}, we find the following result.

\begin{theorem}\label{thm:solution-flattened-div-curl}
    Let $k \geq 0$. Then there exists an operator
    \begin{equation*}
        \bar{\curlL} \colon C^{k, \gamma}_\mathrm{per, div, T}(\overline{\Omega^0}; \R^3) \times \curlB^{k+2, \gamma} \times \R^2 \to C^{k+1, \gamma}_\mathrm{per, div, T}(\overline{\Omega^0}; \R^3),
    \end{equation*} 
    which maps $(\bar{\bm{w}},\eta,\bm{M}) \in   C^{k, \gamma}_\mathrm{per, div, T} \times \curlB^{k+2, \gamma} \times \R^2$ to the unique solution to \eqref{eq:formulation-flat} in $C^{k+1, \gamma}_\mathrm{per, div, T}$. In particular, $\bar{\curlL}$ satisfies the estimate
    \[
        \Holdernormper{\bar{\curlL}(\bar{\bm{w}},\eta,\bm{M})}{k+1}{\gamma} \leq C(R) (|\bm{M}| + \Holdernormper{\bar{\bm{w}}}{k}{\gamma})
    \]
for $\eta \in \curlB^{k+2, \gamma}_{R^{-1}, R}$.
\end{theorem}

\subsection{The dynamic boundary condition}\label{sec:dynamic-boundary-condition}
Recall  that the dynamic boundary condition \eqref{eq:dynBoundary} is given by
\[
    \dfrac{1}{2}|\bm{u}|^2 + g\eta -2\sigma K_M = Q(\bm c)
\]
with the mean curvature $K_M$. Under the flattening transform, the boundary condition transforms to
\begin{equation}\label{eq:dyn-bc-flattened}
    \dfrac{1}{2} B(\bar{\bm{u}},\eta) + g\eta - \sigma \nabla \cdot \left(\dfrac{\nabla \eta}{\sqrt{1 + |\nabla \eta|^2}}\right) = Q(\bm c).
\end{equation}
Here, $B$ is given by
\begin{equation}\label{eq:def-B-dyn-bc}
    B(\bar{\bm{u}},\eta) = \dfrac{1}{\det(D\frakF)^2} |D\frakF \bar{\bm{u}}|^2.
\end{equation}
Note that $B$ satisfies the estimate
\begin{equation}\label{eq:est-B-dyn-bc}
    \Holdernormper{B(\bar{\bm{u}},\eta)}{k}{\gamma} \leq C(R)\Holdernormper{\bar{\bm{u}}}{k}{\gamma}^2
\end{equation}
for any $k \geq 0$ and $\eta \in \curlB^{k+1,\gamma}_{R^{-1},R}$. 

To analyse the dynamic boundary condition, we define the operator
\begin{align*}
	L(\bm{h}) \eta = - \curlM (\bm{h}) : D^2\eta + g \eta,
\end{align*}
and 
\begin{align}\label{eq:def-curlM}
	   \curlM = \curlM(\bm{h}) = \dfrac{\sigma}{(1+\snorm{\bm{h}}^2)^{3/2}} \begin{pmatrix}
		1 + h_2^2 & - h_1 h_2 \\
		-h_2 h_1 & 1 + h_1^2
	\end{pmatrix}.
\end{align}
Here $D^2 \eta$ denotes the Hessian of $\eta$ and $A : B$ denotes a double contraction of the matrices $A, B \in \R^{2 \times 2}$ defined by
\begin{align*}
	A : B = \sum_{i,j = 1}^2 A_{ij} B_{ij}.
\end{align*}

\begin{remark}
	Note that we can recover the $\eta$-dependent part of the dynamic boundary condition \eqref{eq:dynBoundary} for $\bm{h} = \nabla \eta$, that is, \eqref{eq:dynBoundary} can be written as $L(\nabla \eta)\eta = - \tfrac{1}{2} |\bm{u}|^2 + Q(\bm c)$. Similarly, the flattened dynamic boundary condition \eqref{eq:dyn-bc-flattened} then takes the form
    \begin{equation*}
        L(\nabla \eta) \eta = -\dfrac{1}{2}B(\bar{\bm{u}},\eta) + Q(\bm c).
    \end{equation*}
\end{remark}

The matrix $\curlM$ is symmetric and positive definite with eigenvalues
\begin{align*}
	\lambda_1 = \dfrac{1}{(1+\snorm{\bm{h}}^2)^{3/2}}, \quad \lambda_2 = \dfrac{1}{(1+\snorm{\bm{h}}^2)^{1/2}}.
\end{align*}
This in particular means that there exists a constant $\lambda = \lambda(\norm[L^\infty]{\bm{h}})$ such that $L$ is strictly (uniformly) elliptic.
The constant $\lambda$ blows up if $\norm[L^\infty]{\bm{h}} \to +\infty$.
The following result holds.

\begin{theorem}\label{thm:dynBoundaryCond}
	Let $k \geq 0$. Then there exists an operator $\mathcal{K} \colon \Holderspace{k}(\R^2) \times \Holderspace{k}(\R^2) \to \Holderspace{k+2}(\R^2)$, which maps $(\bm{h},f) \in \Holderspace{k}(\R^2) \times \Holderspace{k}(\R^2)$ to the (unique) solution of
	\begin{align}
		L(\bm{h}) \eta = f
		\label{eq:pde}
	\end{align}
	in $\Holderspace{k}(\R^2)$. The operator depends on the constant $\lambda$, the gravity $g$ and on $\norm[L^\infty]{\bm{h}}$ and the estimate
    \begin{equation}\label{eq:schauder-est-dyn-bc}
        \Holdernormper{\curlK(\bm{h},f)}{k+2}{\gamma} \leq C(\Holdernormper{\bm{h}}{k}{\gamma}) \Holdernormper{f}{k}{\gamma}
    \end{equation}
    holds.
\end{theorem}
\begin{remark}
	In the following, $C$ is a generic constant, which might depend on $\lambda$, $g$ and $\norm[L^\infty]{\bm{h}}$.
\end{remark}
\begin{proof}
	We prove the result in two steps. First, we show that \eqref{eq:pde} has a classical solution in $\Holderspace{2}$.
	Then we apply the interior regularity result of \cite[Theorem 6.17]{gilbarg2001} to solutions on an arbitrary domain $\Omega \subset \R^2$ with smooth boundary, which contains at least one periodic cell, to obtain that $\eta \in \Holderspace{k+2}$ for any $k \geq 0$.
	
	To obtain a classical solution of \eqref{eq:pde} in $\Holderspace{0}$, that is, for $f \in \Holderspace{0}$ there exists an $\eta \in \Holderspace{2}$, which satisfies \eqref{eq:pde}, we follow the proof of \cite[Theorem 6.8]{gilbarg2001} using the \emph{method of continuity}.
	Therefore, define 
	\begin{align*}
		L_0 \coloneqq -\Delta + g, \qquad L_1 \coloneqq - \curlM : D^2 + g.
	\end{align*}
	Both $L_0$ and $L_1$ are bounded operators from $\Holderspace{2}$ to $\Holderspace{0}$.
	Then, we define the interpolation operator $L_t \coloneqq (1-t) L_0 + t L_1 \colon \perspace{2} \to \perspace{0}$ for $t \in [0,1]$, which satisfies the estimate
	\begin{align*}
		\Holdernormper{\eta}{2}{\gamma} \leq C \Holdernormper{L_t \eta}{0}{\gamma}
	\end{align*}
	for all $t \in [0,1]$ and $\eta \in \Holderspace{2}$ with a constant $C < \infty$ independent of $t$ and $\eta$.
	
	Fix $t \in [0,1]$ and $\eta \in \Holderspace{2}$ and define $f \coloneqq L_t \eta \in \Holderspace{0}$.
	We show in the following Lemma \ref{lem:comparisonPrincipleEstimate} that the a priori estimate
	\begin{align}
		\pernorm{\eta}{0} \leq \dfrac{1}{g} \pernorm{f}{0} \leq \dfrac{1}{g}\Holdernormper{f}{0}{\gamma}
		\label{eq:aprioriEstimate}
	\end{align}
	holds for all $g > 0$.
	Furthermore, since $L_t \eta = f$ on any domain $\Omega$ with smooth boundary the interior Schauder estimate
	\begin{align}
		\norm[C^{2,\gamma}(\Omega')]{\eta} \leq C \left( \norm[C^0(\Omega)]{\eta} + \norm[C^{0,\gamma}(\Omega)]{f}\right)
		\label{eq:interiorSchauderEstimesGilbarg}
	\end{align}
	holds for any strict, compact subset $\Omega'$ of $\Omega$; see \cite[Corollary 6.3]{gilbarg2001}.
	Assuming that $\Omega$ and $\Omega'$ contain at least one periodic cell, we obtain the estimate
	\begin{align*}
		\Holdernormper{\eta}{2}{\gamma} \leq C \left(\pernorm{\eta}{0} + \Holdernormper{f}{0}{\gamma}\right) \leq C \Holdernormper{f}{0}{\gamma},
	\end{align*}
	where we used \eqref{eq:aprioriEstimate} in the second inequality.
	Thus, by the method of continuity \cite[Theorem 5.2]{gilbarg2001} we find that $L_1$ is surjective if and only if $L_0$ is. By standard results, $L_0$ is invertible since $g > 0$. Additionally, using that solutions of $L_1 \eta = f$ are unique due to \eqref{eq:aprioriEstimate}, we find that $L_1 = L$ is invertible and there exists a classical solution $\eta \in \Holderspace{2}$ of \eqref{eq:pde} for all $f \in \Holderspace{0}$.
	
	As mentioned in the beginning of the proof, we now use that $\curlM \in \Holderspace{k}$ (since $h \in \Holderspace{k}$) and $f \in \Holderspace{k}$ to find that the classical solution $\eta$ constructed above is actually in $\Holderspace{k+2}$; see \cite[Theorem 6.17]{gilbarg2001}. Finally, the estimate \eqref{eq:schauder-est-dyn-bc} follows by applying interior Schauder estimates \cite[Theorem 6.2]{gilbarg2001} and using \eqref{eq:aprioriEstimate}. This completes the proof.
\end{proof}

It remains to prove the a priori estimate \eqref{eq:aprioriEstimate}.

\begin{lemma}\label{lem:comparisonPrincipleEstimate}
	Let $\eta \in \perspace{2}$ be a solution to $L_t \eta = f$ for $f \in \perspace{0}$ and $t \in [0,1]$. Then the estimate
	\begin{align*}
		\pernorm{\eta}{0} \leq \dfrac{1}{g} \pernorm{f}{0}
	\end{align*}
	holds.
\end{lemma}
\begin{proof}
	The proof uses the strong maximum principle.
	Define $\tilde{\eta} \coloneqq \pernorm{f}{0}/g$. Then $v \coloneqq \eta - \tilde{\eta}$ satisfies
	\begin{align*}
		L_t v = f - \pernorm{f}{0} \leq 0
	\end{align*}
	on any domain $\Omega$.
	Now, assume that there exists an $\bm{x}' \in \R^2$ such that $v(\bm{x}') \geq 0$ and $v$ is not constant. Since $v$ is periodic, we can choose the domain $\Omega$ such that $v$ has a non-negative maximum in the interior of $\Omega$. Hence, using the strong maximum principle, we find that $v$ is constant in $\Omega$ -- a contradiction.
	Therefore, either $v(\bm{x}') \leq 0$ or $v(\bm{x}') = \text{const.} > 0$.
	However, in the latter case, we find $L_t v = g v$ and $L_t v \leq 0$, which yields $g v \leq 0$.
	Thus, since $g > 0$ this is impossible, and we find $v(\bm{x}') \leq 0$, which in turn implies that
	\begin{align*}
		\sup_{\bm{x}'} \eta(\bm{x}') \leq \tilde{\eta} = \dfrac{1}{g} \pernorm{f}{0}.
	\end{align*}
	
	Next, we define $u \coloneqq \eta + \tilde{\eta}$.
	This yields that
	\begin{align*}
		L_t u = f + \pernorm{f}{0} \geq 0.
	\end{align*}
	Following the same argument as above, we find $u(\bm{x}') \geq 0$.
	Thus,
	\begin{align*}
		\tilde{\eta} + \inf_{\bm{x}'} \eta(\bm{x}') \geq 0 \quad\Longleftrightarrow\quad \dfrac{1}{g} \pernorm{f}{0} \geq - \inf_{\bm{x}'} \eta(\bm{x}') = \sup_{\bm{x}'} (-\eta(\bm{x}')).
	\end{align*}
	Combining both estimates, we obtain the statement of the lemma.
\end{proof}

\subsection{Final reformulation}\label{sec:final-reformulation}

We can now write the flattened problem \eqref{eq:formulation-flat} and \eqref{eq:dyn-bc-flattened} as a fixed point problem for $(\bar{\bm{u}},\eta)\in C^{k, \gamma}_\mathrm{per, div, T}(\overline{\Omega^0}; \R^3) \times \curlB^{k+2, \gamma}$ of the form
\begin{align*}
    \bar{\bm{u}} &= \bar{\curlL}(\alpha \bar{\bm{u}}, \eta, \bm{M}(\eta, \bm{c})), \\
    \eta &= \mathcal{K}\left(\nabla \eta, Q(\bm{c}) - \frac{1}{2} B(\bar{\curlL}(\alpha \bar{\bm{u}}, \eta, \bm{M}(\eta, \bm{c})),\eta)\right) \eqqcolon \bar{\curlK}(\bar{\bm{u}}, \eta, \bm{c}).
\end{align*}
Therefore, the flattened problem \eqref{eq:formulation-flat} and \eqref{eq:dyn-bc-flattened} is equivalent to the formulation
\begin{align*}
    0 = \Phi - \mathcal{C}(\Phi, \bm{c}) \eqqcolon \curlF(\Phi,\bm{c})
\end{align*}
with $\Phi = (\bar{\bm{u}}, \eta) \in C^{k, \gamma}_\mathrm{per, div, T} \times \curlB^{k+2, \gamma}$ and
\begin{align*}
    \mathcal{C}(\Phi,\bm{c}) = \begin{pmatrix}
            \bar{\curlL}(\alpha \bar{\bm{u}}, \eta, \bm{M}(\eta, \bm{c})) \\
            \bar{\curlK}(\bar{\bm{u}}, \eta, \bm{c})
        \end{pmatrix}.
\end{align*}
The following result, which is the combination of Theorems \ref{thm:solution-flattened-div-curl} and \ref{thm:dynBoundaryCond}, now shows that this is of the form `identity plus compact'.

\begin{theorem}\label{thm:compactness}
    Fix $\delta\in (0, d)$ and $k \geq 0$. Then the operator $\curlC \colon  C^{k, \gamma}_\mathrm{per, div, T}(\overline{\Omega^0}; \R^3) \times \curlB^{k+2, \gamma}_\delta \times \R^2 \to C^{k, \gamma}_\mathrm{per, div, T}(\overline{\Omega^0}; \R^3)  \times C^{k+2, \gamma}_\mathrm{per}(\R^2)$ is compact.
\end{theorem}
\begin{proof}
    Let $(\Phi_\ell)_{\ell \in \N} = (\bar{\bm{u}}_\ell, \eta_\ell)_{\ell\in\N}$ be a bounded sequence in $C^{k, \gamma}_\mathrm{per, div, T} \times \curlB^{k+2, \gamma}_\delta$ and $(\bm{c}_\ell)_{\ell \in \N}$ be a bounded sequence in $\R^2$. We show that this implies that the sequence $(\curlC(\Phi_\ell,\bm{c}_\ell))_{\ell\in\N}$ is bounded in $\Holderspace{k+1}(\overline{\Omega^0}; \R^3) \times \Holderspace{k+3}(\R^2)$. This establishes compactness due to compact embedding of H\"older spaces. Note first  that $\{\eta_\ell\} \subset \curlB^{k+2,\gamma}_{R^{-1}, R}$ for some sufficiently large $R$ and that $\bm{M}_\ell\coloneqq \bm{M}(\eta_\ell, \bm{c}_\ell)$ is bounded in $\R^2$.
    
The bound on the first component follows from Theorem \ref{thm:solution-flattened-div-curl}, which gives the estimate
    \[
        \Holdernormper{\bar{\curlL}(\alpha\bar{\bm{u}}_\ell,\eta_\ell, \bm{M}_\ell)}{k+1}{\gamma} \leq C(R) (|\bm{M}_\ell| + \Holdernormper{\bar{\bm{u}}_\ell}{k}{\gamma}).
    \]
    The right-hand side is bounded uniformly in $\ell \in \N$. 
    
    For the bound on the second component, we use the estimate established in Theorem \ref{thm:dynBoundaryCond}, which reads as
    \[
        \begin{split}
            \Holdernormper{\bar{\curlK}(\Phi_\ell,\bm{c}_\ell)}{k+3}{\gamma} &\leq C(\Holdernormper{\nabla \eta_\ell}{k+1}{\gamma}) \Holdernormper{Q(\bm{c}_\ell) - \tfrac{1}{2}B(\bar{\curlL}(\alpha \bar{\bm{u}}_\ell,\eta_\ell,\bm{M}_\ell),\eta_\ell)}{k+1}{\gamma} \\
            &\leq C(R)(1 + \Holdernormper{B(\bar{\curlL}(\alpha \bar{\bm{u}}_\ell,\eta_\ell,\bm{M}_\ell),\eta_\ell)}{k+1}{\gamma}) \\
            &\leq C\left(R,\Holdernormper{\bar{\curlL}(\alpha \bar{\bm{u}}_\ell,\eta_\ell,\bm{M}_\ell)}{k+1}{\gamma}\right) \\
            &\leq C\left(R,|\bm{M}_\ell|, \Holdernormper{\bar{\bm{u}}_\ell}{k}{\gamma}\right),
        \end{split}
    \]
    where we used estimate \eqref{eq:est-B-dyn-bc} to bound $B$. The above estimate provides an upper bound on $\bar{\curlK}$ in $\Holderspace{k+3}(\R^2)$, which is uniform in $\ell \in \N$. This completes the proof.
\end{proof}

In the bifurcation theory below, it is convenient to impose symmetry conditions. For a fixed $k\ge 0$ and $\gamma \in (0,1)$,  we therefore introduce the space
\[
\curlX=\{ (\bar{\bm{u}}, \eta) \in C^{k, \gamma}_\mathrm{per, div, T}(\overline{\Omega^0}; \R^3) \times C^{k+2, \gamma}_\mathrm{per}(\R^2) : 
\bar{u}_1, \bar{u}_2, \eta \text{ are even in } \bm{x}' \text{ and } \bar{u}_3 \text{ is odd in } \bm{x}'\}.
\]
\begin{remark}
    In the remaining paper, it is typically not relevant which exact $k \geq 0$ and $\gamma \in (0,1)$ are chosen. Therefore, we simply refer to the function space $\curlX$ without explicitly mentioning the chosen $k$ and $\gamma$.
\end{remark}
For $\delta \in (0, d)$ we also introduce the sets
\begin{equation}\label{eq:U}
\mathcal U= \{(\bar{\bm{u}}, \eta) \in \curlX : \eta \in \curlB^{k+2, \gamma}\},
\end{equation}
and
\begin{equation}\label{eq:Udelta}
    \mathcal U_\delta = \{(\bar{\bm{u}}, \eta) \in \curlX : \eta \in \curlB^{k+2, \gamma}_\delta\},
\end{equation}
noting that $\mathcal U=\bigcup_{0<\delta<d} \mathcal U_\delta$. The next result follows directly by noting that $\mathcal C$ preserves these symmetries.

\begin{corollary}\label{Fredholm_op}
The operator $\curlC \colon \mathcal U_\delta \times \R^2 \to \curlX$ is compact for each fixed $\delta>0$.
\end{corollary}

As a consequence, $\curlF$ is of the form `identity plus compact' also with these symmetries, and hence so is its partial Fréchet derivative with respect to $\Phi$. Thus, $D_\Phi \curlF(\Phi, \bm{c})$ is a Fredholm operator of index $0$ on $\curlX$.

\section{Local bifurcation}\label{sec:local-bifurcation}

In this section, we perform a local bifurcation analysis, following the abstract approach in Appendix \ref{sec:abstract local bif}.
While the existence of small-amplitude solutions has already been proved using local bifurcation in \cite{lokharu2020, groves2024}, we give a short proof here as well for two reasons. On the one hand, our formulation of the problem is slightly different, and on the other hand, we also need a non-degeneracy result along the set of non-trivial small-amplitude solutions (Proposition \ref{prop:isomorphism}), which naturally fits in the framework of Appendix \ref{sec:abstract local bif}.
We thus consider the operator $\mathcal F \colon \mathcal U \times \R^2 \subset \curlX\times \R^2\to \curlX$. Compared to Appendix \ref{sec:abstract local bif}, there are two minor differences. First, the operator is defined on a subset of $\curlX$ rather than the whole space. This has no real consequence since all the arguments are local in nature. Second, the trivial solutions are given by $((\bm{U}(\bm{c}),0),\bm{c})$ rather than $(0, \bm{c})$. This can be solved by considering the modified operator
\[
F(\Psi, \bm{c})=\mathcal F((\bm{U}(\bm{c}),0)+\Psi, \bm{c})
\]
where $\Phi=(\bm{U}(\bm{c}),0)+\Psi$. This is done without explicit mention below.

We now establish the assumptions \ref{H1}--\ref{H6} of the abstract local bifurcation result, Theorem \ref{thm:seth}. We note that \ref{H1} follows directly from the analyticity of $\mathcal F$ and the fact that $\curlF$ is of the form `identity plus compact'. We then analyse the kernel of the linearisation of $\mathcal F$ at $\Phi=(\bm{U}(\bm{c}),0)$ in the direction $\Phi=(\bar{\bm{u}}, \eta)$, which is given by
\[
D_\Phi \curlF((\bm{U}(\bm{c}),0), \bm{c})\Phi=
\begin{pmatrix} \bar{\bm{u}}\\ \eta \end{pmatrix}-
\begin{pmatrix}
D \bar{\curlL}(\alpha \bm{U}(\bm{c}), 0, \bm{M}(0,\bm{c})) (\alpha \bar{\bm{u}}, \eta, D_\eta M_j(0, \bm{c})\eta)\\ 
D_\Phi \bar{\curlK}(\bm{U}(\bm{c}), 0, \bm{c}) ( \bar{\bm{u}}, \eta)
\end{pmatrix},
\]
where $\bar{\bm{v}}=D \bar{\curlL}(\alpha \bm{U}(\bm{c}), 0, \bm{M}(0,\bm{c})) (\alpha \bar{\bm{u}}, \eta, D_\eta M_j(0, \bm{c})\eta)$ is the unique solution of
\begin{alignat*}{2}
            \curl \bar{\bm{v}} &= \alpha\bar{\bm{u}}+\nabla \times \tilde{L}(\eta)\quad && \text{in } \Omega^0, \\
            \nabla \cdot \bar{\bm{v}} &= 0 \quad && \text{in } \Omega^0, \\
           \bar{v}_3 &= 0 \quad  && \text{on } \partial\Omega^0, \\
            \int_{\Omega_{00}^0} \bar{v}_j \,\de V  &=  \int_{B_{00}} U_j(\bm{c})|_{z=0}\eta\,\de x \, \de y \quad  &&\text{for } j = 1,2,
\end{alignat*}
in which
    \[
    \tilde{L}(\eta)\coloneqq  - A'(0)\eta\, \bm{U}(\bm{c}) = \frac{\eta}{d} \bm{U}(\bm{c})-\nabla \eta \cdot \bm{U}(\bm{c})\frac{z+d}{d} \bm{e}_3,
    \]
and we have used the fact that
    \[
            D_\eta M_j(0, \bm{c})\eta =  \int_{B_{00}} U_j(\bm{c})|_{z=0}\eta\,\de x \, \de y.
    \]
    We also find that $\zeta = D_\Phi \bar{\curlK}(\bm{U}(\bm{c}),0,\bm{c}) (\bar{\bm{u}}, \eta)$ is the unique solution to
    \[
    \begin{split}
        -\sigma \Delta \zeta + g \zeta &= -\dfrac{1}{2}\left(\partial_1 B(\bm{U}(\bm{c}),0)D \bar{\curlL}(\alpha \bm{U}(\bm{c}), 0, \bm{M}(0,\bm{c})) (\alpha \bar{\bm{u}}, \eta, D_\eta M_j(0, \bm{c})\eta)\ + \partial_2 B(\bm{U}(\bm{c}),0)\eta\right)|_{z=0}\\
        &= (c_1^2+c_2^2)\dfrac{\eta}{d} -c_1 \bar{v}_1 - c_2 \bar{v}_2.
    \end{split}
    \]
Thus, the linearised problem takes the form
\begin{equation}
\label{eq:linearization}
\begin{pmatrix} \bar{\bm{u}}\\ \eta \end{pmatrix}-\begin{pmatrix} \mathcal{A}_{11} & \mathcal{A}_{12} \\ \mathcal{A}_{21} & \mathcal{A}_{22} \end{pmatrix} \begin{pmatrix} \bar{\bm{u}}\\ \eta \end{pmatrix}=\begin{pmatrix} \bm{w} \\ f\end{pmatrix}.
\end{equation}

Let us next show that this can be reduced to an equation for $\eta$ under the non-resonance condition
\begin{equation}
\label{eq:nondegeneracy}
\sqrt{\alpha^2 -|\bm{k}|^2} \not \in \frac{\pi}{d}\mathbb Z_+ \quad \text{for all } \bm{k}\in \Lambda' \text{ such that } |\bm{k}| < |\alpha|,
\end{equation}
where $\Lambda'$ is the dual lattice of $\Lambda$;  see \eqref{eq:lattice}.
For this, we solve the first line in \eqref{eq:linearization} for $\bar{\bm{u}}$ and then plug this into the second equation to obtain a closed equation for $\eta$.
Indeed, the equation
\[
\bar{\bm{u}}-\mathcal A_{11} \bar{\bm{u}}=\tilde{\bm{w}}
\]
is equivalent to
\begin{alignat*}{2}
            \curl \bar{\bm{u}} -\alpha \bar{\bm{u}}&= \curl \tilde{\bm{w}}\quad && \text{in } \Omega^0, \\
            \nabla \cdot \bar{\bm{u}} &= 0\quad && \text{in } \Omega^0, \\
            \bar{u}_3 &= 0 \quad  && \text{on } \partial \Omega^0, \\
            \int_{\Omega_{00}^0} \bar{u}_j \,\de V  &= \int_{\Omega_{00}^0} \tilde{w}_j \,\de V   \quad  &&\text{for } j = 1,2.
\end{alignat*}
Moreover, the operator $I-\mathcal A_{11}$ has Fredholm index zero, being a compact perturbation of the identity operator. 
Under the condition \eqref{eq:nondegeneracy}, the kernel is trivial (see \cite[Lemma 2.2]{lokharu2020}), and thus 
there is an inverse $\curlT \tilde{\bm{w}}\coloneqq (I-\mathcal{A}_{11})^{-1}\tilde{\bm{w}}$, which solves the above system.
Hence, the system \eqref{eq:linearization} is equivalent to
\begin{equation}
\label{eq:reduction to surface}
\eta-(\mathcal{A}_{21} \curlT \mathcal A_{12} + \mathcal{A}_{22}) \eta=f+\mathcal{A}_{21} \curlT \bm{w}.
\end{equation}

We now prove that the operator $I-(\mathcal{A}_{21}\curlT \mathcal A_{12}+\mathcal{A}_{22})$ in the left-hand side of \eqref{eq:reduction to surface} is a Fourier multiplier operator with symbol 
\[
\tilde\rho(\bm{c},\bm{k})\coloneqq \frac{\rho(\bm{c},\bm{k})}{g+\sigma |\bm{k}|^2}, 
\]
in which
\[
\rho(\bm{c},\bm{k})\coloneqq
g+\sigma |\bm{k}|^2-\frac{(\bm{c}\cdot \bm{k})^2}{|\bm{k}|^2} \kappa(|\bm{k}|)+\alpha \frac{(\bm{c}\cdot \bm{k})(\bm{c}\cdot \bm{k}^\perp)}{|\bm{k}|^2},
\]
where   
\[
\kappa(|\bm{k}|) \coloneqq
\begin{cases}
\sqrt{\alpha^2-|\bm{k}|^2}\cot(\sqrt{\alpha^2-|\bm{k}|^2}d) & \text{if } |\alpha |>|\bm{k}|,\\ 
\frac{1}{d}  & \text{if } |\alpha |=|\bm{k}|, \\ 
\sqrt{|\bm{k}|^2-\alpha^2}\coth(\sqrt{|\bm{k}|^2-\alpha^2}d) & \text{if } |\alpha |<|\bm{k}|,\\ 
\end{cases}
\]
and $\bm{k}^\perp=(-k_2, k_1)$.
Indeed, $\bar{\bm{v}}=\curlT \mathcal A_{12}\eta$ is the solution to the system
\begin{alignat*}{2}
            \curl \bar{\bm{v}} -\alpha \bar{\bm{v}}&= \curl \tilde L(\eta)\quad && \text{in } \Omega^0, \\
            \nabla \cdot \bar{\bm{v}} &= 0\quad && \text{in } \Omega^0, \\
            \bar{v}_3 &= 0 \quad  && \text{on } \partial \Omega^0, \\
            \int_{\Omega_{00}^0} \bar{v}_j \,\de V  &= \int_{B_{00}} U_j(\bm{c})|_{z=0}\eta   \,\de x \, \de y   \quad  &&\text{for } j = 1,2.
\end{alignat*}
Setting $\bar{\bm{v}}=\bm{w}^\eta+\bm{v}$, with $\bm{w}^\eta(x,y,z)=\frac{\eta(x,y)}{d}((z+d)\bm{U}(\bm{c})(z))'-\nabla \eta(x,y) \cdot \bm{U}(\bm{c})(z) \frac{z+d}{d}\bm{e}_3$, we can transform this system into    
\begin{alignat*}{2}
            \curl \bm{v} -\alpha \bm{v}&= 0 \quad && \text{in } \Omega^0, \\
            \nabla \cdot \bm{v} &= 0\quad && \text{in } \Omega^0, \\
            v_3 &= c_1 \eta_x+c_2 \eta_y \quad  && \text{on } z=0, \\
            v_3 &= 0 \quad  && \text{on } z=-d, \\
            \int_{\Omega_{00}^0} v_j \,\de V  &= 0  \quad  &&\text{for } j = 1,2,
\end{alignat*}
and hence
\begin{align*}
(\mathcal A_{22}+\mathcal A_{21} \curlT \mathcal A_{12})\eta&=(g-\sigma \Delta)^{-1} \left((c_1^2+c_2^2)\dfrac{\eta}{d} -\bm{c}\cdot \mathcal A_{12} \eta-\bm{c}\cdot \mathcal A_{11} \curlT \mathcal A_{12}\eta\right)|_{z=0}\\
&=(g-\sigma \Delta)^{-1} \left((c_1^2+c_2^2)\dfrac{\eta}{d} -\bm{c}\cdot \curlT \mathcal A_{12}\eta\right)|_{z=0}\\
&=(g-\sigma \Delta)^{-1} \left((c_1^2+c_2^2)\dfrac{\eta}{d}-\bm{c}\cdot \bm{w}^\eta-\bm{c}\cdot \bm{v}\right)|_{z=0}\\
&=-(g-\sigma \Delta)^{-1} (\bm{c}\cdot \bm{v})|_{z=0},
\end{align*}
where we have used the identity $\mathcal A_{11}\curlT=\curlT-I$, and abused notation by writing $\bm{c}\cdot \bm{v}=c_1v_1+c_2v_2$ etc.
The claimed result now follows by using the identity
\[
 (c_1 \hat{v}_1+c_2\hat{v}_2)|_{z=0}=\left(-\frac{(\bm{c}\cdot \bm{k})^2}{|\bm{k}|^2} \kappa(|\bm{k}|)+\alpha \frac{ (\bm{c}\cdot \bm{k})(\bm{c}\cdot \bm{k}^\perp)}{|\bm{k}|^2} \right)\hat{\eta},
\]
see \cite[Section 3.1]{lokharu2020}.    
The kernel of $I-\mathcal A$ is thus spanned by $(\curlT\mathcal A_{12}\eta, \eta)$ for $\eta=\cos(\bm{k}\cdot \bm{x}')$ such that $\bm{k}$ solves the dispersion equation $\rho(\bm{c},\bm{k})=0$. We assume that at $\bm{c}=\bm{c}^*$, the generators $\bm{k}_1$, $\bm{k}_2$ (and $-\bm{k}_1$, $-\bm{k}_2$) of the dual lattice $\Lambda'$ are the only solutions, and let $\eta_1=\cos(\bm{k}_1 \cdot \bm{x}')$, $\eta_2=\cos(\bm{k}_2\cdot \bm{x}')$. This assumption is expected to hold for generic parameter values as discussed in \cite[Sec.~3.1]{lokharu2020}. Thus, the kernel is two-dimensional, and the range has codimension two. The kernel is spanned by $\Phi_i=(\curlT \mathcal A_{12}\eta_i, \eta_i)$, $i=1,2$, and the range is given  by $(\bm{w}, f)$ such that  
$f+\mathcal A_{21} \curlT \bm{w}$ is orthogonal to  $\eta_1$ and $\eta_2$ in $L^2_\mathrm{per}(\R^2)$. Therefore, assumption \ref{H2} is satisfied.

It remains to establish the assumptions \ref{H3}--\ref{H6}. For this, we let
\[
\curlX_i=\spn\{ \Phi_i\} 
\]
and
\[
\curlY_i=\spn\{ (0, \eta_i)\}
\]
and decompose $\curlX= \curlX_1 \bigoplus \curlX_2 \bigoplus \tilde{\curlX}$, 
where $\tilde{\curlX}=\{ (\bar{\bm{u}}, \eta)\in \curlX : (\eta, \eta_i)_{L^2_\mathrm{per}(\R^2)}=0, \, i=1,2\}$ and $\curlY=\curlY_1\oplus \curlY_2\oplus \tilde \curlY$, where $\tilde \curlY=\ran(I-\mathcal A)=\{(\bm{w},f)\in \curlX : f+\mathcal A_{21} \curlT \bm{w} \perp \eta_i, i=1, 2\}$.
We now define the projections onto  $\curlX_i$ and $\curlY_i$, respectively.
The projection $Q_i : \curlX \to \curlX_i$ is given by
\[
Q_i(\bar{\bm{u}}, \eta)=q_i(\curlT\mathcal A_{12}\eta_i, \eta_i)
\]
where
\[
q_i=q_i(\bar{\bm{u}}, \eta)=\frac{(\eta, \eta_i)_{L^2_\mathrm{per}(\R^2)}}{\|\eta_i\|_{L^2_\mathrm{per}(\R^2)}^2}.
\]
Note that $Q_i$ is not the standard orthogonal projection onto the kernel vector $\Phi_i = (\curlT \curlA_{12} \eta_i,\eta_i)$. The two main advantages of this choice of projection are that it allows for a more direct comparison with the local bifurcation result in \cite{lokharu2020} and provides a natural interpretation of the parameterisation in the global bifurcation result, as discussed in Remark \ref{rem:interpretation-of-hat-s}. 

The projection $P_i : \curlY \to \curlY_i$ is given by
\[
P_i (\bm{w},f)=p_i(0, \eta_i)
\]
where
\[
p_i=p_i(\bm{w},f)=\frac{(f+\mathcal A_{21}\curlT \bm{w}, \eta_i)_{L^2_\mathrm{per}(\R^2)}}{\|\eta_i\|_{L^2_\mathrm{per}(\R^2)}^2}.
\]
Therefore, we have \ref{H3} and \ref{H4}. Note also for later use that
\[
\tilde \rho(\bm{c},\bm{k}_i)=\frac{((I-(\mathcal A_{21}\curlT \mathcal A_{12}+\mathcal A_{22}))\eta_i, \eta_i)_{L^2_\mathrm{per}(\R^2)}}{\|\eta_i\|_{L^2_\mathrm{per}(\R^2)}^2}.
\]

To check the transversality condition \ref{H5}, we must compute 
\[
-P_i\left(\partial_{c_j} \mathcal A(\bm{c}^*)\begin{pmatrix} \curlT \mathcal A_{12} \eta_i \\ \eta_i\end{pmatrix}\right), \quad i, j=1,2.
\]
Noting that $\mathcal A_{11}$ is independent of $\bm{c}$, we find that
\[
\partial_{c_j} \mathcal A(\bm{c}^*)\begin{pmatrix} \curlT \mathcal A_{12} \eta_i \\ \eta_i\end{pmatrix}
=\begin{pmatrix} \partial_{c_j} \mathcal A_{12}\eta_i \\ \partial_{c_j}\mathcal A_{21} \curlT \mathcal A_{12} \eta_i + \partial_{c_j}\mathcal A_{22}\eta_i\end{pmatrix}
\]
and thus the corresponding projection coefficient $p_i$ is given by
\[
\begin{split}
-\frac{(\partial_{c_j}\mathcal A_{21} \curlT \mathcal A_{12} \eta_i + \partial_{c_j}\mathcal A_{22}\eta_i+\mathcal A_{21}\curlT \partial_{c_j} \mathcal A_{12}\eta_i, \eta_i)_{L^2_\mathrm{per}(\R^2)}}
{\|\eta_i\|_{L^2_\mathrm{per}(\R^2)}^2}
&=-\partial_{c_j} \frac{((\mathcal A_{21}\curlT \mathcal A_{12}+\mathcal A_{22})\eta_i, \eta_i)_{L^2_\mathrm{per}(\R^2)}}{\|\eta_i\|_{L^2_\mathrm{per}(\R^2)}^2}\\
&=\partial_{c_j}\tilde \rho(\bm{c}^*,\bm{k}_i)\\[1mm]
&=\frac{\partial_{c_j} \rho(\bm{c}^*,\bm{k}_i)}{g+\sigma|\bm{k}_i|^2},
\end{split}
\]
where we used in the first equality that since $\curlA_{11}$ is independent of $\bm{c}$, so is $\curlT = (I - \curlA_{11})^{-1}$. It follows that the transversality condition \ref{H5} in Theorem \ref{thm:seth} is equivalent to the transversality condition 
\begin{align}\label{tranv_cond_small_c}
    \nabla_{\bm{c}}\rho(\bm{c}^*, \bm{k}_1) \quad \text{and} \quad \nabla_{\bm{c}}\rho(\bm{c}^*, \bm{k}_2) \; \text{are not parallel} 
\end{align}
in \cite{lokharu2020}.
Finally, assumption \ref{H6} is satisfied by taking $\tilde{\curlX}_i$ and $\tilde{\curlY}_i$ to be the subspaces of $\tilde{\curlX}$ and $\tilde{\curlY}$, respectively, consisting of functions which are constant in the direction $\bm{\lambda}_i$ (see the proofs of \cite[Theorem 4.1]{lokharu2020} and \cite[Theorem 5.3]{seth2024}).
We have thus recovered the following result from \cite{lokharu2020} in our framework.

\begin{theorem}
\label{thm:local bifurcation}
    Let $\alpha \in \R$, $\sigma>0$ and the depth $d>0$ be given, as well as a laminar
flow $\bm{U}(\bm{c})$. Furthermore, let $\Lambda'$ be the dual lattice of $\Lambda$, see \eqref{eq:lattice}, generated by the linearly independent vectors $\bm{k}_1, \bm{k}_2 \in \Lambda'$. Assume that    
    \begin{enumerate}
        \item the non-resonance condition \eqref{eq:nondegeneracy} holds; also assume that  $\alpha d \notin 2\pi\mathbb{Z}$;
        \item within the lattice $\Lambda'$ the dispersion equation $\rho(\bm{c},\bm{k})=0$ with $c_1=c_1^*, c_2=c_2^*$ has exactly four roots $\pm \bm{k}_1, \pm \bm{k}_2$;
        \item the  transversality condition \eqref{tranv_cond_small_c} holds.
    \end{enumerate}
    Then there exist an $\varepsilon>0$ and real-analytic functions $(\Phi(\bm{s}), \bm{c}(\bm{s}))$ from $B_\epsilon(0; \R^2)$ to $\mathcal U \times \R^2$ such that $\curlF(\Phi(\bm{s}),\bm{c}(\bm{s}))=0$ for all $\bm{s}\in B_\epsilon(0; \R^2)$ and
    \begin{align*}
        \Phi(\bm{s})=(\bm{U}(\bm{c}),0)+s_1\Phi_1  +s_2 \Phi_2+O(|\bm{s}|^2), \quad \bm{c}(\bm{s})=\bm{c}^*+O(|\bm{s}|^2).
    \end{align*}
    In particular, $\Phi = (\bar{\bm{u}},\eta)$ with
    \begin{align*}
        \eta(\bm{x}')=s_1\cos(\bm{k}_1\cdot \bm{x}')+s_2\cos(\bm{k}_2 \cdot \bm{x}')+O(|\bm{s}|^2).
    \end{align*}
\end{theorem}

\begin{remark}
From the abstract result Theorem \ref{thm:seth}, it only follows that $\bm{c}=\bm{c}^* + O(|\bm{s}|)$. The fact that the linear terms in the expansion vanish can be proved as in \cite{lokharu2020}.
\end{remark}

Preparing for the global bifurcation analysis, let us fix a unit direction $\hat{\bm{s}} \in S^1\subset \R^2$, with $\hat{s}_1\hat{s}_2\ne 0$, in the parameter space and consider the operator
\begin{equation}\label{eq:curlG}
    \curlG(\tilde \Phi, \bm{c}, r)=\mathcal F((\bm{U}(\bm{c}),0)+r\hat{s}_1 \Phi_1+r\hat{s}_2 \Phi_2+\tilde \Phi, \bm{c}).
\end{equation}
Moreover, let $\tilde \Phi(\bm{s})=\Phi(\bm{s})-s_1\Phi_1- s_2 \Phi_2-(\bm{U}(\bm{c}(\bm{s})),0)$ with $\Phi(\bm{s})$ and $\bm{c}(\bm{s})$ as in Theorem \ref{thm:local bifurcation}.
The following result is a direct consequence of Proposition \ref{prop-app:isomorphism}.
\begin{proposition}\label{prop:isomorphism}
Under the assumptions of the local bifurcation result, Theorem \ref{thm:local bifurcation}, there exists an $r_0=r_0(\hat{\bm{s}})>0$ such that 
\[
D_{(\tilde \Phi, \bm{c})} \curlG(\tilde \Phi(r\hat{\bm{s}}), \bm{c}(r\hat{\bm{s}}),r) \colon \tilde \curlX \times \R^2 \to \curlX
\]
is an isomorphism for each $r\in (0, r_0)$.
\end{proposition}

\begin{remark}\label{rem:choice-of-hat-s}
    Note that we cannot choose $\hat{\bm{s}} \in S^1$ with $\hat{s}_1\hat{s}_2=0$ in the above proposition. However, this choice corresponds to waves which are $2\tfrac{1}{2}$-dimensional close to the bifurcation point; see \cite[Proof of Theorem 4.1]{lokharu2020}. Therefore, these branches are of less interest in the present analysis. On the other hand, the choice $\hat{s}_1 \hat{s}_2 \neq 0$ guarantees that the bifurcating solutions correspond to genuinely three-dimensional waves; see also Remark \ref{rem:interpretation-of-hat-s}.
\end{remark}

\section{Global bifurcation}\label{sec:global-bif}

The goal of this section is to construct a global curve of solutions to the equation $\curlG(\tilde \Phi, \bm{c}, r)=0$. We do this by using analytic global bifurcation theory, see, for example, \cite{buffoni2003}, in the following non-standard way. As the reformulation of the bifurcation problem suggests, we treat the additional parameter $r$ as the new bifurcation parameter and $(\tilde{\Phi},\bm{c})$ as new variables. Recall in particular that $r$  is a scalar parameter, and we therefore have a one-dimensional parameter space; see Figure \ref{fig:schematic-global-bifurcation} for a schematic depiction. However, at the bifurcation point $r=0$, the linearisation $D_{(\tilde{\Phi},\bm{c})}\curlG(0,\bm{c}^*,0)$ still has a two-dimensional kernel. Hence, we cannot apply the analytic global bifurcation theorem \cite[Theorem 9.1.1]{buffoni2003} directly. Nevertheless, it turns out that $D_{(\tilde{\Phi},\bm{c})}\curlG(\tilde{\Phi}(r\hat{\bm{s}}),\bm{c}(r\hat{\bm{s}}),r)$ is an isomorphism for $r$ in a small neighbourhood of $0$ (see Proposition \ref{prop:isomorphism}), which is sufficient to apply a variation of the global bifurcation theorem recorded in Appendix \ref{app:globalBifurcation}. 

Before we state the resulting theorem, we introduce some additional notation. We define for any $\delta \in (0, d)$ the sets
\begin{equation*}
    \begin{split}
        \tilde{\curlU}_\delta(\hat{\bm{s}}) &\coloneqq \{(\tilde{\Phi},\bm{c},r) \in \tilde{\curlX} \times \R^2 \times \R \,:\, (\bm{U}(\bm{c}),0) + r\hat{s}_1 \Phi_1 + r\hat{s}_2 \Phi_2 + \tilde{\Phi} \in \curlU_\delta\}, \\
        \tilde{\curlU}(\hat{\bm{s}}) &\coloneqq \{(\tilde{\Phi},\bm{c},r) \in \tilde{\curlX} \times \R^2 \times \R \,:\, (\bm{U}(\bm{c}),0) + r\hat{s}_1 \Phi_1 + r\hat{s}_2 \Phi_2 + \tilde{\Phi} \in \curlU\},
    \end{split}
\end{equation*}
where we recall the definition of $\curlU$ and $\curlU_\delta$ given in \eqref{eq:U} and \eqref{eq:Udelta}, respectively.
Additionally, we denote the set of solutions to $\curlG(\tilde{\Phi},\bm{c},r)=0$ in $\tilde{\curlU}_\delta$ or $\tilde{\curlU}$ by
\begin{equation*}
    \begin{split}
        \mathcal{S}_\delta(\hat{\bm{s}}) &\coloneqq \{(\tilde{\Phi},\bm{c},r) \in \tilde{\curlU}_\delta(\hat{\bm{s}}) \,:\, \curlG(\tilde{\Phi},\bm{c},r)=0\}, \\
        \mathcal{S}(\hat{\bm{s}}) &\coloneqq \{(\tilde{\Phi},\bm{c},r) \in \tilde{\curlU}(\hat{\bm{s}}) \,:\, \curlG(\tilde{\Phi},\bm{c},r)=0\}.
    \end{split}
\end{equation*}
As in Section \ref{sec:final-reformulation}, we note that $\tilde{\curlU}(\hat{\bm{s}}) = \bigcup_{0<\delta < d} \tilde{\curlU}_\delta(\hat{\bm{s}})$ and $\curlS(\hat{\bm{s}}) = \bigcup_{0<\delta < d} \curlS_\delta(\hat{\bm{s}})$. Using these notations, we can formulate our first global bifurcation result, which extends the bifurcation curve along a fixed direction $\hat{\bm{s}} \in S^1$ in parameter space.

\begin{theorem}\label{thm:global-bifurcation}
    Under the assumptions of Theorem \ref{thm:local bifurcation} the following holds. Fix $\hat{\bm{s}} \in S^1$ with $\hat{s}_1\hat{s}_2 \neq 0$. Then, there exists a continuous global bifurcation curve
    \begin{equation*}
        \mathfrak{R}(\hat{\bm{s}}) = \{(\tilde{\Phi}(t),\bm{c}(t),r(t)) \in \tilde{\curlU}(\hat{\bm{s}}) \,:\, t \in [0,\infty)\} \subset \mathcal{S}(\hat{\bm{s}})
    \end{equation*}
    of solutions to $\curlG(\tilde{\Phi},\bm{c},r) = 0$. In particular, $\mathfrak{R}(\hat{\bm{s}})$ extends the local bifurcation curve in Theorem \ref{thm:local bifurcation} in the sense that
    \begin{equation*}
        \{(\Phi(r\hat{\bm{s}}) - r\hat{s}_1 \Phi_1 - r\hat{s}_2 \Phi_2 - (\bm{U}(\bm{c}(r\hat{\bm{s}})),0),\bm{c}(r\hat{\bm{s}}),r) \,:\, r \in [0,\varepsilon)\} \subset \mathfrak{R}(\hat{\bm{s}}),
    \end{equation*}
    where $\Phi(\bm{s})$ and $\bm{c}(\bm{s})$ are given in Theorem \ref{thm:local bifurcation} and $\Phi_i$, $i = 1,2$ span the kernel of $D_\Phi \curlF((\bm{U}(\bm{c}^*),0),\bm{c}^*)$.
    Additionally, $\mathfrak{R}(\hat{\bm{s}})$ has a local analytic re-parameterisation and at least one of the following alternatives occurs:
    \begin{enumerate}[label=(a\arabic*),ref=(a\arabic*)]
        \item\label{global-alternative-1} the curve $\mathfrak{R}(\hat{\bm{s}}) = \{(\tilde{\Phi}(t),\bm{c}(t),r(t)) \,:\, t \in [0,\infty)\}$ eventually leaves every set 
        \begin{equation*}
            \tilde{\curlU}_\delta(\hat{\bm{s}}) \cap \{\|\tilde{\Phi}\|_\curlX + |\bm{c}| + |r| < \delta^{-1}\},
        \end{equation*}
        that is, for every $\delta \in (0,d)$ there exists a $t_\delta > 0$ such that $(\tilde{\Phi}(t),\bm{c}(t),r(t)) \notin \tilde{\curlU}_\delta(\hat{\bm{s}}) \cap \{\|\tilde{\Phi}\|_\curlX + |\bm{c}| + |r| < \delta^{-1}\}$ for all $t > t_\delta$.
        \item\label{global-alternative-2} $\mathfrak{R}(\hat{\bm{s}})$ is a closed loop, that is, there exists a $T > 0$ such that $\mathfrak{R}(\hat{\bm{s}}) = \{(\tilde{\Phi}(t),\bm{c}(t),r(t)) \,:\, t \in [0,T]\}$ and $(\tilde{\Phi}(T),\bm{c}(T),r(T)) = (0,0,0)$.
    \end{enumerate}
\end{theorem}
\begin{proof}
    We check that the assumptions of Theorem \ref{thm-app:global-bifurcation} are satisfied. Note that the notation here translates to the one used in Theorem \ref{thm-app:global-bifurcation} via $\xi = (\tilde{\Phi},\bm{c})$ and $U = \tilde{\curlU}(\hat{\bm{s}})$. First, since any fixed $(\tilde{\Phi},\bm{c},r) \in \tilde{\curlU}(\hat{\bm{s}})$ lies in some $\tilde{\curlU}_\delta(\hat{\bm{s}})$, we obtain from Corollary \ref{Fredholm_op} that $D_{(\tilde{\Phi},\bm{c})}\curlG(\tilde{\Phi},\bm{c},r) \colon \tilde{\curlX} \times \R^2 \to \curlX$ is of the form `identity plus compact' and thus a Fredholm operator of index zero, if we identify the complement $\curlX_1 \oplus \curlX_2$ of $\tilde{\curlX}$ in $\curlX$ with $\R^2$. Second, we recall that Theorem \ref{thm:local bifurcation} establishes an analytic local bifurcation curve $\curlR$ given by \begin{equation*}
        \curlR = \{(\tilde{\Phi}(r(t)\hat{\bm{s}}),\bm{c}(r(t)\hat{\bm{s}}),r(t)) \in \tilde{\curlU}_\delta(\hat{\bm{s}}) \,:\, t \in [0,\epsilon)\}
    \end{equation*}
    with $r(t) = t$. Third, Proposition \ref{prop:isomorphism} shows that $D_{(\tilde{\Phi},\bm{c})} \curlG(\tilde{\Phi}(r\hat{\bm{s}}),\bm{c}(r\hat{\bm{s}}),r) \colon \tilde{\curlX} \times \R^2 \rightarrow \curlX$ is an isomorphism for $0 < r \ll 1$. Finally, we define for any $j \in \N$ with $j>1/d$ that $\mathcal{Q}_j = \tilde{\curlU}_{1/j}(\hat{\bm{s}}) \cap \{\|\tilde{\Phi}\|_\curlX + |\bm{c}| + |r| < j\}$, which are bounded and open subsets of $\tilde{\curlU}(\hat{\bm{s}})$ and satisfy $\tilde{\curlU}(\hat{\bm{s}}) = \bigcup_{j} \mathcal{Q}_j$. We now check that $\curlS \cap \mathcal{Q}_j$ is compact for all $j$. Assuming that $(\tilde{\Phi},\bm{c},r) \in \curlS \cap \mathcal{Q}_j$, we can apply the estimates \eqref{eq:schauder-estimate} and \eqref{eq:schauder-est-dyn-bc} established in Corollary \ref{cor:schauder} and Theorem \ref{thm:dynBoundaryCond}, respectively. By compact embedding of $\Holderspace{k+1}$ into $\Holderspace{k}$, this shows that $\curlS \cap \mathcal{Q}_j$ is compact. Therefore, we have established all assumptions of Theorem \ref{thm-app:global-bifurcation} and the stated result follows. Note in particular that since the $\tilde{\curlU}_\delta$ are nested sets, we automatically obtain the alternative that $(\tilde{\Phi}(t),\bm{c}(t),r(t))$ leaves $\tilde{\curlU}_\delta(\hat{\bm{s}})$ for every $\delta \in (0,d)$.
\end{proof}

\begin{remark}\label{rem:interpretation-of-hat-s}
   Due to the choice of projections in the local bifurcation theory, we can give a concrete interpretation to the parameter $\bm{s} = r\hat{\bm{s}}$.
    Recall that any solution on the global branch $\mathfrak{R}(\hat{\bm{s}})$ is of the form
    \begin{equation*}
        \Phi = (\bar{\bm{u}},\eta) = (\bm{U}(\bm{c}),0) + r\hat{s}_1 \Phi_1 + r\hat{s}_2 \Phi_2 + \tilde{\Phi},
        \qquad \tilde{\Phi} = (\tilde{\bm{u}},\tilde{\eta}) \in \tilde{\curlX}.
    \end{equation*}
    Using that $\Phi_i = (\curlT\mathcal{A}_{12}\eta_i,\eta_i)$ has surface
    component $\eta_i = \cos(\bm{k}_i \cdot \bm{x}')$, and that $(\tilde{\eta},\eta_i)_{L^2_\mathrm{per}(\R^2)} = 0$
    by the definition of $\tilde{\curlX}$, the surface component thus reads as
    \begin{equation*}
        \eta(\bm{x}') = r\hat{s}_1 \cos(\bm{k}_1 \cdot \bm{x}') + r\hat{s}_2 \cos(\bm{k}_2 \cdot \bm{x}') + \tilde{\eta}(\bm{x}'),
        \qquad (\tilde{\eta}, \eta_i)_{L^2_\mathrm{per}(\R^2)} = 0, \quad i = 1,2.
    \end{equation*}
    Therefore, we find that $r \hat{\bm{s}}_j$ are the Fourier coefficients of the cosine modes $\cos(\bm{k}_j \cdot \bm{x}')$ for $j = 1,2$. In particular, the condition $\hat{s}_1\hat{s}_2 \neq 0$ guarantees that the solutions on the global bifurcation branch are genuinely three-dimensional surface waves as long as $r(t) \neq 0$.
\end{remark}

\begin{remark}\label{rem:two-dimensional solution set}
We have obtained a curve $\mathfrak{R}(\hat{\bm s})$ of solutions parameterised by $t$ for each fixed $\hat{\bm s} \in S^1$ with $\hat s_1 \hat s_2\ne 0$. A natural question is how these curves depend on $\hat{\bm s}$ and whether we obtain some form of two-dimensional surface or variety by varying $t$ and $\hat{\bm s}$. A simple consequence of the implicit function theorem is that at each point where $D_{(\tilde{\Phi},\bm{c})}\curlG(\tilde{\Phi},\bm{c},r)$ is an isomorphism, the curve $\mathfrak{R}(\hat{\bm s})$ is locally part of a two-dimensional analytic surface obtained by varying $\hat{\bm s}$. However, this surface is not necessarily made up of different curves $\mathfrak{R}(\hat{\bm s})$. Indeed, these curves might start to diverge at points where  $D_{(\tilde{\Phi},\bm{c})}\curlG(\tilde{\Phi},\bm{c},r)$ fails to be invertible, as the simple example $xy=s$ shows. For $s=0$, the solution set is the union of the $x$-axis and the $y$-axis, while for $s\ne 0$ it has the solution $y=s/x$, which switches between following the $x$-axis and the $y$-axis as $x\to \pm \infty$ and $x\to 0^\pm$. We make no attempt at describing the local variety of solutions  in a way which  involves $\hat{\bm s}$ near points where $\ker D_{(\tilde{\Phi},\bm{c})}\curlG(\tilde{\Phi},\bm{c},r)$ is non-trivial.  
\end{remark}

We now transfer the global bifurcation result in Theorem \ref{thm:global-bifurcation} to a global bifurcation result for doubly periodic waves in the original Beltrami-flow problem \eqref{eq:formulation}. In particular, we provide a more refined set of alternatives. That is, we prove that the alternative \ref{global-alternative-1} can be interpreted as the blow-up of the surface gradient $\nabla \eta$ in $C^{0,\gamma}_{\mathrm{per}}$, the blow-up of the wave velocity $\bm{c}$, or that the surface eventually intersects the flat bed. Specifically, we can exclude an isolated blow-up of the velocity field $\bm{u}$ along the global bifurcation branch.
To do this, we exploit that the Bernoulli equation prevents the $L^2$-blow-up of the velocity field at the surface provided that the surface profile is $C^1$-bounded and the wave velocity $\bm{c}$ is bounded. To use this observation, we thus need to bound $\bar{\bm{u}}$ in terms of its surface values. 
Such a bound is provided in Lemma \ref{est-of-u-in-eta} below, which relies on the following  low-order  estimate for solutions to \eqref{eq:formulation-flat}
\begin{equation}
        \Holdernormper{\bar{\bm{u}}}{0}{\gamma}
        \leq C(\delta^{-1}, R) (\|\bar{\bm{u}}\|_{C^0_\mathrm{per}} + \|\bar{\bm{w}}\|_{C^0_\mathrm{per}}), \quad \eta \in \curlB_{\delta, R}^{1,\gamma}, \label{est-u_0_gamma}
\end{equation} 
proved in Appendix \ref{sec:Schauder estimates}; see Theorem \ref{thm-app:schauder}.

\begin{lemma}\label{est-of-u-in-eta} Let $\eta \in \curlB^{1,\gamma}_{R^{-1},R}$ and $\bar{\bm{u}} \in C^{1}_\mathrm{per}$ with $\nabla \cdot \bar{\bm{u}} = 0$, $\bar{\bm{u}} \cdot \bm{n} = 0$ on $\partial \Omega^0$. Then, the estimate
\[
\|\bar{\bm{u}}\|_{C^{0,\gamma}_{\text{per}}} \le C(R) \left( \|\bar{\bm{u}}|_{z=0}\|_{L^2} + \|\nabla \times (A(\eta)\bar{\bm{u}}) - \alpha \bar{\bm{u}}\|_{C^{0}_\mathrm{per}}  \right)
\]  
holds.
\end{lemma}

\begin{proof}

We argue by contradiction and assume there are sequences $\eta_n \in \curlB^{1,\gamma}_{R^{-1},R}$  and $\bar{\bm{u}}_n \in C^{1}_\mathrm{per}$ with $\|\bar{\bm{u}}_n\|_{C^{0,\gamma}_{\text{per}}} = 1$ and $\nabla \cdot \bar{\bm{u}}_n = 0$, $\bar{\bm{u}}_n \cdot \bm{n} = 0$ on $\partial \Omega^0$ such that
\[
\frac{1}{n} = \frac{1}{n} \|\bar{\bm{u}}_n\|_{C^{0,\gamma}_{\text{per}}} \ge \|\bar{\bm{u}}_n|_{z=0}\|_{L^2} + \|\nabla \times (A(\eta_n)\bar{\bm{u}}_n) - \alpha \bar{\bm{u}}_n\|_{C^{0}_{\text{per}}} 
\]
for all $n \in \N$.
Hence, the right-hand side tends to 0 as $n \to \infty$. On the other hand, by compact embedding of H\"older spaces, there exist $\bar{\bm{u}} \in C^{0,\gamma}$ and $\eta \in C^{1,\gamma}$ so that, after restricting to an appropriate subsequence,  $\eta_n \to \eta$ in $C^{1,\gamma'}$ and $\bar{\bm{u}}_n \to \bar{\bm{u}}$ in $C^{0,\gamma'}$ for any $\gamma' < \gamma$. 
Specifically, we have that $\bar{\bm{u}}_n$ converges uniformly to $\bar{\bm{u}}$. We then have
\begin{subequations}
    \begin{alignat*}{2}
            \nabla \times (A(\eta)\bar{\bm{u}}) &=\alpha\bar{\bm{u}} \quad &&\text{in } \Omega^0,\\
            \nabla \cdot \bar{\bm{u}} &= 0 &&\text{in } \Omega^0,\\
            \bar{\bm{u}}  &= 0 &&\text{on } z=0, \\
            \bar{\bm{u}} \cdot \bm{n} &= 0 &&\text{on } z=-d,
    \end{alignat*}
\end{subequations}
where the first two lines are interpreted weakly.
We claim that this implies that $\bar{\bm{u}} = 0$. Indeed, going back to the original variables, we get
\begin{subequations}
    \begin{alignat*}{2}
            \nabla \times \bm{u} &=\alpha \bm{u} \quad &&\text{in } \Omega^\eta,\\
            \nabla \cdot \bm{u} &= 0 &&\text{in } \Omega^\eta,\\
            \bm{u}  &= 0 &&\text{on } z=\eta, \\
            \bm{u} \cdot \bm{n} &= 0 &&\text{on } z=0.
    \end{alignat*}
\end{subequations}
By applying another curl and using that $\bm u$ is divergence free, we find that $-\Delta \bm{u}=\alpha^2 \bm{u}$ with $\bm{u}=0$ on the surface. Since $\eta\in C^{1,\gamma}$, it follows that $\bm{u}$ is $C^{1,\gamma}$ up to the surface.
On the other hand, $\partial_3 u_3 = -\partial_1 u_1 - \partial_2 u_2$. Along the surface, we also have
\[
\partial_i u_j + \partial_3 u_j \partial_i \eta = 0,\quad i=1,2,\ j=1,2,3,
\]
as well as
\[
\partial_2 u_3 - \partial_3 u_2 = \partial_3 u_1 - \partial_1 u_3 = \partial_1 u_2 - \partial_2 u_1 = 0
\]
by using that $\bm{u}$ vanishes at the surface.
Combining these relations, we get
\[
\partial_1 u_1 = -\partial_3 u_1 \partial_1 \eta = -\partial_1 u_3 \partial_1 \eta,
\]
as well as
\[
\partial_2 u_2 = -\partial_3 u_2 \partial_2 \eta = -\partial_2 u_3 \partial_2 \eta.
\]
Thus,
\[
\partial_3 u_3 = -\partial_1 u_1 - \partial_2 u_2 = \partial_1 u_3 \partial_1 \eta + \partial_2 u_3 \partial_2 \eta,
\]
in other words,
\[
\partial_n u_3 = 0
\]
on the surface. We claim that this implies that $u_3 \equiv 0$. Indeed, the function $w(\bm{x},x_4) = e^{\alpha x_4} u_3$ is harmonic in the extended domain $\tilde{\Omega}^\eta = \{(\bm{x},x_4): \bm{x} \in \Omega^\eta\} \subset \mathbb{R}^4$ with $w = 0$ and $\partial_n w = 0$ on the boundary component $\{x_3 = \eta\}$. Extending $w$ by zero across the boundary, we find that $w$ is weakly harmonic and therefore strongly harmonic and real analytic across the boundary. Since $w$ vanishes in an open set, it therefore vanishes identically. 
Once we know that $u_3 \equiv 0$, we get $\partial_3 u_1 = \alpha u_2$ and $\partial_3 u_2 = -\alpha u_1$. Since also $(u_1,u_2) = (0,0)$ at the surface, we get by solving this system of ODEs for fixed $(x_1,x_2)$ that $(u_1,u_2) \equiv 0$ in $\Omega^\eta$. Returning to the flattened domain, we thus obtain $\bar{\bm{u}} = 0$, which yields that $\|\bar{\bm{u}}_n\|_{C^0_\mathrm{per}} \to 0$ as $n \to \infty$.
On the other hand, we have the uniform Schauder estimate \eqref{est-u_0_gamma} 
\[
1 = \|\bar{\bm{u}}_n\|_{C^{0,\gamma}_{\text{per}}} \leq C(R) \left( \|\bar{\bm{u}}_n\|_{C^{0}_{\text{per}}} + \|\nabla \times (A(\eta_n)\bar{\bm{u}}_n) - \alpha \bar{\bm{u}}_n\|_{C^{0}_{\text{per}}}  \right) \rightarrow 0.
\]
This is a contradiction, hence the estimate must hold.
\end{proof}

We now have all required tools to obtain our main result, the global bifurcation of doubly periodic waves on Beltrami flows, which specifically proves the existence of genuinely three-dimensional, doubly periodic waves outside of a perturbative regime close to the laminar flows.

\begin{theorem}\label{thm:global-bif-Beltrami}
    Let $k \geq 0$ and let $\Sigma$ be the set of doubly periodic solutions $(\bm{u},\eta,\bm{c}) \in C^{k+1,\gamma}_\mathrm{per}(\Omega^\eta;\R^3) \times C^{k+3,\gamma}_{\mathrm{per}}(\R^2;\R) \times \R^2$ to \eqref{eq:formulation}, equipped with the $C^{k+1, \gamma}_\mathrm{per}\times C^{k+2,\gamma}_\mathrm{per}\times \R^2$ topology. Then, under the assumptions of Theorem \ref{thm:local bifurcation}, the following holds. There is a connected subset $\mathcal{D}$ of $\Sigma$, which contains $(\bm{U}(\bm{c}^*),0,\bm{c}^*)$ and is the union of curves $\mathcal{D}(\hat{\bm{s}}) \coloneqq \{(\bm{u}_{\hat{\bm{s}}}(t),\eta_{\hat{\bm{s}}}(t),\bm{c}_{\hat{\bm{s}}}(t)) \in \mathcal{D} : t \in [0,\infty)\}$ with $\hat{\bm{s}} \in S^1 $ satisfying $\hat{s}_1\hat{s}_2 \neq 0$.
    Each curve $\mathcal{D}(\hat{\bm{s}})$ admits a local analytic re-parameterisation and at least one of the following alternatives occurs:
    \begin{enumerate}
        \item\label{alt:D-unbounded} $\mathcal{D}(\hat{\bm{s}})$ is unbounded in the sense that there exists a sequence $(\bm{u}_n,\eta_n,\bm{c}_n)_{n \in \N}$ such that either
        \begin{enumerate}
            \item\label{alt:blow-up-eta} $\|\nabla \eta_n\|_{C^{0,\gamma}_\mathrm{per}} \rightarrow \infty$, or 
            \item\label{alt:blow-up-c} $|\bm{c}_n| \rightarrow \infty$
        \end{enumerate}
        as $n \rightarrow \infty$;
        \item\label{alt:D-selfintersect} $\mathcal{D}(\hat{\bm{s}})$ contains a sequence $(\bm{u}_n,\eta_n,\bm{c}_n)_{n \in \N}$ such that $\min \eta_n \rightarrow -d$, that is, an intersection of the surface profile with the flat bed occurs;
        \item\label{alt:D-loop} $\mathcal{D}(\hat{\bm{s}})$ is a closed loop, that is, there exists a $T > 0$ such that $(\bm{u}(T),\eta(T),\bm{c}(T)) = (\bm{U}(\bm{c}^*),0,\bm{c}^*)$.
    \end{enumerate}
    Finally, the solutions on the curve $\mathcal{D}(\hat{s})$ are genuinely three-dimensional, doubly periodic waves, except for a discrete set of $t \in [0,\infty)$ with $r(t) = 0$.
\end{theorem}
\begin{remark}
    We specifically point out that the blow-up of the fluid velocity field $\bm{u}$ cannot occur as an isolated alternative. 
\end{remark}
\begin{proof}
    Fix $\hat{\bm{s}} \in S^1$ with $\hat{s}_1 \hat{s}_2 \neq 0$. We recall that any solution $(\tilde{\Phi},\bm{c},r) \in \tilde{\mathcal{U}}(\hat{\bm{s}})$ to $\curlG(\tilde{\Phi},\bm{c},r) = 0$ yields a solution to $\mathcal{F}(\Phi,\bm{c}) = 0$ via $\Phi = (\bar{\bm{u}},\eta) = (\bm{U}(\bm{c}),0) + r\hat{s}_1 \Phi_1 + r\hat{s}_2 \Phi_2 + \tilde{\Phi}$, which gives a solution $(\bar{\bm{u}},\eta,\bm{c}) \in C^{k,\gamma}_\mathrm{per}(\overline{\Omega^0};\R^3) \times C^{k+2,\gamma}_\mathrm{per}(\R^2;\R) \times \R^2$ to the flattened system \eqref{eq:formulation-flat} and \eqref{eq:dyn-bc-flattened}. If $\eta > -d$, we can then apply the inverse flattening transform, see Lemma \ref{lem:bound-flattening-trafo}, to obtain a solution to the Beltrami-flow problem \eqref{eq:formulation}, which lies in $C^{k,\gamma}_\mathrm{per}(\overline{\Omega^\eta};\R^3) \times C^{k+2,\gamma}_{\mathrm{per}}(\R^2;\R) \times \R^2$. Therefore, the global bifurcation curve $\mathfrak{R}(\hat{\bm{s}})$ obtained in Theorem \ref{thm:global-bifurcation} directly translates into a curve $\mathcal{D}(\hat{\bm{s}})$ of solutions to the Beltrami-flow problem. Following Remark \ref{rem:interpretation-of-hat-s}, the solutions on $\mathcal{D}(\hat{s})$ are genuinely three-dimensional surface waves unless $r(t) = 0$. Since the curve $\mathcal{D}(\hat{s})$ has a local analytic re-parameterisation, this can only occur at isolated points $t$. We then obtain the connected subset $\mathcal{D}$ by taking the union of the $\mathcal{D}(\hat{\bm{s}})$ over all $\hat{\bm{s}} \in S^1$ with $\hat{s}_1 \hat{s}_2 \neq 0$.

    We now prove the refined set of alternatives. For this, we note that alternative \ref{global-alternative-2} in Theorem \ref{thm:global-bifurcation} directly translates to alternative \ref{alt:D-loop} above. Therefore, it remains to show that alternative \ref{global-alternative-1} in Theorem \ref{thm:global-bifurcation} implies alternatives \ref{alt:D-unbounded} or \ref{alt:D-selfintersect} above. Thus, assume that $\mathfrak{R}(\hat{\bm{s}})$ satisfies alternative \ref{global-alternative-1}. Then, there is a sequence $(\tilde{\Phi}_n,\bm{c}_n,r_n)_{n \in \N}$ such that $\|\tilde{\Phi}_n\|_\curlX + |\bm{c}_n| + |r_n| \rightarrow \infty$ or for all $\delta \in (0,d)$ there exists an $n_\delta$ such that $(\tilde{\Phi}_n,\bm{c}_n,r_n) \notin \tilde{\curlU}_\delta(\hat{\bm{s}})$ for all $n > n_\delta$. We show that the first case yields alternative \ref{alt:D-unbounded} above, and the second case yields alternative \ref{alt:D-selfintersect} above.

    We start with the second case that $(\tilde{\Phi}_n,\bm{c}_n,r_n)_{n \in \N}$ eventually leaves all $\tilde{\curlU}_\delta$. Then the corresponding sequence $((\bar{\bm{u}}_n,\eta_n),\bm{c}_n)$ of solutions to $\mathcal{F}((\bar{\bm{u}},\eta),\bm{c}) = 0$ eventually leaves all $\curlU_\delta$, see \eqref{eq:Udelta}, and therefore the sequence $\min \eta_n \rightarrow -d$ as $n \rightarrow \infty$. Applying the inverse flattening transform, we thus find that alternative \ref{alt:D-selfintersect} above holds.

    Therefore, it remains to consider the first case that $\|\tilde{\Phi}_n\|_\curlX + |\bm{c}_n| + |r_n| \rightarrow \infty$ as $n \rightarrow \infty$. In this case, we first show that the corresponding sequence $(\Phi_n,\bm{c}_n)$ of solutions to $\mathcal{F}(\Phi_n,\bm{c}_n) = 0$, given by
    \begin{equation*}
        \Phi_n = (\bm{U}(\bm{c}_n),0) + r_n(\hat{s}_1 \Phi_1 + \hat{s}_2 \Phi_2) + \tilde{\Phi}_n,
    \end{equation*}
    satisfies $\|\Phi_n\|_\curlX + |\bm{c}_n| \rightarrow \infty$. If $|\bm{c}_n| \rightarrow \infty$, there is nothing to show. If $|\bm{c}_n|$ is uniformly bounded, we obtain a uniform bound on the laminar flows $\bm{U}(\bm{c}_n)$. In addition, using that $\curlX$ can be written as the direct sum $\curlX = \curlX_1 \oplus \curlX_2 \oplus \tilde{\curlX}$ with $\curlX_i = \operatorname{span}\{\Phi_i\}$ for $i = 1,2$ and the fact that $\tilde{\Phi}_n \in \tilde{\curlX}$ for all $n \in \N$ we obtain the estimate
    \begin{equation*}
        \|\tilde{\Phi}_n\|_\curlX + |r_n| \lesssim \|\tilde{\Phi}_n + r_n(\hat{s}_1 \Phi_1 + \hat{s}_2 \Phi_2)\|_\curlX = \|\Phi_n - (\bm{U}(\bm{c}_n),0)\|_{\curlX} \leq \|\Phi_n\|_\curlX + \|(\bm{U}(\bm{c}_n),0)\|_\curlX
    \end{equation*}
    from the continuity of the projections from $\curlX$ onto the $\curlX_i$. Then, subtracting $\|(\bm{U}(\bm{c}_n),0)\|_\curlX$, which is uniformly bounded, we indeed find $\|\Phi_n\|_\curlX \rightarrow \infty$ as $n \rightarrow \infty$.

    Assuming that neither alternative \ref{alt:blow-up-c}, \ref{alt:D-selfintersect}, nor \ref{alt:D-loop} occurs, we now show that $\|\Phi_n\|_\curlX \rightarrow \infty$ implies alternative \ref{alt:blow-up-eta}, that is, $\|\nabla \eta_n\|_{C^{0,\gamma}_\mathrm{per}} \rightarrow \infty$, where $\Phi_n = (\bar{\bm{u}}_n,\eta_n)$. For this, we use that since $\mathcal{F}(\Phi_n,\bm{c}_n) = 0$ it holds that
    \begin{equation*}
        \begin{split}
            \bar{\bm{u}}_n &= \bar{\mathcal{L}}(\alpha \bar{\bm{u}}_n,\eta_n,  \bm{M}(\eta_n,\bm{c}_n)), \\
            \eta_n &= \bar{\mathcal{K}}(\bar{\bm{u}}_n, \eta_n,\bm{c}_n) = \mathcal{K}\left(\nabla \eta_n, Q(\bm{c}_n) - \dfrac{1}{2}B\left(\bar{\mathcal{L}}(\alpha \bar{\bm{u}}_n,\eta_n,  \bm{M}(\eta_n,\bm{c}_n)),\eta_n\right)\right),
        \end{split}
    \end{equation*}
    with $Q(\bm{c}) = \tfrac{1}{2}(c_1^2 + c_2^2)$ and $\bm{M}(\eta,\bm{c}) = (\int_{\Omega_{00}^\eta} U_1(\bm{c}) \,dV, \int_{\Omega_{00}^\eta} U_2(\bm{c}) \,dV)$. Since alternative \ref{alt:blow-up-c} does not occur by assumption, $|\bm{c}_n|$ is uniformly bounded and thus, $|Q(\bm{c}_n)|$ is uniformly bounded and $|\bm{M}(\eta_n,\bm{c}_n)| \lesssim \|\eta_n\|_{L^\infty}$. Next, we use the regularity estimate \eqref{eq:schauder-without-lower-order-term} to obtain, for any $0 \leq \ell \leq k-1$, the estimates
    \begin{equation}\label{eq:reg-est-global-bif-u}
        \|\bar{\bm{u}}_n\|_{C^{\ell+1,\gamma}_\mathrm{per}} \leq C_{dc}\left(\max(\eta_n+d)^{-1},\|\eta_n\|_{C^{\ell+2,\gamma}_\mathrm{per}}\right)\left(|\bm{M}(\eta_n,\bm{c}_n)| + \alpha \|\bar{\bm{u}}_n\|_{C^{\ell,\gamma}_\mathrm{per}}\right).
    \end{equation}
    Note that we replaced the generic constant $C(R)$ in \eqref{eq:schauder-without-lower-order-term} by a specific constant $C_{dc}$, which depends on $\max(\eta_n+d)^{-1}$ and $\|\eta_n\|_{C^{\ell+2,\gamma}_\mathrm{per}}$, which follows directly from the proof of \eqref{eq:schauder-without-lower-order-term}.
    Additionally, combining \eqref{eq:schauder-est-dyn-bc} and \eqref{eq:est-B-dyn-bc}, yields the estimate
    \begin{equation}\label{eq:reg-est-global-bif-eta}
        \begin{split}
            \|\eta_n\|_{C^{\ell'+2,\gamma}_\mathrm{per}} &\leq C_{db}(\|\nabla\eta_n\|_{C^{\ell',\gamma}_\mathrm{per}}) \left\|Q(\bm{c}_n) - \dfrac{1}{2}B\left(\bar{\mathcal{L}}(\alpha \bar{\bm{u}}_n,\eta_n,  \bm{M}(\eta_n,\bm{c}_n)),\eta_n\right)\right\|_{C^{\ell',\gamma}_\mathrm{per}} \\
            &\leq C_{db}(\|\nabla\eta_n\|_{C^{\ell',\gamma}_\mathrm{per}}) \Bigg(|Q(\bm{c}_n)| \\
            &\qquad \qquad + \dfrac{1}{2} C_B\left(\max(\eta_n+d)^{-1},\|\eta_n\|_{C^{\ell'+1,\gamma}_\mathrm{per}}\right) \|\bar{\mathcal{L}}(\alpha \bar{\bm{u}}_n,\eta_n,  \bm{M}(\eta_n,\bm{c}_n))\|_{C^{\ell',\gamma}}^2\Bigg) \\
            &\leq C_{db}(\|\nabla\eta_n\|_{C^{\ell',\gamma}_\mathrm{per}}) \left(|Q(\bm{c}_n)| + \dfrac{1}{2} C_B\left(\max(\eta_n+d)^{-1},\|\eta_n\|_{C^{\ell'+1,\gamma}_\mathrm{per}}\right) \|\bar{\bm{u}}_n\|_{C^{\ell',\gamma}_\mathrm{per}}^2\right)
        \end{split}
    \end{equation}
    for any $0 \leq \ell' \leq k$. Again, we replace the generic constants $C(R)$ in \eqref{eq:est-B-dyn-bc} and \eqref{eq:schauder-est-dyn-bc} by specific constants $C_{B}$ and $C_{db}$, respectively, to more precisely capture the blow-up behaviour in $\eta_n$. This shows that if $\|\bar{\bm{u}}_n\|_{C^{\ell+1,\gamma}_\mathrm{per}} + \|\eta_n\|_{C^{\ell+3,\gamma}_\mathrm{per}} \rightarrow \infty$ then also $\|\bar{\bm{u}}_n\|_{C^{\ell,\gamma}_\mathrm{per}} + \|\eta_n\|_{C^{\ell+2,\gamma}_\mathrm{per}} \rightarrow \infty$ for $0 \leq \ell \leq k-1$. Indeed, if $\|\bar{\bm{u}}_n\|_{C^{\ell + 1,\gamma}_\mathrm{per}} \rightarrow \infty$, then by \eqref{eq:reg-est-global-bif-u}, we obtain $\|\bar{\bm{u}}_n\|_{C^{\ell,\gamma}_\mathrm{per}} + \|\eta_n\|_{C^{\ell+2,\gamma}_\mathrm{per}} \rightarrow \infty$. On the other hand, if $\|\eta\|_{C^{\ell+3,\gamma}_\mathrm{per}} \rightarrow \infty$, we can apply \eqref{eq:reg-est-global-bif-eta} with $\ell' = \ell+1$ in combination with \eqref{eq:reg-est-global-bif-u} to find that $\|\bar{\bm{u}}_n\|_{C^{\ell,\gamma}_\mathrm{per}} + \|\eta_n\|_{C^{\ell+2,\gamma}_\mathrm{per}} \rightarrow \infty$. A bootstrapping argument then directly shows that $\|\Phi_n\|_\curlX \rightarrow \infty$ implies $\|\bar{\bm{u}}_n\|_{C^{0,\gamma}_\mathrm{per}} + \|\eta_n\|_{C^{2,\gamma}_\mathrm{per}} \rightarrow \infty$. Using \eqref{eq:reg-est-global-bif-eta} again with  $\ell'=0$ then leads to $\|\bar{\bm{u}}_n\|_{C^{0,\gamma}_\mathrm{per}} + \|\eta_n\|_{C^{1,\gamma}_\mathrm{per}} \rightarrow \infty$.

We next show that blow-up of $\|\bar{\bm{u}}_n\|_{C^{0,\gamma}_\mathrm{per}}$ implies blow-up of $\|\eta_n\|_{C^{1,\gamma}_\mathrm{per}}$ and therefore, the blow-up of $\bm{u}_n$ can be removed as an isolated alternative. For this, we note that the Bernoulli equation yields that
\[
\frac{1}{2} \int_{B_{00}} B(\bar{\bm{u}}, \eta) \, d\bm{x}' = \int_{B_{00}} (Q(\bm{c}) - g\eta) \, d\bm{x}',
\]
which means that $\|\bar{\bm{u}}_n|_{z=0}\|_{L^2}$ is uniformly bounded if $\|\eta_n\|_{C^1}$ and $\bm{c}_n$ are bounded. Therefore, by  Lemma \ref{est-of-u-in-eta}, we conclude that
\begin{equation*}
    \|\bar{\bm{u}}_n\|_{C^{0,\gamma}_\mathrm{per}}  \rightarrow \infty \implies \|\eta_n\|_{C^{1,\gamma}_\mathrm{per}}  \rightarrow \infty.
\end{equation*}

Finally, we show that $\|\eta_n\|_{C^0_\mathrm{per}}$ is bounded, so that blow-up of $\|\eta_n\|_{C^{1,\gamma}_\mathrm{per}}$ implies blow-up of $\|\nabla\eta_n\|_{C^{0,\gamma}_\mathrm{per}}$.
Indeed, clearly $\eta_n>-d$ for all $n$. On the other hand, at a maximum of $\eta$, we have that $K_M \le 0$, so that
\[
g\eta = Q(\bm{c})-\frac12 |\bm u|^2+2\sigma K_M \le Q(\bm{c}).
\]
Hence, $\eta_n \le Q(\bm c_n)/g$ is also bounded from above if $|\bm c_n|$ is uniformly bounded.
This shows that $\|\Phi_n\|_\curlX \rightarrow \infty$ indeed implies alternative \ref{alt:blow-up-eta} and thus completes the proof of the refined alternatives.
\end{proof}

\section{Discussion}\label{sec:discussion}

In this paper, we constructed a global family of genuinely three-dimensional gravity-capillary water waves on Beltrami flows. We obtained this family by reformulating the steady water wave problem as a bifurcation problem of the form `identity plus compact' with a two-dimensional kernel. Here, only a combination of both kernel elements yields true three-dimensional waves. We established the global continuation of the local bifurcation result, and we obtained that the global continuation breaks down if either the surface gradient or the wave velocity blows up, the surface intersects the flat bottom, or the bifurcation curve is a closed loop.

We now discuss related open questions and potential future directions.

\paragraph{Blow-up of the surface gradient in $C^0$.} In Theorem \ref{thm:global-bif-Beltrami}, we establish the blow-up of the surface gradient $\nabla \eta$ in $C^{0,\gamma}$ as an alternative in the global bifurcation result. This is somewhat unsatisfactory, as one would hope to obtain a blow-up in $C^0$, which would allow for a direct physical interpretation of this alternative, in the sense that the wave starts to `overturn'. The main technical issue currently preventing such an extension is that the Schauder estimates we use for the sharpened blow-up alternative have constants that depend on the Hölder norm of $\nabla \eta$ rather than the simple $C^0$-norm. Moreover, the proof of Lemma \ref{est-of-u-in-eta} to bound the velocity field by its $L^2$-norm at the surface crucially relies on $\eta \in C^{1,\gamma}$ to obtain compactness and close a contradiction argument. Therefore, control of just $\|\nabla \eta\|_{C^0}$ is insufficient to prevent the blow-up of higher derivatives and the velocity field. Closing this gap seems genuinely difficult in the present setting. In comparison, Nguyen managed to reduce the blow-up alternative to the $C^1$ norm of the surface profile $\eta$ in the global continuation analysis of gravity-capillary waves for Darcy flow in \cite{nguyen2026-03Nonlinearity}. This was accomplished using a sharp $L^p \to W^{1,p}$ estimate of the Neumann--Dirichlet operator $G[\eta]^{-1}$ for $1<p<2+\epsilon$ by Dahlberg and Kenig (see \cite[Theorem 4.11]{nguyen2026-03Nonlinearity} and \cite{dahlberg1987-05AnnofMath}) depending on $\|\eta\|_{C^1}$; where $\varepsilon$ could be arbitrarily close to $0$ depending on the domain (and thus, on $\eta$). Translated to our setting, one could use this estimate in the irrotational case $\alpha =0$ to bound the trace of the velocity field on the surface. However, in \cite{nguyen2026-03Nonlinearity} the quantity which appears in the dynamic boundary condition has the simple form $G[\eta]^{-1}\partial_x \eta$, and has bounded H\"older norm by Morrey's inequality. This results in a H\"older estimate of $D^2 \eta$ by inverting the mean curvature operator. In our setting, on the other hand, the dynamic boundary condition contains squares of derivatives of $G[\eta]^{-1}\partial_x \eta$ which end up in $L^{\frac{p}{2}}$, with $p/2 <1+\epsilon/2$, by the Dahlberg--Kenig estimate. It is therefore not clear how to use the dynamic boundary condition to upgrade this to a H\"older estimate of the gradient of $\eta$, since even if the mean curvature bound yields a bound on $\eta$ in $W^{2,p/2}$, it is not enough to directly conclude a Hölder estimate unless $\varepsilon > 2$ and thus $p/2 > 2$.

\paragraph{Large-amplitude overhanging waves.} One of the main assumptions in our problem formulation is that the surface is a graph. This inherently prevents the description of overhanging waves, that is, waves in which some horizontal coordinates correspond to more than one surface point. In two dimensions, there exists an explicit family of periodic capillary waves in infinite depth, called Crapper waves \cite{crapper1957-08JFluidMech}, which includes overhanging waves; see also \cite{kinnersley1976-09JFluidMech} for a corresponding family in finite depth, and \cite{akers2014} for a family of gravity perturbed Crapper waves. In the case of vorticity, it has recently been shown that periodic and solitary overhanging waves can arise for constant vorticity \cite{hur2022,davila2026-05Inventmath}. Furthermore, there are recent two-dimensional global bifurcation results for water waves with vorticity which explicitly allow for overhanging wave profiles (see, e.g., \cite{constantin2016, haziot2023-04ArchRationMechAnal, wahlen2023, wahlen2024}). These results do not carry over to the three-dimensional case, in particular, we recall that there are no non-trivial three-dimensional water waves with constant vorticity \cite{martin2022}. Nevertheless, numerical results suggest that overhanging three-dimensional water waves can be found using a continuation approach based on a dimension-breaking bifurcation from an overhanging two-dimensional wave  \cite{akers2017-01WaveMotion}. However, a rigorous proof of their existence remains an open question.

\section*{Acknowledgments}
We thank Max Engelstein for helpful discussions on Appendix \ref{sec:Schauder estimates}.
All three authors have been supported by the Swedish Research Council (grant no.~2020-00440). 
BH was also partially supported by the Deutsche Forschungsgemeinschaft (DFG, German Research Foundation) -- Project-IDs 444753754 and 543917644.
All three authors were supported by the Swedish Research Council under grant no.~2021-06594 while in residence at Institut Mittag-Leffler in Djursholm, Sweden, during the fall semester of 2023. All three authors also acknowledge funding from the Deutsche Forschungsgemeinschaft (DFG, German Research Foundation) -- project number 545145736, to facilitate in-person discussions.

\appendix
\section{Technical results}\label{sec-app:technical-result}

In this appendix, we prove some technical results about periodic solutions to the homogeneous div-curl problem
\begin{equation}\label{eq:div-curl-homogeneous}
	\begin{alignedat}{2}
         	\nabla \times \bm{v} &= 0 \quad &&\text{in } \Omega^\eta, \\
         	\nabla \cdot \bm{v} &= 0 &&\text{in } \Omega^\eta, \\
         	\bm{v} \cdot \bm{n} &= 0 &&\text{on } \partial\Omega^\eta, 
	\end{alignedat}
 \end{equation}
 that are needed in our analysis in Section \ref{sec:div-curl-problem}. We assume that $\eta \in \curlB^{k+2,\gamma}$ and consider  solutions $\bm v\in C^{k+1,\gamma}_\mathrm{per}$ with $k\ge 0$.  Note that any  solution is necessarily of the form
 \[
 \bm{v}=\nabla( b_1 x+b_2 y - \phi)=b_1\bm{e}_1+b_2 \bm{e}_2 -\nabla \phi,
 \]
where $\phi =\phi(\bm{b}) \in \Holderspace{k+2}$, $\bm{b}=(b_1, b_2)$, is the unique (up to additive constants) solution of
\begin{alignat*}{2}
        \Delta \phi &= 0 && \text{ in } \Omega^\eta, \\
        \partial_n \phi &= (b_1 \bm{e}_1 + b_2 \bm{e}_2) \cdot \bm{n}\quad && \text{ on } \partial \Omega^\eta.
\end{alignat*}
Thus, the space of solutions to \eqref{eq:div-curl-homogeneous} is two-dimensional and parameterised by $\bm{b}$. We shall now show that it can also be uniquely specified by imposing integral constraints.

\begin{proposition}\label{prop:div-curl-isomorphism}
    Let $\bm{\lambda}_1, \bm{\lambda}_2 \in \R^2$ be linearly independent. Then for all $\bm{M} = (M_1,M_2) \in \R^2$ the system \eqref{eq:div-curl-homogeneous} has a unique solution $\bm{v} \in C^{k+1,\gamma}_\mathrm{per, div, T}$ satisfying
    \begin{equation}\label{eq:div-curl-mean-fixed}
            \int_{\Omega_{00}^\eta} v_j \,\de V  = M_j  \quad \text{for } j = 1,2.
    \end{equation} 
\end{proposition}

We will prove this result by relating the integral conditions to more classical conditions in terms of fluxes.
Define the cross sections
\[
    \begin{split}
        \Sigma_1(a_1) &\coloneqq \{(a_1 \bm{\lambda}_1 + a_2 \bm{\lambda}_2,z) : a_2 \in (0,1), z \in \R\} \cap \Omega^\eta, \\
        \Sigma_2(a_2) &\coloneqq \{(a_1 \bm{\lambda}_1 + a_2 \bm{\lambda}_2,z) : a_1 \in (0,1), z \in \R\} \cap \Omega^\eta. 
    \end{split}
\]
Then we have that
\begin{equation}
    \Omega_{00}^\eta = \bigcup_{a \in (0,1)} \Sigma_j(a)
    \label{eq:union-domain}
\end{equation}
for $j = 1,2$. Furthermore, due to the periodicity of $\eta$, we have $\Sigma_j(1) = \Sigma_j(0)+(\bm{\lambda}_j,0)$. 
The proof of Proposition \ref{prop:div-curl-isomorphism} relies on the following standard result, which is proved by applying the divergence theorem to the vector field $\bm{v}$ in the domain $\Omega_{j, a_1,a_2} \coloneqq \bigcup_{a \in (a_1,a_2)} \Sigma_j(a) \subset \Omega^\eta$.

\begin{lemma}\label{lem:div-curl-constant-fluxes}
    Let $\bm{v} \in \Holderspace{k+1}$ be a solution of \eqref{eq:div-curl-homogeneous}.
    Then, the fluxes
    \[
        f_j(a) \coloneqq \int_{\Sigma_j(a)} \bm{v} \cdot \bm{n} \,\de S
    \]
    are independent of $a$ for $j = 1,2$.
\end{lemma}

\begin{proof}[Proof of Proposition \ref{prop:div-curl-isomorphism}]
We first show the uniqueness, that is, we prove that $\bm{v}=0$ if $M_1=M_2=0$.
    Let $\bm{n}_j \in \R^3$ be the outer normal vector on $\Sigma_j(a)$ for $a \in (0,1)$ and $j = 1,2$. Note that $\bm{n}_j$ is independent of $a$. Since by assumption
    \[
        \int_{\Omega_{00}^\eta} v_j \,\de x = 0
    \]
    for $j = 1,2$, this also holds for every linear combination of $v_1$ and $v_2$.
    In particular, since $\bm{n}_j$ has no contribution in the $z$-direction, that is, $n_{j,3} = 0$, it holds that
    \[
        \int_{\Omega_{00}^\eta} \bm{v} \cdot \bm{n}_j \,\de x = 0, \quad j=1,2.
    \]
    Finally, using Lemma \ref{lem:div-curl-constant-fluxes} there exist $\bar{f}_1,\bar{f}_2 \in \R$ such that $f_j(a) = \bar{f}_j$ for $j = 1,2$ and $a \in (0,1)$. Then, by Fubini's theorem and \eqref{eq:union-domain}  we find that
    \[
        0 = \int_{\Omega_{00}^\eta} \bm{v} \cdot \bm{n}_j \,\de x =|\bm{\lambda}_j| \int_0^1 \int_{\Sigma_j(a)} \bm{v} \cdot \bm{n}_j \,\de S \,\de a =|\bm{\lambda}_j|  \int_0^1 f_j(a) \,\de a = |\bm{\lambda}_j| \bar{f}_j.
    \]
    This shows that $\bm{v}$ is a fluxless harmonic knot and must therefore vanish by the Hodge Decomposition Theorem (see e.g.~the orthogonality part of \cite[Proposition 2]{cantarella2002} which can be adapted to the present geometrical setting using \cite[Theorem 2.1]{lokharu2019}). This completes the uniqueness part.

Now consider the linear map 
    \[
        \bm{v} \mapsto (M_1,M_2) \coloneqq \left(\int_{\Omega_{00}^\eta} v_1 \,\de x, \int_{\Omega_{00}^\eta} v_2 \,\de x\right)
    \]
from the two-dimensional space of solutions of \eqref{eq:div-curl-homogeneous} to $\R^2$. By the uniqueness part above, this map has a trivial kernel. Hence, it is an isomorphism. This completes the proof.
\end{proof}

\section{Abstract multi-parameter local bifurcation theory}
\label{sec:abstract local bif}

We follow the general setup in \cite{seth2024}.
Consider the operator
$$ F\colon X \times \R^n \to Y,$$
where $X$ and $Y$ are Banach spaces. We assume the following:
\begin{enumerate}[label=(B\arabic*)]
    \item\label{H1} $F$ is $C^k$ with $k \geq 2$, $k=\infty$ or $k=\omega$ (that is, $F$ is analytic), $F(0,\bm{c})=0$ for all $\bm{c} \in \R^n$ and there is $\bm{c}^*\in \mathbb{R}^n$ such that $ L\coloneqq D_1F[0,\bm{c}^*]  \colon X  \to Y$ is a Fredholm operator of index 0.
    \item\label{H2} The kernel of $L$ is $n$-dimensional and $\ker L=\spn\{x_1,x_2,\ldots,x_n\}$.
    \item\label{H3} The spaces $X$ and $Y$ are decomposed as 
    \[
      X= \left ( \bigoplus^{n}_{i=1} X_i \right) \bigoplus \tilde{X}, \qquad
      Y= \left ( \bigoplus^{n}_{i=1} Y_i \right) \bigoplus \tilde{Y} 
    \]
with $X_i=\spn\{x_i\}$, $Y_i=\spn\{y_i\}$ and $\tilde{Y}=\ran L$.
    \item\label{H4} $Q_i$ and $P_i$ are projections onto $X_i$ and $Y_i$ along $\widehat{X}_i \bigoplus \tilde{X}$ and $\widehat{Y}_i \bigoplus \tilde{Y}$ respectively, where   
    $$\widehat{X}_i =\bigoplus^{n}_{j=1,j \neq i} X_j$$ and likewise for $\widehat{Y}_i$.

    \item\label{H5} The matrix $(\nu_{ij})$ determined by
    \begin{equation}\label{eq:nu-matrix}
        P_iD_1D_{j+1}F[0,\bm{c}^*](x_i,c_j-c^*_j)= \nu_{ij}(c_j-c^*_j)y_i
    \end{equation}
    is invertible.

    \item\label{H6} There exist closed subspaces $\tilde{X}_i \subset \tilde{X}$ and $\tilde{Y}_i \subset \tilde{Y}$ for each $i=1,\ldots,n$ such that 
    $$F(\widehat{X}_i \oplus \tilde{X}_i,\bm{c}) \subseteq \widehat{Y}_i \oplus \tilde{Y}_i,$$
    and 
    $$(I-P_i)D_1F[0,\bm{c}^*]|_{\widehat{X}_i \oplus \tilde{X}_i}\colon \widehat{X}_i \oplus \tilde{X}_i \to \widehat{Y}_i \oplus \tilde{Y}_i$$
    is a Fredholm operator of index $0$ with kernel $\widehat{X}_i$.
\end{enumerate}

\begin{theorem}[{\cite[Thm.~A.1]{seth2024}}]
\label{thm:seth}
Let $F$ be an operator with the above properties. Then there exists an $\epsilon>0$ such that for every $\bm{s} \in B_{\epsilon}(0;\R^n) \subset \R^n$ the equation $F(x,\bm{c})=0$ has a solution $(x(\bm{s}),\bm{c}(\bm{s}))$ with $(x(\cdot), \bm{c}(\cdot)) \in C^{k-1}(B_{\epsilon}(0;\R^n), X \times \mathbb{R}^n )$ and
$$x(\bm{s})=\sum^n_{i=1}s_ix_i+o(|\bm{s}|), \quad \bm{c}(\bm{s})=\bm{c}^*+O(|\bm{s}|).$$
\end{theorem}

\begin{remark}
    We point out that the local bifurcation result in Theorem \ref{thm:seth} is stronger than the one in \cite[Thm.~I.19.6]{kielhofer2012}, which obtains a number of one-dimensional curves compared to the $n$-dimensional surface of solutions obtained in Theorem \ref{thm:seth}.
\end{remark}

For any $x\in X$ we write
$$x=\sum_{i=1}^n s_ix_i+ \tilde{x},$$
with $\tilde{x} \in \tilde{X}$. 
Using this, it is convenient for a multi-parameter global bifurcation analysis, as in Section \ref{sec:global-bif}, to introduce the operator
$$G(\tilde{x},\bm{c},\bm{s})=F(\textstyle\sum_{i=1}^n s_ix_i + \tilde{x},\bm{c}),$$
where we think of $\bm{s}=(s_1, \ldots, s_n)$ as parameters and solve for $(\tilde{x},\bm{c})$ in terms of $\bm{s}$. For the global bifurcation argument, we then need to check that the linear operator
\begin{align*}
     A(\bm{s})\colon \tilde{X} \times \R^n &\to Y\\
     (\tilde{\bm{v}},b)  &\mapsto G_{\tilde{x}}(\tilde{x}(\bm{s}),\bm{c}(\bm{s}),\bm{s})\tilde{\bm{v}}+ G_{\bm{c}}(\tilde{x}(\bm{s}),\bm{c}(\bm{s}),\bm{s})b
\end{align*}
with $\tilde x(\bm{s})=x(\bm{s})-\sum_{i=1}^n s_i x_i$ is a linear isomorphism.

\begin{proposition}
\label{prop-app:isomorphism}
    Let $F$ and $(x(\bm{s}), \bm{c}(\bm{s}))$ be as in Theorem \ref{thm:seth}. Fix $\hat{\bm{s}}\in S^{n-1}$ with $\hat{s}_i \ne 0$ for all $i = 1,\dots,n$. Then the linear map $A(r\hat{\bm{s}})$ is an isomorphism for $0<r\ll 1$.
\end{proposition}

\begin{proof}
Let $P=\sum^n_{i=1} P_i$, $Q=\sum^n_{i=1} Q_i$, $\tilde{P}=I-P$, $\tilde{Q}=I-Q$.
We decompose the operator as 
\begin{align*}
 A(\bm{s})= \begin{pmatrix}
\tilde{P}G_{\tilde{x}}(\tilde{x}(\bm{s}),\bm{c}(\bm{s}),\bm{s}) & \tilde{P}G_{c_1}(\tilde{x}(\bm{s}),\bm{c}(\bm{s}),\bm{s}) & \cdots &\tilde{P}G_{c_n}(\tilde{x}(\bm{s}),\bm{c}(\bm{s}),\bm{s})\\
P_{1}G_{\tilde{x}}(\tilde{x}(\bm{s}),\bm{c}(\bm{s}),\bm{s}) & P_{1}G_{c_1}(\tilde{x}(\bm{s}),\bm{c}(\bm{s}),\bm{s}) & \cdots & P_{1}G_{c_n}(\tilde{x}(\bm{s}),\bm{c}(\bm{s}),\bm{s})\\
\vdots & \vdots & \ddots &\vdots\\
P_{n}G_{\tilde{x}}(\tilde{x}(\bm{s}),\bm{c}(\bm{s}),\bm{s}) &  P_{n}G_{c_1}(\tilde{x}(\bm{s}),\bm{c}(\bm{s}),\bm{s}) & \cdots & P_{n}G_{c_n}(\tilde{x}(\bm{s}),\bm{c}(\bm{s}),\bm{s})
\end{pmatrix}.   
\end{align*}
In what follows, we identify $\oplus_{i=1}^n Y_i$ with $\R^n$ and abuse notation by identifying $P_i y$ with a real number.
The upper left corner $a_{11}(\bm{s})$ equals $\tilde P D_1F[0,\bm{c}^*](I-\tilde Q)=\tilde P L(I-\tilde Q) \colon \tilde X \to \tilde Y$ for $\bm{s}=0$, hence it is an isomorphism for sufficiently small $\bm{s}$. The remaining entries are $o(|\bm{s}|)$.
Since $A(\bm{s})$ is Fredholm of index zero for $\bm{s}=0$, by continuity this also holds for small $\bm{s}$. Hence,  $A(\bm{s})$ is invertible if and only if $\ker A(\bm{s})$ is trivial.
We furthermore note that by assumption \ref{H6} above, 
\[
P_i G(\tilde x, \bm{c}, \bm{s})\Big|_{s_i=0, \tilde x \in \tilde X_i}=0.
\]
It follows that $P_i G_{c_j}(\tilde x(\bm{s}),\bm{c}(\bm{s}),\bm{s})|_{s_i=0}=0$ for all $i$ and $j$.
Hence, we can factorise
\begin{align*}
 A(r\hat{\bm{s}})=\begin{pmatrix}
a_{11}(r\hat{\bm{s}}) & a_{12}(r\hat{\bm{s}}) & \cdots & a_{1(n+1)}(r\hat{\bm{s}})\\
r d_{1}(r, \hat{\bm{s}}) & r\hat{s}_1 k_{11}(r\hat{\bm{s}}) & \cdots & r\hat{s}_1 k_{12}(r\hat{\bm{s}})\\
\vdots & \vdots & \ddots & \vdots \\
r d_{n}(r, \hat{\bm{s}}) & r\hat{s}_n  k_{n1}(r\hat{\bm{s}}) & \cdots & r\hat{s}_n k_{nn}(r\hat{\bm{s}})
\end{pmatrix}   
\end{align*}
where  the functions $d_i(r, \hat{\bm{s}})$ are continuous in $r \in [0, \epsilon)$ and $\hat{\bm{s}} \in S^{n-1}$, while the matrix $K(\bm{s})=(k_{ij}(\bm{s}))$ is continuous in $\bm{s} \in B_\varepsilon(0)$ and simplifies to
\[
k_{ij}(0)=\nu_{ij}
\]
for $\bm{s}=0$ with $\nu_{ij}$ defined in \eqref{eq:nu-matrix}. Hence, $K(\bm{s})$ is invertible with uniformly bounded inverse for $\bm{s}$ sufficiently small due to assumption \ref{H5}. Since $\hat{s}_i \ne 0$ for all $i=1, \ldots, n$, the equation $A(\bm{s})(\tilde{\bm{v}}, b)=0$ can be written as
\begin{align*}
    a_{11}(\bm{s})\tilde{\bm{v}}+O(|\bm{s}|)b=0,\\
    D(r, \hat{\bm{s}})\tilde{\bm{v}}+K(\bm{s}) b=0,
\end{align*}
where $D(r, \hat{\bm{s}})=(\hat{s}_1^{-1} d_1(r, \hat{\bm{s}}), \ldots, \hat{s}_n^{-1} d_n(r, \hat{\bm{s}}))^T$. Since $K(\bm{s})$ is invertible for sufficiently small $\bm{s}$, we can solve the second equation for $b$,
\[
    b=-K(\bm{s})^{-1}D(r, \hat{\bm{s}})\tilde{\bm{v}}.
\]
Substituting this in the first equation, we have reduced to
\[
(a_{11}(\bm{s})+o(|\bm{s}|))\tilde{\bm{v}}=0
\]
and since $a_{11}(0)$ is invertible, it follows that $\tilde{\bm{v}}=0$ if $|\bm{s}|$ is sufficiently small. Hence, $\ker A(\bm{s})$ is trivial.
\end{proof}

\begin{remark}
    From the proof, it follows that the size of $r$, where $A$ is an isomorphism, can be chosen uniformly if we cut away a conical neighbourhood of the union of the coordinate hyperplanes $\{s\in \R^n\colon s_i=0\}$. However, since we will fix the direction $\hat{\bm{s}}$ in the global bifurcation analysis, this will not be needed.
\end{remark}

\section{A variation of the analytic global bifurcation theorem}\label{app:globalBifurcation}

In this section, we record a variation of the analytic global bifurcation theorem \cite[Theorem 9.1.1]{buffoni2003} or \cite[Theorem 6]{constantin2016}, which applies to the setting of Section \ref{sec:global-bif}. Here the assumptions which guarantee the bifurcation of a local branch of nontrivial solutions are replaced by the assumption that such a branch exists, and the assumption that the bifurcation parameter is not constant along the local branch by the assumption that the partial Fr\'echet derivative of the nonlinear operator is an isomorphism (see items \ref{app:gb-assump2} and \ref{app:gb-assump3} below). Similar alterations are made in \cite[Theorem 6.1]{chen2018} and \cite[Corollary 4.1]{walsh2014}. Let $X,Y$ be Banach spaces over $\R$ and let $U \subset X \times \R$. Additionally, let $G : U \rightarrow Y$ be an $\R$-analytic function. Then, the following theorem holds.

\begin{theorem}\label{thm-app:global-bifurcation}
    Assume that the following statements are true.
    \begin{enumerate}[ref=(\roman*)]
        \item\label{app:gb-assump1} $D_\xi G(\xi,r)$ is Fredholm with index zero for all $(\xi,r) \in U$ with $G(\xi,r) = 0$;
        \item\label{app:gb-assump2} there exists an analytic local bifurcation branch $\curlR = \{(\xi(t),r(t)) : t \in (0,\varepsilon) \text{ and } G(\xi(t),r(t)) = 0\}$, which emerges from $(\xi(0),r(0)) = (0,r^*)$ for some $r^* \in \R$;
        \item\label{app:gb-assump3} $D_\xi G(\xi(t),r(t)) : X \times \R \rightarrow Y$ is an isomorphism for $0 < t \ll 1$;
        \item\label{app:gb-assump4} there exists a sequence $(\mathcal{Q}_j)_{j \in \N}$ of bounded, closed subsets of $U$ with $U = \bigcup_{j \in \N} \mathcal{Q}_j$ such that the set $\curlS \cap \mathcal{Q}_j$ is compact for each $j \in \N$, where $\curlS = \{(\xi,r) \in U \,:\, G(\xi,r) = 0\}$.
    \end{enumerate}
    Then, the results of \cite[Theorem 6]{constantin2016} hold. In particular, there exists a continuous curve $\mathfrak{R}$ extending $\curlR$ with
    \begin{equation*}
        \mathfrak{R} = \{(\xi(t),r(t)) \,:\, t \in [0,\infty)\} \subset \curlS \subset U.
    \end{equation*}
    The curve $\mathfrak{R}$ has a local analytic re-parameterisation at each point, and at least one of the following alternatives occurs:
    \begin{enumerate}
        \item[(a1)] for every $j \in \N$ there exists $t_j > 0$ such that $(\xi(t),r(t)) \notin \mathcal{Q}_j$ for all $t \geq t_j$;
        \item[(a2)] $\mathfrak{R}$ is a closed loop, that is, there exists a $T > 0$ such that $\mathfrak{R} = \{(\xi(t),r(t)) : t \in [0,T]\}$ and $(\xi(T),r(T)) = (0,r^*)$.
    \end{enumerate}
\end{theorem}
\begin{proof}
    Combining assumptions \ref{app:gb-assump2} and \ref{app:gb-assump3}, we obtain that
    \begin{equation*}
        \curlR \subset \mathfrak{N} = \{(\xi,r) \in \curlS \,:\, \ker(D_\xi G(\xi,r)) = \{0\}\},
    \end{equation*}
    cf.~\cite[Equation (9.1)]{buffoni2003}. This inclusion is the only part of the proof of \cite[Theorem 9.1.1]{buffoni2003} that relies on the fact that the kernel at the bifurcation point is one-dimensional. Hence, the rest of the proof of \cite[Theorem 9.1.1]{buffoni2003} as well as the changes made by \cite{constantin2016} hold unchanged under the above assumptions, and the statement follows.
\end{proof}

\section{Schauder estimates}
\label{sec:Schauder estimates}

In this section, we provide a proof for the Schauder estimates used in the main part. To do this, we recall the flattened div-curl problem \eqref{eq:formulation-flat} which is given by
\begin{subequations}
    \label{eq:new formulation-flat}
    \begin{alignat}{2}
            \nabla \times (A(\eta)\bar{\bm{u}}) &=\bar{\bm{w}} \quad &&\text{in } \Omega^0, \label{eq:new curl-flat}\\
            \nabla \cdot \bar{\bm{u}} &= 0 &&\text{in } \Omega^0, \label{eq:new div-flat}\\
            \bar{\bm{u}} \cdot \bm{n} &= 0 &&\text{on } \partial \Omega^0. \label{eq:new kinboundary-flat} 
    \end{alignat}
\end{subequations}
Here, we recall from \eqref{eq:formulation-flat} that $A(\eta) \in \R^{3\times 3}$ is given by
\begin{equation}\label{eq:new Aeta}
        A(\eta) = \dfrac{1}{\det(D\frakF)}D\frakF^T D\frakF,
\end{equation}
which contains at most first-order derivatives of $\eta$. 
Note that $A(\eta)$ is symmetric and positive definite,
    \[
     A(\eta) \bm{\xi} \cdot \bm{\xi}= \dfrac{1}{\det(D\frakF)}|D\frakF\bm{\xi}|^2.
    \]
Moreover, the quadratic form is bounded from above and below by positive constants which only depend on $\max (\eta+d)^{-1}$ and
$\|\eta\|_{C^1_\mathrm{per}}$, and hence there are similar upper bounds for the matrix norms of $A(\eta)$ and $(A(\eta))^{-1}$.

\begin{theorem}\label{thm-app:schauder}
    Let either $k \geq 0$ or $k = -1$ and $\eta \in \curlB^{k+2,\gamma}_{\delta, R}$. Additionally, let $\bar{\bm{w}} \in C^{k,\gamma}_\mathrm{per}(\overline{\Omega^0};\R^3)$ if $k \geq 0$ and $\bar{\bm{w}} \in C^{0,\gamma}_\mathrm{per}(\overline{\Omega^0};\R^3)$ if $k = -1$. Then, a solution $\bar{\bm{u}} \in C^{k+1,\gamma}_\mathrm{per}(\overline{\Omega^0};\R^3)$ to the flattened system \eqref{eq:new formulation-flat} satisfies the Schauder estimate
    \begin{equation}
        \Holdernormper{\bar{\bm{u}}}{k+1}{\gamma}
            \leq C(\delta^{-1}, R) (\|\bar{\bm{u}}\|_{C^0_\mathrm{per}} + \|\bar{\bm{w}}\|_{C^{k,\gamma}_\mathrm{per}}), \label{new est-u_1_gamma}
    \end{equation}
    if $k \geq 0$, and
    \begin{equation}
        \Holdernormper{\bar{\bm{u}}}{0}{\gamma}
            \leq C(\delta^{-1}, R) (\|\bar{\bm{u}}\|_{C^0_\mathrm{per}} + \|\bar{\bm{w}}\|_{C^0_\mathrm{per}}), \label{new est-u_0_gamma}
    \end{equation}
    if $k = -1$.
\end{theorem}

\begin{remark}
    In the case $k=-1$ the equations in this section have to be considered in the weak sense.
\end{remark}

When $k\ge 1$, the Schauder estimate \eqref{new est-u_1_gamma} could be proved by considering the equivalent system \eqref{eq:div-curl-full} for $\bm{u}$, taking another curl, and using the identity 
\begin{equation}
\label{eq:curlcurl}
\nabla \times (\nabla\times \bm{u})=\nabla (\nabla \cdot \bm{u})-\Delta \bm{u}
\end{equation}
together with \eqref{eq:div-curl-full eq}, to obtain $-\Delta \bm{u}=\curl \bm{w}$. Flattening, one obtains the transformed system
\begin{equation}
\label{eq:transformed Laplace}
        -D\frakF^{-1}\nabla \cdot A(\eta)^{-1}\nabla \left( \dfrac{1}{\det(D\frakF)} D\frakF \bar{\bm{u}}\right) = \curl A(\eta) \bar{\bm{w}},
\end{equation}
together with suitable boundary conditions. The estimate \eqref{new est-u_1_gamma} for $k\ge 1$ now follows by applying the  theory for elliptic systems with general boundary conditions developed in \cite{agmon1964}. For $k = 0$ it is however difficult to locate this estimate in the literature. Note that the right-hand side of \eqref{eq:transformed Laplace} is not in any H\"older space in this case. However, it does consist of derivatives of H\"older functions, and one can therefore in principle use gradient estimates for systems in divergence form to prove the estimate; in particular it is possible to adapt the proof in \cite[Appendix A]{engelstein2016AnnSciEcoleNormSup} to the current situation by considering  $\tilde{\bm{u}}= (\det D\frakF)^{-1} D\frakF \bar{\bm{u}}$ as the unknown. 
It is not clear how to get the estimate \eqref{new est-u_0_gamma} this way, though.
A natural alternative approach is to consider a vector potential for $\bar{\bm{u}}$ satisfying a second-order elliptic system in divergence form. This unfortunately seems to lead to a mismatch of regularity of the coefficients. Instead, our strategy is to freeze the coefficients in the first-order system \eqref{eq:new formulation-flat}, and only after that introduce a vector potential to prove estimates for the frozen system with constant coefficients. We then go back to the first-order div-curl system when we do the classical perturbation argument. Here we use the special structure of the system in order to simplify the proof; in particular, we do not distinguish between interior points and boundary points. Furthermore, we choose to prove \eqref{new est-u_1_gamma} by the same approach to streamline the presentation.

We begin with a lemma for a system with constant coefficients.

\begin{lemma}
\label{lemma:frozen}
Let  $A_0\in \mathbb R^{3\times 3}$ be a positive definite, symmetric matrix and $k \ge -1$. Additionally, let $\bar{\bm{w}}_j \in  C_\mathrm{per}^{k+1, \gamma}(\overline{\Omega^0}; \R^3)$ and 
\begin{itemize}
    \item $\bar{\bm{w}}_0 \in  C_\mathrm{per}^{k, \gamma}(\overline{\Omega^0}; \R^3)$, $\bar{f} \in C_\mathrm{per}^{k, \gamma}(\overline{\Omega^0})$ if $k\ge 0$, and
    \item $\bar{\bm{w}}_0 \in  C_\mathrm{per}^{0}(\overline{\Omega^0}; \R^3)$, $\bar{f} \in C_\mathrm{per}^{0}(\overline{\Omega^0})$ if $k=-1$.
\end{itemize}
Then there is a constant $C$, only depending on $\|A_0\|$ and $\|A_0^{-1}\|$,  such that any solution $\bar{\bm{u}} \in C_\mathrm{per}^{k+1, \gamma}(\overline{\Omega^0}; \R^3)$ of
\begin{subequations}
    \label{eq:frozen}
    \begin{alignat}{2}
            \nabla \times (A_0 \bar{\bm{u}}) &=\bar{\bm{w}}_0+\textstyle \sum_{j=1}^3 \partial_j \bar{\bm{w}}_j \quad &&\text{in } \Omega^0, \label{eq:curl-frozen}\\
            \nabla \cdot \bar{\bm{u}} &= \bar{f} &&\text{in } \Omega^0, \label{eq:div-frozen}\\
            \bar{\bm{u}} \cdot \bm{n} &= 0 &&\text{on } \partial \Omega^0, \label{eq:kinboundary-frozen}
    \end{alignat}
\end{subequations}
satisfies
\begin{equation}
\label{eq:frozen estimate 1}
\|\bar{\bm{u}}\|_{C_\mathrm{per}^{0, \gamma}}\le C(\|\bar{\bm{u}}\|_{C_\mathrm{per}^0}+\|\bar{\bm{w}}_0\|_{C_\mathrm{per}^{0}}+\textstyle \sum_{j=1}^3\|\bar{\bm{w}}_j\|_{C_\mathrm{per}^{0,\gamma}}+\|\bar{f}\|_{C_\mathrm{per}^{0}}) \quad \text{if } k=-1,
\end{equation}
and
\begin{equation}
\label{eq:frozen estimate 2}
\|\bar{\bm{u}}\|_{C_\mathrm{per}^{k+1, \gamma}}\le C(\|\bar{\bm{u}}\|_{C_\mathrm{per}^0}+\|\bar{\bm{w}}_0\|_{C_\mathrm{per}^{k,\gamma}}+\textstyle \sum_{j=1}^3\|\bar{\bm{w}}_j\|_{C_\mathrm{per}^{k+1,\gamma}}+\|\bar{f}\|_{C_\mathrm{per}^{k,\gamma}}) \quad \text{if } k\ge 0.
\end{equation}
\end{lemma}

\begin{proof}
Throughout the proof, we use the notation $a \lesssim b$ to indicate that $a\le Cb$ for some constant $C$ which only depends on  $\|A_0\|$ and $\|A_0^{-1}\|$.
We first perform a change of variables, which transforms the equations to a standard div-curl problem with $A_0=I$. While not entirely necessary, this simplifies the analysis.  By taking a square root of $A_0$ and composing with a rotation, we can always decompose $A_0=\frac{1}{\det J} J^T J$ where $J$ is a matrix with $\det J=(\det A_0)^{-1}>0$, which leaves the plane $\{z=0\}$ invariant and maps the domain $\Omega^0$ to a  new periodic slab $\tilde \Omega^0$ for some  $\tilde d=d J\bm{e}_3\cdot \bm{e}_3>0$ and lattice $\tilde \Lambda$ spanned by $J\bm{\lambda}_1$ and $J\bm{\lambda}_2$. Then 
\[
\bm{u}\coloneqq \frac{1}{\det J} J \bar{\bm{u}} \circ J^{-1}
\]
satisfies
\begin{subequations}
    \begin{alignat}{2}
            \nabla \times \bm{u} &=\bm{w}_0+\textstyle \sum_{j=1}^3 \partial_j \bm{w}_j \quad &&\text{in } \tilde \Omega^0, \\
            \nabla \cdot \bm{u} &= f &&\text{in } \tilde \Omega^0, \\
            \bm{u} \cdot \bm{n} &= 0 &&\text{on } \partial \tilde \Omega^0, 
    \end{alignat}
 \end{subequations}
where $f=\frac{1}{\det J} \bar{f}\circ J^{-1}$, $\bm{w}_0=\frac{1}{\det J} J \bar{\bm{w}}_0\circ J^{-1}$ and $\bm{w}_j$ are the columns of the matrix
\[
W=\frac{1}{\det J} J (\bar W\circ J^{-1}) J^T,
\]
where $\bar W$ has columns $\bar{\bm{w}}_1$, $\bar{\bm{w}}_2$ and $\bar{\bm{w}}_3$.

To remove the $f$ in the divergence equation we decompose $\bm{u}$ as $\bm{u}=\nabla \phi+\tilde{\bm{u}}$, where
\begin{subequations}
    \label{decomposition phi}
    \begin{alignat}{2}
            \Delta \phi&= f \quad &&\text{in } \tilde \Omega^0, \\
          \nabla \phi \cdot \bm{n} &= 0 &&\text{on } \partial \tilde \Omega^0.
    \end{alignat}
    \end{subequations}
Then $\tilde{\bm{u}}$ satisfies
\begin{subequations}
    \label{eq:u-tilde-system}
    \begin{alignat}{2}
            \nabla \times  \tilde{\bm{u}} &=\bm{w}_0+\textstyle \sum_{j=1}^3 \partial_j \bm{w}_j \quad &&\text{in } \tilde \Omega^0, \label{eq:curl-tilde-u}\\
            \nabla \cdot \tilde{\bm{u}} &= 0 &&\text{in } \tilde \Omega^0,\\
            \tilde{\bm{u}} \cdot \bm{n} &= 0 &&\text{on } \partial \tilde \Omega^0.
    \end{alignat}
\end{subequations}
Here we note that $f$ has total integral $0$ by \eqref{eq:div-frozen}--\eqref{eq:kinboundary-frozen} and the divergence theorem, so that there is a unique solution $\phi$ of \eqref{decomposition phi} with zero average. Moreover, we have 
$\phi \in C_\mathrm{per}^{k+2, \gamma}$ with $\|\phi\|_{C_\mathrm{per}^{k+2,\gamma}}\lesssim \|f\|_{C_\mathrm{per}^{k,\gamma}}$ if $k\ge 0$, and $\phi \in C_\mathrm{per}^{1,\gamma'}$ with
$\|\phi\|_{C_\mathrm{per}^{1,\gamma'}}\lesssim \|f\|_{C_\mathrm{per}^{0}}$ for any $\gamma'\in (0,1)$  if $k=-1$ by combining Chapter 6.7 with Theorem 8.33 in \cite{gilbarg2001}. In particular, we can choose $\gamma'=\gamma$.
By subtracting a horizontal constant vector field from $\tilde{\bm{u}}$, we can make sure that
\[
\int_{\tilde \Omega_{00}} \tilde{u}_j \,\de V=0, \quad j=1,2.
\]
This modification gives rise to the term $\|\bm{u}\|_{C_\mathrm{per}^0}$ in the right-hand side of the final estimate.
By \cite[Theorem 2.1]{lokharu2019} there is a vector potential $\tilde{\bm{v}}\in C_\mathrm{per}^{k+2, \gamma}(\overline{\tilde \Omega^0}; \R^3)$ with $\tilde{\bm{v}}\times \bm{n}=0$ on $\partial \tilde \Omega^0$ such that $\nabla \times \tilde{\bm{v}}=\tilde{\bm{u}}$ and $\nabla \cdot \tilde{\bm{v}}=0$ (the proof in \cite{lokharu2019} is written for more general domains, but only in the case $k=0$; an inspection of the proof shows that it also applies for any $k\ge -1$). Then, 
$\nabla \times (\nabla \times \tilde{\bm{v}}) =\bm{w}$, where $\bm{w}$ denotes the right-hand side in \eqref{eq:curl-tilde-u}, and using \eqref{eq:curlcurl} and the fact that $\tilde{\bm{v}}$ is divergence-free, we obtain
    \begin{alignat*}{2}
            -\Delta \tilde{\bm{v}}&=
           \bm{w}_0+\textstyle \sum_{j=1}^3\partial_j \bm{w}_j \quad &&\text{in } \tilde \Omega^0, \\
            \nabla \cdot \tilde{\bm{v}} &= 0 &&\text{in } \tilde \Omega^0,\\
            \tilde{\bm{v}} \times \bm{n} &= 0 &&\text{on } \partial \tilde \Omega^0.
    \end{alignat*}
As usual, we replace this with a system where the second equation is just imposed on the boundary.
We note that this system has a unique solution $\tilde{\bm{v}}$ in $C_\mathrm{per}^{k+2, \gamma}$ with $\int_{\tilde \Omega^0} \tilde{v}_3 \,\de V=0$, and that
\[
\|\tilde{\bm{v}}\|_{C_\mathrm{per}^{k+2,\gamma}} \lesssim \|{\bm{w}}_0\|_{C_\mathrm{per}^{k,\gamma}}+ \textstyle \sum_{j=1}^3 \|{\bm{w}}_j\|_{C_\mathrm{per}^{k+1,\gamma}}, \quad k \ge 0,
\]
and
\[
\|\tilde{\bm{v}}\|_{C_\mathrm{per}^{1,\gamma}} \lesssim \|{\bm{w}}_0\|_{C_\mathrm{per}^{0}}+ \textstyle \sum_{j=1}^3 \|{\bm{w}}_j\|_{C_\mathrm{per}^{0,\gamma}}, \quad k=-1.
\]
Indeed, if the right-hand side is zero, then we directly get that $\tilde{v}_1=\tilde{v}_2=0$ since they satisfy $\Delta \tilde{v}_j=0$ with Dirichlet boundary conditions. On the other hand, $\tilde{v}_3$ satisfies the same equation, but with a Neumann condition, and hence is constant. By the assumption of zero average, we get $\tilde{v}_3=0$. We refer to \cite[Theorem 6.6 and Theorem 8.33]{gilbarg2001} for the global Schauder estimate for $\tilde{v}_1$ and $\tilde{v}_2$. For $\tilde{v}_3$, we can write the boundary condition as  $\partial_3 \tilde{v}_3=-\partial_1 \tilde{v}_1-\partial_2 \tilde{v}_2$ and consider the right-hand side as given, and again appeal to \cite[Chapter 6.7 and Theorem 8.33]{gilbarg2001}.
Transforming back using $J$, we obtain \eqref{eq:frozen estimate 1} and \eqref{eq:frozen estimate 2}.
\end{proof}

\begin{proof}[Proof of Theorem \ref{thm-app:schauder}]
We first consider the case $k\ge 0$.
We can assume that there is $\bm{P}$ and $\bm{Q}$ such that
\[
\|\bar{\bm{u}}\|_{C_\mathrm{per}^{k+1,\gamma}} \le 2\frac{|D^{k+1} \bar{\bm{u}}(\bm{P})-D^{k+1}\bar{\bm{u}}(\bm{Q})|}{|\bm{P}-\bm{Q}|^\gamma}.
\]
If not,
\[
[D^{k+1} \bar{\bm{u}}]_\gamma \le \frac12 \|\bar{\bm{u}}\|_{C_\mathrm{per}^{k+1,\gamma}}
\]
which implies $[D^{k+1} \bar{\bm{u}}]_\gamma \le 2\|\bar{\bm{u}}\|_{C_\mathrm{per}^{k+1}}$, 
and a standard interpolation argument gives
\[
\|\bar{\bm{u}}\|_{C_\mathrm{per}^{k+1,\gamma}} \lesssim \|\bar{\bm{u}}\|_{C_\mathrm{per}^0}.
\]
We next let $\theta\in (0,1)$ (independent of $\bar{\bm{u}}$, $\bar{\bm{w}}_j$ and $g$) be determined later and consider the following two cases.

Case 1: $|\bm{P}-\bm{Q}|\ge \theta$. This implies that
\[
\|\bar{\bm{u}}\|_{C_\mathrm{per}^{k+1,\gamma}} \le 2\frac{|D^{k+1} \bar{\bm{u}}(\bm{P})-D^{k+1}\bar{\bm{u}}(\bm{Q})|}{|\bm{P}-\bm{Q}|^\gamma} \le 4\theta^{-\gamma} \|D^{k+1}\bar{\bm{u}}\|_{C_\mathrm{per}^0}
\]
and we can again conclude by interpolation once $\theta$ has been fixed.

Case 2:  $|\bm{P}-\bm{Q}|<\theta$. Without loss of generality, we assume that $\bm{P}'\in B_{00}$, where $\bm{P} = (\bm{P}',P_z)$, and that $\theta$ is so small that $\overline{B_{2\theta}(\bm{P}')}\subset B_{00}$. Consider a smooth cut-off function $\zeta\in C^\infty(\R^3)$ such that $\zeta(\bm{x})\equiv 1$ when $|\bm{x}-\bm{P}|\le \theta$ and $\zeta(\bm{x})\equiv 0$ when $|\bm{x}-\bm{P}|\ge 2\theta$.  Additionally, choose $\zeta$ such that $|D^\ell \zeta|\le C\theta^{-\ell}$. Finally, extend $\zeta$ periodically with respect to the lattice $\Lambda$ in $\bm{x}'$. Now consider $\bar{\bm{v}}\coloneqq \zeta \bar{\bm{u}}$ and note that it satisfies
\begin{subequations}
    \begin{alignat}{2}
            \nabla \times (A_0 \bar{\bm{v}}) &=\zeta \bar{\bm{w}} - A\bar{\bm{u}}\times \nabla \zeta +\nabla \times ((A_0-A)\bar{\bm{v}}) \quad &&\text{in } \Omega^0, \\
            \nabla \cdot \bar{\bm{v}} &= \bar{\bm{u}}\cdot \nabla \zeta &&\text{in } \Omega^0, \\
            \bar{\bm{v}} \cdot \bm{n} &= 0 &&\text{on } \partial \Omega^0, 
    \end{alignat}
\end{subequations}
where $A_0=A(\bm{P})$.
Applying  Lemma \ref{lemma:frozen}, estimate \eqref{eq:frozen estimate 2}, we get
\begin{equation}
\label{eq:perturbation estimate}
\begin{aligned}
\|\bar{\bm{v}}\|_{C_\mathrm{per}^{k+1,\gamma}}&\le C'(\|\bar{\bm{v}}\|_{C_\mathrm{per}^0}+\|\zeta \bar{\bm{w}}\|_{C_\mathrm{per}^{k,\gamma}}+\|A\bar{\bm{u}}\times \nabla \zeta\|_{C_\mathrm{per}^{k, \gamma}} +\|(A_0-A) \bar{\bm{v}}\|_{C_\mathrm{per}^{k+1,\gamma}}+\|\bar{\bm{u}}\cdot \nabla \zeta\|_{C_\mathrm{per}^{k, \gamma}})\\
&\le C''(\theta [ D^{k+1} \bar{\bm{u}}]_{\gamma}+
\theta^{-(k+1+\gamma)}\|\bar{\bm{u}}\|_{C_\mathrm{per}^{k+1}}  +  \theta^{-(k+\gamma)}\|\bar{\bm{w}}\|_{C_\mathrm{per}^{k,\gamma}}),
\end{aligned}
\end{equation}
for some constants $C', C'' \ge 0$.
Here we have used that $A$ is in $C_\text{per}^{k+1, \gamma}$.
Choosing $\theta<1/(4C'')$ yields the estimate
\[
\|\bar{\bm{v}}\|_{C_\mathrm{per}^{k+1,\gamma}} \le \frac14 [D^{k+1} \bar{\bm{u}}]_\gamma+ 
C'''(\|\bar{\bm{u}}\|_{C_\mathrm{per}^{k+1}} +\|\bar{\bm{w}}\|_{C_\mathrm{per}^{k,\gamma}} )
\]
for some constant $C'''\ge 0$.
In particular, we get that
\begin{align*}
\frac12 \|\bar{\bm{u}}\|_{C_\mathrm{per}^{k+1,\gamma}} &\le \frac{|D^{k+1}\bar{\bm{u}}(\bm{P})-D^{k+1}\bar{\bm{u}}(\bm{Q})|}{|\bm{P}-\bm{Q}|^\gamma}=
\frac{|D^{k+1} \bar{\bm{v}}(\bm{P})-D^{k+1}\bar{\bm{v}}(\bm{Q})|}{|\bm{P}-\bm{Q}|^\gamma}\\
&
 \le \frac14 [D^{k+1}\bar{\bm{u}}]_\gamma+ 
C'''(\|\bar{\bm{u}}\|_{C_\mathrm{per}^{k+1}} +\|\bar{\bm{w}}\|_{C_\mathrm{per}^{k,\gamma}} ).
\end{align*}
and hence
\[
\|\bar{\bm{u}}\|_{C_\mathrm{per}^{k+1,\gamma}}\le 4C'''(\|\bar{\bm{u}}\|_{C_\mathrm{per}^{k+1}} +\|\bar{\bm{w}}\|_{C_\mathrm{per}^{k,\gamma}} ).
\]
Using interpolation, we then get
\[
\|\bar{\bm{u}}\|_{C_\mathrm{per}^{k+1,\gamma}}\le C(\|\bar{\bm{u}}\|_{C_\mathrm{per}^{0}} +\|\bar{\bm{w}}\|_{C_\mathrm{per}^{k,\gamma}}),
\]
for some $C=C(\delta^{-1},R)$. This proves \eqref{new est-u_1_gamma}.

To prove \eqref{new est-u_0_gamma} we simply replace \eqref{eq:frozen estimate 2} by  \eqref{eq:frozen estimate 1} in the estimate \eqref{eq:perturbation estimate} and get
\begin{align*}
\|\bar{\bm{v}}\|_{C_\mathrm{per}^{0,\gamma}}&\le C'(\|\bar{\bm{v}}\|_{C_\mathrm{per}^{0}}+ \|\zeta \bar{\bm{w}}\|_{C_\mathrm{per}^{0}}+\|A\bar{\bm{u}}\times \nabla \zeta\|_{C_\mathrm{per}^{0}} +\|(A_0-A) \bar{\bm{v}}\|_{C_\mathrm{per}^{0,\gamma}}+\|\bar{\bm{u}}\cdot \nabla \zeta\|_{C_\mathrm{per}^{0}})\\
&\le C''(\theta^\gamma [\bar{\bm{u}}]_\gamma+\theta^{-1} \|\bar{\bm{u}}\|_{C_\mathrm{per}^{0}} +\|\bar{\bm{w}}\|_{C_\mathrm{per}^{0}}),
\end{align*}
where we have used $A \in C_\text{per}^{0, \gamma}$. The rest of the argument proceeds in the same way as before.
\end{proof}

\section{Analyticity}\label{app:analyticity}
In this section, we will prove that the solutions to \eqref{eq:formulation-without-integral-constraint} are real analytic provided that they are regular enough. 

\begin{theorem} \label{thm_analyticity}
Let $\bm q$  be a point on the surface $\Gamma \coloneqq \{z=\eta(\bm{x}^\prime)\}$. Assume  $(\bm{u},\eta)$ solves \eqref{eq:formulation-without-integral-constraint}, and that $\bm{u}$ and $\eta$ are of class $C^{2, \gamma}$, with $\eta>-d$ locally near $\bm q$. Then $\bm{u},\eta \in C^{\omega}$ in a possibly smaller neighbourhood of $\bm q$.
\end{theorem}

\begin{proof}
We begin by deriving an elliptic system from \eqref{eq:formulation-without-integral-constraint}, to which we apply the regularity theory by Koch, Leoni and Morini \cite{koch2005}. Taking the curl of the first equation in \eqref{eq:formulation-without-integral-constraint} and using \eqref{eq:curlcurl} together with the divergence-free condition, we get $\Delta \bm{u}=-\alpha^2 \bm{u}$. Thus, $\bm u$ satisfies
\begin{subequations}
    \begin{alignat}{2}
            \Delta  \bm{u} +\alpha^2 \bm{u} &=0 &&\text{in } \Omega^\eta, \label{eq:elliptic system}\\
            \bm{u} \cdot \bm{n} &= 0 &&\text{on } \partial\Omega^\eta, \label{eq:tangential BC} \\
            \nabla \cdot \bm{u} &= 0 &&\text{on } \partial\Omega^\eta, \label{eq:div BC}\\
             \nabla \times \bm{u}\cdot \bm n &= 0 &&\text{on } \partial\Omega^\eta, \label{eq:curl BC}\\
            \dfrac{1}{2} \snorm{\bm{u}}^2 + g\eta - 2\sigma K_M &= Q \qquad &&\text{on } \Gamma, \label{eq:transmission BC}
    \end{alignat}
\end{subequations}
where \eqref{eq:div BC}--\eqref{eq:curl BC} follow by combining the interior equations and the boundary conditions and evaluating on the boundary. We claim that this satisfies the hypotheses of \cite[Theorem 3.1]{koch2005}, namely that \eqref{eq:elliptic system}--\eqref{eq:curl BC} constitute an elliptic system with complementing boundary conditions, and that \eqref{eq:transmission BC} is a transmission condition in the sense of \cite{koch2005}. Here $N=n=3$, and we assign the weights $s_k=0$, $t_j=2$ and $r_1=-2$, $r_2=r_3=-1$, so that $\mu=3$ and $r_0=1$, $t_0=2$. 
It is clear that the system \eqref{eq:elliptic system} is elliptic since it is diagonal with the Laplacian as principal part. To verify that the boundary conditions \eqref{eq:tangential BC}--\eqref{eq:curl BC} are complementing at $\bm q \in \Gamma$, we first note that by rotation and translation invariance, we can assume that $\bm q=\bm 0$ and $\nabla \eta(\bm 0)=0$. Then, the principal part of system frozen at $\bm q$ takes the form
\begin{subequations}
    \label{eq:complementing sys}
    \begin{alignat}{2}
    \Delta \bm{u}&= 0  && \text{in  }  z>0, \\
    u_3 &= 0 && \text{on  }  z=0,\\    \partial_x u_1+\partial_y u_2+\partial_z u_3 &= 0 && \text{on  } z=0,\\
    \partial_x u_2-\partial_y u_1 &= 0 \qquad && \text{on  }  z=0.
    \end{alignat}
\end{subequations}
We show that \eqref{eq:complementing sys} has no nontrivial  solution of the form $\bm u= \bar{\bm{u}}(z)e^{i\bm{\xi}'\cdot \bm{x}'}$, with $\bm{\xi}' \in \mathbb{R}^2 \setminus \{\bm{0}\}$, which is bounded as $z \to \infty$. Substituting the above form into \eqref{eq:complementing sys}, we obtain
\begin{subequations}
    \label{eq:complementing_sys_ode}
    \begin{alignat}{2}
    -\snorm{\bm{\xi}'}^2 \bar{\bm{u}}+\partial^2_{\tilde{z}}\bar{\bm{u}} &= 0  &&\text{in  } z \geq0, \label{eq:compl u} \\
    \bar{u}_3(\bm{0}) &= 0, \label{eq:bdry_compl_u_3}\\
    i\xi_1 \bar{u}_1(\bm{0})+i\xi_2 \bar{u}_2(\bm{0})+\partial_{z} \bar{u}_3(\bm{0}) &= 0, \\
    i\xi_1 \bar{u}_2(\bm{0})-i\xi_2 \bar{u}_1(\bm{0})&= 0
    \end{alignat}
\end{subequations}
From \eqref{eq:compl u} and the fact that only solutions which are bounded as $z \to \infty$ are considered, we have $\bar{u}_j=A_je^{-\snorm{\bm{\xi}'}z}$. Equation \eqref{eq:bdry_compl_u_3} implies that $A_3=0$ and hence $\bar{u}_3=0$. From the last two equations we get
\begin{align*}
    i\xi_1 A_1+i\xi_2 A_2 &= 0,\\
    i\xi_2 A_1-i\xi_1 A_2 &= 0.
\end{align*}
which yields $A_1=A_2=0$ since the determinant is $\snorm{\bm{\xi}'}^2 \neq 0$ for $\bm{\xi}' \in \mathbb{R}^2 \setminus \{\bm{0}\}$. Therefore, the system \eqref{eq:complementing_sys_ode} has only the trivial solution $\bar{\bm{u}}=0$ and hence the complementing condition holds. Finally, the mean curvature in the transmission condition \eqref{eq:transmission BC} can be expressed as $\operatorname{tr} (D_\tau \bm{n})$ (that is, the trace of the tangential gradient of the normal vector) and hence the derivative of the condition \eqref{eq:transmission BC} with respect to $D_\tau {\bm n}$ is a non-zero multiple of the identity matrix, and therefore a sign-definite matrix.
\end{proof}

In the statement of Theorem \ref{thm_analyticity}, we required  the solution $(\bm{u},\eta)$ to be of class  $C^{2, \gamma}$ locally. In fact, this requirement is automatically satisfied for solutions with less regularity. 
\begin{corollary}\label{cor-app:analyticity}
If $(\bm{u},\eta)\in C^{1, \gamma}_\mathrm{per, div, T}(\overline{\Omega^\eta}; \R^3) \times \curlB^{2, \gamma}$ solves \eqref{eq:formulation-without-integral-constraint}, 
then $\eta \in C^\omega(\R^2)$ and $\bm{u}\in C^\omega(\overline{\Omega^\eta}; \R^3)$.
\end{corollary}
\begin{proof}
Indeed, $\bar{\bm{u}} \in C^{1, \gamma}$ by Lemma \ref{lem:bound-flattening-trafo}, and by Theorem \ref{thm:dynBoundaryCond} (with $Q(\bm c)$ replaced by a general constant $Q$), 
\[\eta=\mathcal{K}\left(\nabla \eta, Q - \frac{1}{2} B(\bar{\bm{u}},\eta)\right) \in C^{3, \gamma}\]
since $\nabla\eta$ and $B(\bar{\bm{u}},\eta)$ are of class $C^{1, \gamma}$ (recall that $B(\bar{\bm{u}},\eta)$ contains at most one derivative in $\eta$). Consequently,
\[\bar{\bm{u}} = \bar{\curlL}(\alpha \bar{\bm{u}}, \eta, \bm{M}) \in C^{2, \gamma}\]
in view of Theorem \ref{thm:solution-flattened-div-curl}. Applying Lemma \ref{lem:bound-flattening-trafo} in the opposite direction, we obtain $\bm u \in C^{2, \gamma}$. It now follows by 
\cite[Theorem 6.8.1]{morrey1966} that $\bm u$ is analytic in $\Omega^\eta$.  By Theorem \ref{thm_analyticity}, $\eta$ is also analytic and $\bm u$ is analytic in a neighbourhood of the surface $\{z=\eta\}$. Finally, by  
 applying \cite[Theorem 2.2]{koch2005} with $r_0=0$ (see also \cite[Theorem 6.8.2]{morrey1966}), we obtain that $\bm u$ is  analytic in a neighbourhood of the bottom $\{z=-d\}$.
\end{proof}
\begin{remark}
It is, in fact, also possible to prove a purely local version where the regularity assumption in Theorem \ref{thm_analyticity} is lowered to $\bm u \in C^{1, \gamma}$ locally near $\bm q$, but since it is a bit technical and we do not need it, we refrain from writing down a proof. See, for example, \cite[Remark 4.7]{koch2005} for similar results in the scalar case.
\end{remark}

\begin{remark}
Note that Theorem \ref{thm_analyticity} was proved by Craig and Matei \cite{craig2007regularity} in the case $\alpha=0$, independently from \cite{koch2005} but using a similar approach. See also \cite{ChenLiWang13, Henry12, WeissZhang12} for similar results for two-dimensional gravity-capillary waves with vorticity.
\end{remark}

\printbibliography

\bigskip
\bigskip

\hypertarget{affBH}{}

\noindent \textbf{Bastian Hilder} \\
\noindent Department of Mathematics, Technische Universität München, Boltzmannstraße 3, 85748 Garching b. München, Germany \\
\textit{Email address}: \texttt{bastian.hilder@tum.de}

\bigskip
\hypertarget{affGT}{}

\noindent \textbf{Giang To} \\
\noindent Centre for Mathematical Sciences, Lund University, PO Box 118, 22100 Lund, Sweden \\
\textit{Email address}: \texttt{giang.to@math.lu.se}

\bigskip
\hypertarget{affEW}{}

\noindent \textbf{Erik Wahlén} \\
\noindent Centre for Mathematical Sciences, Lund University, PO Box 118, 22100 Lund, Sweden \\
\textit{Email address}: \texttt{erik.wahlen@math.lu.se}

\end{document}